\documentclass{article}
\usepackage[T1]{fontenc}
\usepackage[utf8]{inputenc}
\usepackage{geometry}
\usepackage[dvipsnames]{xcolor}
\definecolor{myred}{HTML}{c20014}
\definecolor{mygreen}{HTML}{008000}

\usepackage{mathtools}
\usepackage{amsmath}
\usepackage{amsfonts}
\usepackage{caption}
\usepackage{array,multirow,makecell}
\usepackage{amsthm}
\usepackage{amssymb}

\usepackage{mathrsfs}
\usepackage{verbatim}
\usepackage{dsfont}
\usepackage[inline]{enumitem}
\setlist{noitemsep}
\usepackage{csquotes}
\usepackage{cite}
\usepackage{tikz-cd}

\usepackage{accents}

\usepackage[multiple]{footmisc}
\usepackage{mleftright}
\usepackage{comment}

\usepackage{url}
\usepackage[colorlinks=true,linkcolor=myred,citecolor=mygreen]{hyperref}
\hypersetup{linktocpage}
\usepackage[capitalise]{cleveref}
\usepackage[font=small,labelfont=bf]{caption}

\crefname{equation}{}{}
\Crefname{equation}{Equation}{Equations}

\numberwithin{equation}{section}

\crefrangeformat{enumi}{#3#1#4 through~#5#2#6}

\newtheorem{theo}{Theorem}[section]
\newtheorem{prop}[theo]{Proposition}
\newtheorem{coro}[theo]{Corollary}
\newtheorem{lemma}[theo]{Lemma}

\theoremstyle{definition}
\newtheorem{defi}[theo]{Definition}

\theoremstyle{remark}
\newtheorem{rem}[theo]{Remark}

\newtheorem{as}[theo]{Assumption}

\crefname{theo}{Theorem}{Theorems}
\crefname{claim}{Claim}{Claims}
\crefname{defi}{Definition}{Definitions}
\crefname{enumi}{}{}
\Crefname{enumi}{Item}{Items}
\crefname{rem}{Remark}{Remarks}
\crefname{prop}{Proposition}{Propositions}

\newcommand{\dt}[1]{\frac{\mathrm{d}#1}{\mathrm{d}t}}
\newcommand{\dl}[2]{\frac{\partial#1}{\partial#2}}
\newcommand{\ddl}[2]{\frac{\partial^2#1}{\partial#2^2}}
\newcommand{\dtt}[1]{\frac{\mathrm{d}^2#1}{\mathrm{d} t^2}}
\newcommand{\D}{\mathcal{D}}
\newcommand{\A}{\mathcal{A}}

\newcommand{\B}{\mathcal{B}}
\renewcommand{\d}{\mathrm{d}}
\renewcommand{\H}{\mathcal{H}}
\renewcommand{\L}{\mathcal{L}}

\renewcommand{\S}{\mathcal{S}}

\newcommand{\R}{\mathbb{R}}

\newcommand{\N}{\mathbb{N}}
\newcommand{\C}{\mathcal{C}}
\newcommand{\Z}{\mathbb{Z}}

\renewcommand{\i}{\mathrm{i}}

\DeclareMathOperator{\re}{Re}
\DeclareMathOperator{\im}{Im}
\DeclareMathOperator{\dive}{div}

\newcommand{\Co}{\mathbb{C}}
\newcommand{\F}{\mathcal{F}}

\newcommand{\dom}{\operatorname{dom}}

\DeclareMathOperator{\supp}{supp}

\newcommand{\Chi}{\mathrm{X}}

\newcommand{\Oph}{\operatorname{Op}_h}
\newcommand{\Opht}{\operatorname{Op}_h^\tau
}

\newcommand{\diff}{\operatorname{Diff}}

\newcommand{\cha}{\operatorname{char}}

\newcommand{\uD}{\mathrm{1D}}

\providecommand{\keywords}[1]
{
  \small	
  \textbf{{Keywords.}} #1
}
\providecommand{\AMS}[1]
{
  \small	
  \textbf{{AMS subject classifications.}} #1
}

\renewcommand{\leq}{\leqslant}
\renewcommand{\geq}{\geqslant}
\renewcommand{\epsilon}{\varepsilon}
\newcommand{\eps}{\varepsilon}

\title{Optimal stabilization rate for the wave equation with hyperbolic boundary condition\thanks{This work was supported by the Research Council of Finland Grant number 349002 and Agence Nationale de la Recherche, QuBiCCS project ANR-24-CE40-3008. The authors would like to thank Lassi Paunonen for arranging a visit of the first author at Tampere University in June 2025.}}

\author{Hugo Parada\thanks{Universit\'e de Lorraine, CNRS, Inria, IECL, F-54000 Nancy, France. Email: \nolinkurl{hugo.parada@inria.fr}.} \and Nicolas Vanspranghe\thanks{Université Paris-Saclay, CNRS, CentraleSupélec, Inria, Laboratoire des signaux et systèmes, 91190 Gif-sur-Yvette, France. Email: \nolinkurl{nicolas.vanspranghe@centralesupelec.fr}.}\footnote{Corresponding author.}}

\begin{document}

\maketitle

\abstract{
% We consider linear waves on a bounded domain in the following setting. One part of the boundary is governed by a coupled lower-dimensional wave equation (i.e., dynamic Ventcel/Wentzell boundary condition) and is subject to viscous damping. The other (possibly empty) part is kept at rest. When the dynamic boundary geometrically controls the domain, we show that the total energy of classical solutions decays at the rate $1/t$. The proof relies on an analysis of high-frequency quasimodes, suitable boundary estimates obtained in different microlocal regimes, and a special decoupling argument. Optimality is assessed via an appropriate quasimode construction... \edn{à changer}
We show that the energy of classical solutions to the wave equation with  hyperbolic boundary condition (i.e., dynamic Wentzell boundary condition) and damping on the boundary decays like $1/t$. In fact we allow mixed boundary conditions: a possibly empty, disjoint part of the boundary may be kept at rest provided that the dynamic part satisfies the geometric control condition. We also prove that this decay rate is sharp.
Our results follow from resolvent estimates, which we establish by studying high-frequency quasimodes.
% Our approach relies on resolvent estimates 
}


\AMS{35L05, 35L20, 35B35, 35B40, 93C20.}

\tableofcontents

\section{Introduction}

\subsection{Main result}

\label{sec:intro}
Let $\Omega$ be a smooth bounded domain in $\R^d$, $d \geq 2$. Consider a nonempty relatively open subset $\Gamma_0$ of $\Gamma$ such that, letting $\Gamma_1 \triangleq \Gamma \setminus \Gamma_0$,
$
\overline{\Gamma_0} \cap \overline{\Gamma_1} = \emptyset
$. In other words, either  $\Gamma_1$ is empty, or $\Gamma$ is made of two disjoint components $\Gamma_0$ and $\Gamma_1$.
% , with boundary $\Gamma = \Gamma_0 \cup \Gamma_1$. 
In this paper we are interested in the initial boundary value problem
\begin{subequations} 
\label{eq:IBVP}
\begin{align}
\label{eq:wave-equation}
&(\partial_t^2 - \Delta)u = 0&&\mbox{in}~\Omega \times (0, +\infty), \\
\label{eq:HBC}
&(\partial_t^2 + \partial_t - \Delta_\Gamma)u = - \partial_\nu u &&\mbox{on}~\Gamma_0 \times (0, +\infty), \\
\label{eq:dirichlet}
&u = 0 &&\mbox{on}~\Gamma_1 \times (0, +\infty), \\
&(u, \partial_t u)|_{t = 0} = (u_0, u_1) &&\mbox{in}~{\Omega},
\end{align}
\end{subequations}
where $\partial_\nu$ is the outward normal derivative, $\Delta$ is the usual Laplacian, and $\Delta_\Gamma$ is the Laplace--Beltrami operator on $\Gamma$ equipped with the canonical Riemannian metric induced by $\R^d$; see \cref{sec:laplace-beltrami} for a definition. \Cref{eq:HBC} represents a \emph{hyperbolic} boundary condition, or \emph{dynamic} {Wentzell}\footnote{Also spelled ``Ventcel'', ``Venttsel'{''}, etc.} boundary condition, subject to viscous damping or dissipative feedback action. The \emph{energy} of solutions $u$ to \cref{eq:IBVP} is given by
\begin{equation}
\label{eq:energy}
E(u, t) \triangleq \frac{1}{2}\int_\Omega |\nabla u|^2 + |\partial_t u|^2 \, \d x + \frac{1}{2} \int_{\Gamma_0} |\nabla_\Gamma u|^2 + |\partial_t u|^2 \, \d \sigma, \quad t \geq 0,
\end{equation}
where $\d x$ denotes the standard Lebesgue measure on $\R^d$, $\d \sigma$ denotes the induced surface measure on $\Gamma$, and $\nabla_\Gamma$ is the Riemannian gradient on $\Gamma$; see \cref{sec:laplace-beltrami} for further details. Solutions $u$ to \cref{eq:IBVP} satisfy, at least formally, the \emph{energy identity}
\begin{equation}
\label{eq:energy-id}
E(u,t) \big|_0^T = - \iint_{\Gamma_0 \times (0, T)} |\partial_t u|^2 \, \d \sigma \, \d t, \quad T \geq 0,
\end{equation}
and in particular the energy is nonincreasing. In fact, a standard argument based on unique continuation will show that $E(u,t) \to 0$ as $t \to + \infty$ for all finite-energy solutions $u$ to \cref{eq:IBVP}; see \cref{sec:wp-as}. We are interested in \emph{quantified} rates of decay, and our main result reads as follows.
%% Peut être
\begin{as}[Dynamic hypotheses] 
\label{as:dyn}
No \emph{generalized bicharacteristic} (see \cref{sec:gcc}) has contact of infinite order with $\Gamma \times (0, +\infty)$, as stated in \Cref{ass:finite-contact}, and $\Gamma_0$ satisfies the \emph{geometric control condition} \cite{BarLeb92}, as in \cref{def:gcc}.
\end{as}
\begin{theo}[Energy decay rate]
\label{th:intro} Suppose that \cref{as:dyn} holds.
% no \emph{generalized geodesics} has contact of infinite order with $\Gamma \times (0, +\infty)$
Let $(u_0, u_1)\in H^{2}(\Omega)\times H^{1}(\Omega)$ be such that $(u_{0},u_{1})|_{\Gamma_0} \in H^{2}(\Gamma_0)\times H^{1}(\Gamma_0)$ and $(u_0, u_1)|_{\Gamma_1} = 0$, and let $u$ the unique solution of \cref{eq:IBVP} with initial data $(u_0, u_1)$. Then, 
\begin{equation}
E(u,t)=o(t^{-1}), \quad t\to +\infty.
\end{equation}
% Furthermore, for each $\eps > 0$ there exists a finite-energy solution $u$ to
\end{theo}
(The geometric control condition is automatically satisfied when $\Gamma_1$ is empty; on the other hand the finite-order contact condition is always satisfied when $\Gamma$ is locally real-analytic.) The decay rate of \cref{th:intro}, which is restated in more detail as \cref{th:poly-stab} below, cannot be improved, at least in the polynomial scale. 
% Indeed, we  have a lower bound on the decay of some solutions.
\begin{theo}[Lower bound on decay,  disk case]\label{th:intro-sharp} 
Suppose that $\Omega$ is the open unit disk $\mathbb{D} \triangleq \{x \in \R^2 : |x| < 1\}$, in which case $\Gamma = \Gamma_0 = \partial \mathbb{D}$. Then,
for all $\eps >0$ there exist  initial data $(u_0, u_1)\in H^{2}(\Omega)\times H^{1}(\Omega)$ with $(u_{0},u_{1})|_{\Gamma} \in H^{2}(\Gamma)\times H^{1}(\Gamma)$ 
% and $(u_0, u_1)|_{\Gamma_1} = 0$ 
such that the corresponding solution $u$ to \cref{eq:IBVP} satisfies
\begin{equation}
\label{eq:op-sup}
\sup_{t \geq 0} t^{1+\eps} E(u, t) = + \infty.
\end{equation}
\end{theo}
In particular \cref{th:intro-sharp} shows that the stability of \cref{eq:IBVP} is \emph{non-uniform} in general,\footnote{We
% believe that the stability is in fact never uniform and 
conjecture that the conclusion of \cref{th:intro-sharp} holds regardless of the geometry.} since any rate that is valid for \emph{all} finite-energy solutions  would lead to exponential decay of the energy. Note that, in the setting of \cref{th:intro-sharp}, \cref{as:dyn} is satisfied: here, the obstruction to uniform stabilization lies within the \emph{coupling} between interior and boundary dynamics rather than a failure of the geometric control condition.
% As a matter of fact, our proofs of \cref{th:intro,th:intro-sharp} will reveal that the loss of uniform stability is, so to speak, localized around the boundary $\Gamma_0$
% % \edn{that the lack of uniform stability is essentially a local phenomenon} 
% and has little to do with 
% % global
% propagation across the domain $\Omega$. 
This stands in contrast with many works dealing with non-uniform energy decay of waves, for instance \cite{Leb96proceedings,AnaLea14}, and constitutes the key feature which motivates our analysis.

\subsection{Background and literature review}

Generally speaking,  dynamic (or kinetic) boundary conditions such as \cref{eq:HBC} arise in problems where the boundary carries its own momentum and may model displacements of, e.g., point masses at the endpoints of one-dimensional media or thick membranes in the multidimensional setting.
% In the multidimensional setting, wave equations with dynamic boundary conditions 

Notably,  \emph{acoustic} boundary conditions for wave equations give rise to coupled systems that are quite close to \cref{eq:IBVP}. Such boundary conditions were introduced
% \cite{MorIng86book} 
% for the velocity potential of fluid or gases in wave motion were introduced
(in a time-harmonic form) by physicists Morse and Ingard in \cite{MorIng86book} for small, irrotational perturbations of acoustic flows. Later,  Beale and Rosencrans gave in \cite{BeaRos74,Bea76} a more precise formulation of  the associated initial boundary value problems: in a bounded medium delimited by an impenetrable (non-porous) membrane, the velocity potential $\phi$ solves the wave equation $(\partial_t^2 - c^2 \Delta)\phi = 0$ and is subject to the boundary condition $\partial_t \delta = \partial_\nu \phi$, where the normal displacement $\delta$ satisfies $(m \partial_t^2 + d\partial_t + k)\delta = \rho \partial_t \phi$ if  tangential oscillations are neglected (``{locally reacting}'' surface in the terminology of \cite{MorIng86book}) or $(m \partial_t^2 + d\partial_t + k \Delta_\Gamma) \delta = \rho \partial_t \phi$ if transverse tension is taken into account. Here $c$ is the speed of sound in the medium,  $m$, $d$ and $k$ are surface mass, resistivity and stiffness, and $\rho$ is the density of the gas at rest. 
Spectral properties of the corresponding semigroup generators were investigated in \cite{Bea76}, whereas local energy decay results along with a scattering theory for a related {exterior} problem were given in \cite{Bea77}. Especially relevant to our study are the findings of Littman and Liu in \cite{Liu98}. For the locally reacting acoustic model posed in an infinite strip, they proved that the energy of solutions \emph{cannot} decay uniformly and established polynomial decay rates for smoother solutions. Their approach relies on a careful examination of the spectrum of the damped generator combined with a suitable spectral decomposition of solutions. Still in the locally reacting case  (that is, without the Laplace--Beltrami operator in the boundary dynamics) but for bounded domains,
 polynomial stabilization rates were established in \cite{MunQin03,AbbNic15}, either under \emph{a priori} knowledge of a resolvent estimate for the static Neumann feedback problem,  or under differential multiplier (``Lions'') geometric conditions; see also \cite{Tou13thesis,NicPau25}. In contrast, when the damping acts in the interior \emph{and} on the boundary, uniform stabilization is possible: in the locally reacting case, see for instance \cite{GaoLia18} (damping region covering the boundary) or \cite{CavMor20} (geometric control condition); in the ``non-locally reacting'' case,
 % , where the damping region geometrically controls the domain
 see also \cite{VicFro16} (differential multiplier condition).

 % \cite{VicFro16} (differential multiplier conditions for the interior damping region, non-locally reactive case), \cite{GaoLia18} (damping region covering the boundary, locally reactive case), and \cite{CavMor20} (geometric control condition for the interior damping region, locally reactive case).

 In the one-dimensional setting, early investigations of dynamic boundary conditions mostly concern  Euler--Bernoulli beam models coupled with rigid body ordinary differential equations \cite{LitMar88b,LitMar88}, for instance  when the beam is clamped at one end and attached to a tip mass at the other end \cite{Rao98,ConMor98}; see also \cite{ZhaWei14,FkiPau25} for related later works. Similarly as for acoustic models and while mere \emph{asymptotic} stability is preserved, it was observed that the addition of such boundary dynamics to an otherwise uniformly exponentially stable system may result in the existence of arbitrary slow finite-energy solutions. 
 % Here lack of uniform stability can be proved by means of a \emph{compact perturbation} argument f to Russell.
 On the other hand, one-dimensional wave equations with dynamic boundary conditions have been studied in different contexts, including control of
 gantry crane systems \cite{ConMid98,AndCor20} and reduction of drilling torsional vibrations \cite{RomBre18,VanFer22}; see also \cite{ChiNgu24}, the references therein, and the recent work \cite{DahChi25arxiv} which has the particularity of considering stability in 
$L^p$-based state spaces.
% as well.
 % Here lack of uniform stability can be proved by means of a \emph{compact perturbation} argument due to Russell
% To our knowledge, that work of theirs as well as were one the first to report
% otherwise stable system
% a scattering theory and  are developed in \cite{}
% \edn{change date} and
% give rise to coupled systems that are closely related
% \footnote{Compared to \cref{eq:HBC}, acoustic boundary  conditions typically read as $\partial_\nu u = \partial_t \delta$ where $\delta$ is the variable subject to second-order dynamics.}
% . Compared to \cref{eq:HBC}, such boundary  conditions typically read as $\partial_\nu u = \partial_t \delta$, where  $\delta$ may represent the normal displacement of a membrane and is subject to second-order dynamics. Relevant models include equations of the form 
% % With physical parameters set to one for simplicity of exposition, we may consider
% $(\partial_t^2 - \partial_t +  1) \delta = - \partial_t u$
% for the evolution of velocity potentials in the study of acoustic waves
% when tangential
% \edn{Relevant to our study...}

% We now return to our model \cref{eq:IBVP} and give further insight regarding the hyperbolic boundary condition \cref{eq:HBC}.

% We now return to the multidimensional case.
Motivated by diffusion processes, Wentzell introduced in \cite{Ven59}  (non-dynamic) boundary conditions
% that involve the normal derivative and second-order tangential derivatives of solutions.
of the form $- \Delta_\Gamma u + \mbox{``lower-order'' terms} = - \partial_\nu u$. These also arise as approximations in transmission problems involving a thin layer around the domain, in the small thinness limit; see for instance \cite{BenLem96} and more specifically \cite[Appendix A]{Buffe17}.
Elliptic and parabolic problems with such boundary conditions were investigated in \cite{FavGol02,FabGol10,BonDam10} to cite only a few; see also the survey paper \cite{ApuNaz00}. As for hyperbolic problems, there has been abundant research on the wave equation with static (i.e., without second-order \emph{time} derivatives) or dynamic (such as \cref{eq:HBC}) Wentzell boundary conditions. It was shown in \cite{Hem00} by means of spectral analysis that the wave equation posed in a square with damped, {static} Wentzell boundary condition on two adjacent sides (and homogeneous Dirichlet boundary condition elsewhere) is \emph{not} uniformly stable on its natural energy space. Most interestingly, here the lack of uniform stability is not related to a violation of the geometric control condition, nor does it come from boundary dynamics. A similar configuration was considered in \cite{NicLao10}, with the difference that only one  side of the square is subject to the Wentzell boundary condition: there, it was proved that the energy of classical solutions decays like $1/t$,
although in that case it is not clear what conclusions to draw since geometric control fails. Localized interior damping, however, can achieve uniform stabilization of the model with static Wentzell boundary condition even in the absence of boundary damping: see \cite{CavCav09,CavLas12}, with the caveat that both references are restricted to geometries where  multiplier techniques are effective. Concerning waves with dynamic Wentzell boundary as in \cref{eq:IBVP}, the aforementioned negative results constitute strong evidence that one cannot expect uniform stability via damping acting on the boundary only, which is confirmed by our \cref{th:intro-sharp}.
% but to our knowledge this fact, made precise by our \cref{th:intro-sharp}, has not been proved before.

In \cite{Buffe17}, Buffe considered the wave equation with both static and dynamic Wentzell boundary condition and \emph{localized} damping on the boundary; that is, the damping is of the form $b(x) \partial_t u$, the control function $b$ being nonnegative and sufficiently smooth. When the damping region (i.e., the support of $b$) is nontrivial,
 he proved by means of Carleman estimates that the energy of classical solutions decays like $1/\log(2 +t)^2$.
 This result generalizes to Wentzell boundary conditions the \emph{a priori}, ``worst-case scenario'' logarithmic decay rates of \cite{Leb96proceedings,LebRob97}, which were established for interior damping and Neumann boundary damping; see also the logarithmic local energy decay rate of \cite{Bur98} for exterior problems. Our setup is slightly different since we consider mixed boundary conditions and our damping function is constant, still the result of \cite{Buffe17} is applicable when $\Gamma_1 = \emptyset$ and yields logarithmic energy decay for classical solutions to \cref{eq:IBVP}. 
 
 In the present paper we complement Buffe's analysis of the problem with hyperbolic boundary condition by exhibiting the best possible energy decay rate in general domains under
 % (\cref{th:intro-sharp}) and showing that this rate is achieved under 
 additional dynamical hypotheses, namely the geometric control condition.
 % (\cref{th:intro}).
To arrive at these results there are a number of difficulties that differentiate our work from the literature discussed above and warrant new technical developments, some of which we highlight here.
% Settling these questions comes with a number of difficulties and warranted technical development differentiating our work from the references discussed so far
% Settling these questions
\begin{itemize}
    \item 
    We tackle the stabilization problem in the presence of both time and tangential second-order derivatives on the boundary, which is arguably the most difficult case. To capture the coupling between interior and boundary dynamics, we need precise normal and tangential trace estimates, which we will establish by means of microlocal techniques (\cref{sec:microlocal}).
    \item In order to exploit the geometric control condition, we implement a special decoupling argument (\cref{sec:decoupling}) which makes it possible to use, as a black box, Bardos, Lebeau and Rauch's boundary observability estimate for waves with homogeneous Dirichlet data \cite{BarLeb92}. In particular we  avoid a delicate propagation analysis for waves with second-order boundary conditions.
    \item Our investigation of  traces on the dynamic boundary hints at the presence of a stability bottleneck in the microlocal regime where tangential derivatives are negligible. This motivates, in the special case of the disk, a spectral  investigation of \emph{focusing modes}, which exhibit high-frequency oscillations in the radial direction, in order to assess the sharpness of our decay rate (\cref{sec:focusing}).  
    % In this regime and  locally near the boundary, the system \cref{eq:IBVP} should resemble its one-dimensional counterpart, which is simpler. In our proof of optimality (\cref{sec:optimality}), we will therefore embed one-dimensional quasimodes into the multidimensional problem to obtain a lower bound on the resolvent.
    % for the latter, lower bounds on the decay rate can be established more easily via direct spectral computations
\end{itemize}
 % Our proofs relies on resolvent estimates, which we then turn into energy decay rates
 % to Wentzell boundary conditions
 % % \edn{thereby generalized
 % % arguably
% }

We should also point out that the energy decay rate $1/t$ for \cref{eq:IBVP} was claimed in the recent work \cite{LiLianXiao23} but only under \emph{ad-hoc} differential multiplier conditions that severely restricts the geometry of the domain, even when the whole boundary is damped; furthermore, there seem to be gaps in their proof, namely the use of incorrect trace inequalities for Sobolev spaces of half-integer order.

Let us conclude this section with some additional references that are farther away from the topics of  stability and stabilization, but are nonetheless relevant to our work. The recent work \cite{baudouin2022unified}, which we discuss in more depth in \cref{sec:num} below, deals with observability of waves posed in an two-dimensional annulus, where the inner boundary is subject to a hyperbolic boundary condition and the outer boundary is left at rest.
In \cite{Vit17}, an extensive study of well-posedness is carried out for a nonlinear model generalizing \cref{eq:IBVP}; there, nonlinear (but \emph{subcritical})  potential terms are considered,
in addition to nonlinear, monotone damping. Related questions were also explored in \cite{GraSai12}. In \cite{GraLas14}, the authors investigate smoothing properties of waves with a hyperbolic boundary condition and \emph{strong} (Kelvin--Voigt) interior or boundary damping. They show that the resulting semigroup is analytic when the strong damping acts in the interior and the boundary, and that it is  of Gevrey class (which is a weaker property) if only  the interior is damped. Finally, a physical derivation for waves with dynamic boundary conditions (without tangential terms) can be found in \cite{Gol06}.

% Notably, one of the results proved in \cite{Vit17} states that, even in the presence of linear damping in the interior and on the boundary, 

% \edn{
% severly restricts

% strong evidence that
% }

% \edn{Vit... extensive study of the nonlinear problem}

\section{Preliminaries and operator model}
\label{sec:pre}

\subsection{The Laplace--Beltrami operator on the boundary}
\label{sec:laplace-beltrami}

In this preliminary section, we recall for the reader's convenience some basic elements of differential geometry and
give the proper definition of the Laplace--Beltrami operator $\Delta_\Gamma$. We furthermore  set up, near $\Gamma_0$, \emph{normal geodesic coordinates} that will be useful in the microlocal analysis of \cref{sec:microlocal}. Our notation and presentation borrow  from \cite{RouLeb22a,RouLeb22b}, which we refer the reader to for additional details.

% set up local coordinates near $\Gamma_0$

% recall for the reader's convenience some facts

\subsubsection{Elements of differential geometry}

% \para{The Laplace--Beltrami operator} 
The boundary $\Gamma$ of $\Omega$ is a smooth $(d-1)$-dimensional submanifold of the Euclidean space $\R^{d}$. Given $x \in \Gamma$ and a \emph{local chart} $(O, \kappa)$ around $x$, the tangent space $T_x \Gamma$ is the set of linear maps $v : \C^\infty(\Gamma) \to \R$ of the form
\begin{equation}
\label{eq:basis-tx}
v(f) = \sum_{i=1}^{d-1} a^{i} \frac{\partial(f \circ \kappa^{-1})}{\partial y_i}, \quad f \in \C^\infty(\Gamma), \quad \quad (a^1, \dots, a^{d-1}) \in \R^{d-1}.
\end{equation}
Each $T_x\Gamma$ is a $(d-1)$-dimensional vector space that does not depend on the particular choice of local chart around $x$.
Let $n \in \N$. We shall identify each $T_x \R^{n}$ with $\R^n$.
Given a smooth function $F = (F_1, \dots, F_n)$ from $\Gamma$ to $\R^{n}$, the tangent map (or differential) of $F$ at $x \in \Gamma$ is the linear map $T_x F: T_x \Gamma \to \R^{n}$ defined by
\begin{equation}
T_x F(v) = (v(F_1), \dots, v(F_n)), \quad v \in  T_x \Gamma.
\end{equation}
Note that for $f \in \C^\infty(\Gamma)$, $T_xf(v) = v(f)$.
Since $\Gamma$ is a submanifold of $\R^{d}$,
the tangent map $T_x i$ of the inclusion $ i : \Gamma \to \R^{d}$ is injective for $x \in \Gamma$, and we may canonically identify $T_x \Gamma$ with its image under $T_x i$.
% , so that $T_x \Gamma$ is at time viewed as a $(d-1)$-dimensional subspace of $\R^{d}$.
% In particular, the usual scalar product in $\R^{d}$ induces a smooth Riemannian metric $g$ on $\Gamma$:
% Furthermore, there is a canonical embedding of $T_x \Gamma$ into $\R^{d}$:
% As a result, for each $x \in \Gamma$ we may canonically identify the tangent space $T_x\Gamma$ with its image under the differential at $x$ of the inclusion map $\id : \Gamma \to \R^d$, meaning that we may view $T_x\Gamma$ as a
% $(d-1)$-dimensional subspace of $\R^d$. 
Then, $\Gamma$ is equipped with the (smooth) Riemannian \emph{metric} induced by the scalar product of $\R^{d}$:
\begin{equation}
g_x(u, v) \triangleq T_x i(v) \cdot T_x i(w), \quad v, w \in T_x \Gamma, \quad x \in \Gamma. 
\end{equation}
Let $x \in \Gamma$ and $C = (O, \kappa)$ be a local chart around $x$. We write $y = \kappa(x)$. Since $g_x$ is a scalar product on $T_x \Gamma$ and using the basis given by \cref{eq:basis-tx}, there exists a unique symmetric positive definite matrix $(g_{ij}^C(y))_{1\leq i,j \leq d-1}$ such that
\begin{equation}
g_x(v, w) = \sum_{i,j=1}^{d-1} g_{ij}^C(y) v^i w^j, \quad v, w \in T_x \Gamma,
\end{equation}
where $w^i$ and $v^i$, $i = 1, \dots, d-1$, are the coefficients of $w$ and $v$ in \cref{eq:basis-tx}, respectively. The components $y \in \kappa(O) \mapsto g_{ij}^C(y) \in \R$ are all smooth. The {Riemannian} \emph{gradient} $\nabla_g f(x)$ at $x \in \Gamma$ of a function $f \in \C^\infty(\Gamma)$ is the unique element of $T_x \Gamma$ satisfying
\begin{equation}
g_x(\nabla_g f(x), v) = v(f), \quad v \in T_x \Gamma.
\end{equation}
The Riemannian gradient $\nabla_g f$ defines a smooth vector field on $\Gamma$. If we choose a local chart $C = (O, \kappa)$ around $x = \kappa^{-1}(y) \in \Gamma$, then the coordinates of $\nabla_gf(x)$ in the associated basis \cref{eq:basis-tx} are 
% \edn{[consistent pullback notation?]}
\begin{equation}
(\nabla_g f)^{C,i}(y) = \sum_{j = 1}^{d-1} g^{C,ij}(y) \frac{\partial(f \circ \kappa^{-1})}{\partial y_j}\bigg|_{y},
\end{equation}
where $(g^{C,ij}(y))_{1\leq i,j\leq d-1}$ is the inverse of the matrix $(g^C_{ij}(y))_{1\leq i,j\leq d-1}$ introduced above. 

Since $\Gamma$ is the boundary of the smooth bounded domain $\Omega$, $\Gamma$ can be oriented by means of an outward-pointing vector field \cite[Section 16.1.3]{RouLeb22b}, and there exists a canonical {Riemannian} \emph{volume form} $\d V_g$; see, e.g., \cite[Section 17.1.4]{RouLeb22b}. For our purpose, it is enough to define the \emph{integral} of functions $f \in \C(\Gamma)$ with respect to $\d V_g$. The manifold $\Gamma$ is compact and thus we may find local  charts 
$C_n = (O_n, \kappa_n)$, $n = 1, \dots, N$  such that the $O_n$ constitute an open covering of $\Gamma$. Let $\theta_1, \dots, \theta_N$ form a \emph{partition of unity} subordinate to that covering (see, e.g., \cite[Section 15.2.3]{RouLeb22b}, \cite[Lemma 9.3]{Bre11book}, and also the next subsection). Then,
\begin{equation}
\label{eq:vol-form}
\int_\Gamma f \, \d V_g = \sum_{n = 1}^{N} \int_{\kappa_n(O_n)} \theta_n(\kappa_n^{-1}(y))f(\kappa_n^{-1}(y)){\det(g^{C_n}(y))}^{1/2} \, \d y, \quad f \in \C(\Gamma),
\end{equation}
where $\d y$ denotes the Lebesgue measure on $\R^{d-1}$. The integral in \cref{eq:vol-form} does not depend on the choice of local charts.
% whose representative in a local chart $C = (O, \kappa)$ around $x \in \Gamma$ satisfies
% \begin{equation}
% \d V_g^C(y) = (\det(g^C(y)))^{1/2} \, \d x^{C,1} \wedge \dots\wedge \d x^{C,d-1},
% \end{equation}
% where the $\d x^{C,i}$ form the dual basis of \cref{eq:basis-tx} and $g^C(y)$ denotes the matrix $(g^C_{ij}(y))_{1\leq i,j \leq d-1}$;
% see \cite[Section 15.6.2]{RouLeb22b} for a definition of the wedge product, and also \cite[Section 17.1.4]{RouLeb22b} for details regarding this volume form.
% \edn{à detailler et bien comprendre..}. 
% For our purpose,
With that in hand, we define the \emph{divergence} $\dive_g u$ of a smooth vector field $u$ on $\Gamma$ as the unique element in $\C^{\infty}(\Gamma)$ satisfying
\begin{equation}
\int_\Gamma \varphi \dive_g u \, \d V_g = - \int_\Gamma g(\nabla_g \varphi, u) \, \d V_g, \quad \varphi \in \C^\infty(\Gamma).
\end{equation}
Finally, the {Laplace--Beltrami} operator $\Delta_g$ on $\Gamma$ is defined by
% $\mathscr{L}(U, Y) \L(U, Y) \mathscr{F}$ $\mathscr{S}'(\R^d)$
\begin{equation}
\Delta_g f  =\dive_g \nabla_g f, \quad f \in \C^\infty(\Gamma).
\end{equation}
Given a local chart $C = (O, \kappa)$ around $x \in \Gamma$ with $y = \kappa(x)$, we have
\begin{equation}
\label{eq:pullback-beltrami}
\Delta_g u(x) = \det(g^C(y))^{-1/2} \sum_{i,j=1}^{d-1} \mleft.\frac{\partial}{\partial y_i} \mleft( \det(g^C(y))^{1/2} g^{C,ij}(y) \frac{\partial(u\circ \kappa^{-1})}{\partial y_j} \mright)\mright|_{y}.
\end{equation}
% \edn{A compléter...} \edn{voir avec Hugo l'histoire de densités sur des variétés...}
Here, the differential operators $\nabla_g$, $\dive_g$ and $\Delta_g$ act on smooth functions and vector fields. 

By Riesz' representation theorem, \cref{eq:vol-form} defines a unique Radon measure on $\Gamma$, the Riemannian measure. In fact, as $\Gamma$ is the boundary of $\Omega$, this is precisely the \emph{surface measure} $\d \sigma$ induced by the Lebesgue measure; see also \cite[Proposition 17.2]{RouLeb22b}. The Riemannian measure gives a natural meaning to $L^p$-spaces on $\Gamma$, and in particular the (complex\footnote{We consider suitable complexifications of the previous objects.}) Hilbert space $L^2(\Gamma)$ is the space of all (equivalence classes of) measurable functions $f : \Gamma \to \Co$ such that
\begin{equation}
\label{eq:L2-gamma}
\int_\Gamma |f|^2 \, \d V_g = \int_\Gamma |f|^2 \, \d \sigma < + \infty;
\end{equation}
equivalently, it is the completion of the space of smooth functions on $\Gamma$ with respect to the norm defined by \cref{eq:L2-gamma}. Similarly, the Sobolev space $H^1(\Gamma)$ is the completion of smooth functions with respect to
\begin{equation}
\label{eq:H1-gamma}
\|f\|^2_{H^1(\Gamma)} \triangleq \int_\Gamma g(\nabla_g f, \overline{\nabla_g f}) \, \d V_g.
\end{equation}
At each $x \in \Gamma$, the map $T_x i$ embeds the tangent space $T_x \Gamma$ into $\R^{d}$, and by letting $\nabla_\Gamma f(x) \triangleq T_xi(\nabla_gf(x))$ we may simply rewrite \cref{eq:H1-gamma} as
\begin{equation}
\|f\|^2_{H^1(\Gamma)} = \int_\Gamma \nabla_\Gamma f \cdot \overline{\nabla_\Gamma f} \, \d \sigma = \int_\Gamma |\nabla_\Gamma f|^2 \, \d \sigma.
\end{equation}
In the context of \cref{eq:IBVP}, it is particularly important to note that, if $f$ is the trace on $\Gamma$ of a sufficiently smooth function $u$ defined in $\Omega$, then $\nabla_\Gamma f$ coincides with the \emph{tangential gradient} (or surface gradient) of $u$: in other words,
if $x \in \Gamma$ and $\nu(x)$ is the unit outward normal vector at $x$, then $\nabla_\Gamma f(x)$ is tangent to $\Gamma$ at $x$ and
\begin{equation}
\nabla u(x) = \nabla_\Gamma f(x) + (\nabla u(x) \cdot \nu(x)) \nu(x).
\end{equation}
% dz
% Viewing $\Gamma$ as a submanifold of $\R^d$, this is precisely the
% This gives a natural
% In fact,
% Having in mind Sobolev spaces on $\Gamma$, let us briefly come back to the Riemannian volume form $\d V_g$.
With that in mind and when there is no risk of confusion, we will directly write $\nabla_\Gamma u$ instead of $\nabla_\Gamma f$ or $\nabla_\Gamma (u|_\Gamma)$.
We shall also use fractional Sobolev spaces $H^s(\Gamma)$, $s \in \R$, which are modeled after $H^s(\R^{d-1})$ via
local charts; see for instance \cite[Chapter 4, Section 3]{Tay96book}. Since $\Gamma$ is a manifold without boundary, note that $H^s(\Gamma)$ and $H^{-s}(\Gamma)$ are dual. Similarly as in the Euclidian case,  differential operators on $\Gamma$ have natural extensions to Sobolev spaces and even \emph{distributions}; see, e.g., \cite[Sections 16.3.3 and 18.2.1]{RouLeb22b}. In particular, $\Delta_g$, which we will now write $\Delta_\Gamma$, extends to a bounded linear operator from $H^2(\Gamma)$ into $L^2(\Gamma)$ and from $H^1(\Gamma)$ into $H^{-1}(\Gamma)$; see also \cite[Section 17.3]{RouLeb22b}.

\subsubsection{Local setting near the boundary}
\label{sec:normal-geodesic}
Throughout the manuscript we will use the following notation:
\begin{equation}
D_{y_i} = - \i \partial_{y_i}, \quad i  =1, \dots, d.
\end{equation}
Let $x^0 \in \Gamma$.
By \cite[Theorem 9.7]{RouLeb22a} there exists a local chart $C = (O, \kappa)$ (in $\R^d$) around $x^0$ such that, writing $\kappa = (\kappa_1, \dots, \kappa_d)$, $y = \kappa(x) = (\kappa_1(x), \dots, \kappa_d(x)) = (y_1, \dots, y_d)$:
\begin{enumerate}
    \item  $\kappa(x^0) = 0$ and
    \begin{equation}
    \label{eq:para-gamma}
    \kappa(\Omega \cap O) = \{y \in \kappa(O) : y_d \geq 0\}, \quad \kappa(\Gamma \cap O) = \{ y \in \kappa(O) : y_d = 0 \};
    \end{equation}
    \item There exist smooth coefficients $a_{ij} : \kappa(O) \to \R$, $1 \leq i,j \leq d-1$, such that
    \begin{equation}
    \label{eq:pullback-laplacian}
    - \Delta^C \triangleq - \Delta (\kappa^{-1})^\ast    = D^2_{y_d}  + \sum_{i,j =1}^{d-1} a_{ij}(y) D_{y_i} D_{y_j} \quad \mbox{in}~\kappa(O)
    \end{equation}
    (with pullback notation, $(\kappa^{-1})^\ast f = f \circ \kappa^{-1}$);
    \item 
    There exists a constant $c > 0$ such that
    \begin{equation}
    \label{eq:elliptic}
    \sum_{i,j = 1}^{d-1} a_{ij}(y) \xi_i \xi_j \geq c |\xi_\tau|^2, \quad \xi_\tau = (\xi_1, \dots, \xi_{d-1}) \in \R^{d-1}, \quad y \in \kappa(O).
    \end{equation}
\end{enumerate}
Such local coordinates are referred to as \emph{normal geodesic coordinates}. In fact, it is possible to choose such a local chart  in a way that guarantees that the pullback of the normal derivative $\partial_\nu = \nabla \cdot \nu $ coincides, on the boundary, with the derivative with respect to the normal variable $y_d$:
\begin{equation}
\partial_\nu^C \triangleq \partial_\nu (\kappa^{-1})^\ast = \partial_{y_d} = \i D_{y_d};
\end{equation}
see the proof of \cite[Theorem 9.7]{RouLeb22a} and in particular \cite[Equation (9.4.5)]{RouLeb22a}.

Since $\Gamma_0$ is compact, there exists an open covering of $\Gamma_0$ made of such charts $C_n = (O_n, \kappa_n)$, $n = 1, \dots, N$. We use a partition of unity: by, e.g., \cite[Lemma 9.3]{Bre11book} there exist smooth functions $\theta_n : \R^{d} \to [0, 1]$, $n = 0, \dots, N$ such that:
\begin{enumerate}
    \item 
\begin{equation}
\sum_{n=0}^N\theta_n(x) = 1, \quad x \in \R^d;
\end{equation}
    \item For $n = 1, \dots, N$, $\theta_n|_{O_n} \in \D(O_n)$;
    \item The support of $\theta_0$ is contained in $\R^{d} \setminus \Gamma_0$ and $\theta_0|_{\Omega} \in \D(\Omega)$.
\end{enumerate}
Therefore, for each $n = 1,\dots N$, there exist smooth coefficients  $a_{ij}^n : \kappa_n(O_n) \to \R$, $1 \leq i,j \leq d-1$ such that $-\Delta^{C_n}$ has the form \cref{eq:pullback-laplacian}. Now, pick $\theta'_n \in \D(O_n)$ such that $\theta'_n = 1$ near the support of $\theta_n$, and let
\begin{equation}
\tilde{a}_{ij}^n \triangleq (\theta'_n \circ \kappa_n^{-1}) a_{ij}^n, \quad 1 \leq i,j\leq d-1;
\end{equation}
then the $\tilde{a}_{ij}^n$ trivially extend to smooth coefficients in $\R^{d}$ that are compactly supported and coincide with $a^n_{ij}$ near $\kappa_n(\supp(\theta_n))$, the image of the support of $\theta_n$ under $\kappa_n$. The differential operator on $\R^{d}$ defined by \cref{eq:pullback-laplacian} with $a^n_{ij}$ replaced by $\tilde{a}^n_{ij}$ is written \smash{$-\tilde{\Delta}^{C_n}$.}
On the other hand, by restricting each $\kappa_n$ to $O_n \cap \Gamma_0$ we see that $C_n$ also defines a local chart in $\Gamma_0$ as a $(d-1)$-dimensional submanifold of $\R^{d}$, and \cref{eq:pullback-beltrami} provides smooth coefficients $b^n_{ij} : \kappa_n(\Gamma_0 \cap O_n) \to \R$ such that
\begin{equation}
-\Delta_\Gamma^{C_n} = \sum_{i,j=1}^{d-1} b^n_{ij}(y_\tau) D_{y_i}D_{y_j} \quad \mbox{in}~ \kappa_n(\Gamma_0 \cap O_n),
\end{equation}
with, for some $c> 0$, the ellipticity condition
\begin{equation}
\label{eq:ellip-belt}
\sum_{i,j}^{d-1} b_{ij}^n(y_\tau) \xi_i \xi_j \geq c|\xi_\tau|^2, \quad y_\tau \in \kappa(O_n \cap \Gamma_0), \quad \xi_\tau = (\xi_1, \dots, \xi_{d-1})\in\R^{d-1}.
\end{equation}
% where $y_\tau$ indicates the tangential variable $(y_1, \dots, y_{d-1}) \in \R^{d-1}$. 
Exactly as before, using the cutoffs $\theta'_n$ we define localized coefficients $\tilde{b}_{ij}^n$ and differential operators $-\tilde{\Delta}_\Gamma^{C_n}$ on $\R^{d-1}$. Since $\Gamma_0$ is covered by a finite number of those charts $C_n$, we may choose an ellipticity constant $c > 0$ in \cref{eq:elliptic,eq:ellip-belt} that does not depend on $n$. 

% dz

% \edn{[TO CHECK: pullback of normal derivative in those coordinates]}
% \edp{\color{red} C'est bien $D_{y_d}$! je me suis planté, on fait il y a des gens meme dans le papier de Rémi, il juste dit dans cet coordenées on l'a.  \color{blue}Recall that the normal derivative of $u$ in $x$–coordinates is
% \[
% \partial_\nu u(x)=\nabla_x u(x)\cdot \nu(x).
% \]
% Under the change of variables $y=\kappa(x)$ we set 
% $\tilde u=u\circ\kappa^{-1}$. Using the chain rule,
% \[
% D_{y_d}\tilde u(y)
%   =\nabla_x u(x)\cdot \partial_{y_d}x,
% \]
% so identifying $\partial_{y_d}x$ with the normal direction shows that
% $D_{y_d}\tilde u$ coincides with $\partial_\nu u$ on the boundary. Identifying the last coordinate $y_d$ with the parameter $t$, we write $x=\kappa^{-1}(y_\tau,y_d)=x(y_\tau,t)$, 
% where $x(y_\tau,t)$ denotes the solution of the ODE system \cite[Eq.~(9.4.5)]{RouLeb22a}. In particular, we have
% \[\dot x_d=1, \qquad 
% x_j(y_\tau,0)=y_j\quad (j<d), 
% \qquad x_d(y_\tau,0)=0.
% \]
% therefore $x_d(y_\tau,t)=t$ and $\partial_{y_i}x_d=0$ for $i<d$ and
% $\partial_{y_d}x_d=1$. Thus $\partial_{y_i}x$ ($i<d$) are tangential and
% $\partial_{y_d}x$ is normal. Hence, by the chain rule,
% \[
% D_{y_d}(u\circ\kappa^{-1})(y_\tau,0)=\partial_\nu u(x).
% \]

% }
%%%%%%%%%%%%%%%%%%%%%%%%%%%%%%%%%%%%%%%%%%%%%%
\subsection{Well-posedness and asymptotic stability}
\label{sec:wp-as}

In this section, we establish well-posedness of \cref{eq:IBVP} and asymptotic convergence to zero of the energy of solutions. What we present here is either standard or relatively straightforward, and there might be some overlap with, e.g., \cite{Buffe17,Vit17}. Still, we go over most details for the sake of clarity and self-containedness.

\subsubsection{Functional setting and semigroup generation}
\label{sec:sg-generation}

By multiplying the wave equation \cref{eq:wave-equation} by a smooth  function $\varphi$ of the space variable and using Green's formula along with the boundary conditions \cref{eq:HBC,eq:dirichlet}, we (formally) derive a  weak formulation of \cref{eq:IBVP}: for all sufficiently smooth $\varphi$ with $\varphi|_{\Gamma_1} = 0$, solutions $u$ to \cref{eq:IBVP} satisfy
\begin{equation}
\label{eq:weak-form}
\dtt{} \mleft (\int_\Omega u \varphi \, \d x + \int_{\Gamma_0} u \varphi \, \d \sigma\mright)
+ \dt{} \int_{\Gamma_0} u \varphi \, \d \sigma
+ \int_\Omega \nabla u \cdot \nabla \varphi \, \d x + \int_{\Gamma_0} \nabla_\Gamma u \cdot \nabla_\Gamma \varphi \, \d \sigma = 0.
\end{equation}
\Cref{eq:weak-form} suggests working with the following spaces:
\begin{subequations}
\begin{align}
&H \triangleq L^2(\Omega) \times L^2(\Gamma_0), \\
\label{eq:V} &V \triangleq \{ (u, w) \in H^1(\Omega) \times H^1(\Gamma_0) : u|_{\Gamma_0} = w \} \simeq \{ u \in H^1(\Omega) : u|_{\Gamma_0} \in H^1(\Gamma_0)\}.
\end{align}
\end{subequations}
Both spaces are Hilbert spaces if equipped with their natural product structures, but whenever $\Gamma_1 \not = \emptyset$ we equip  $V$ with the (equivalent, due to a standard Poincar\'e-type inequality for functions vanishing on $\Gamma_1$) norm $\|\cdot \|^2_V \triangleq \int_\Omega |\nabla\cdot|^2 \, \d x + \int_{\Gamma_0} |\nabla_\Gamma \cdot|^2 \, \d \sigma$.
The identification in \cref{eq:V} is straightforward, and we shall use it without further comment. 
Furthermore,  $V$ as a subspace of $H$ is dense, the embedding of $V$ into $H$ is clearly continuous. As a result, we may consider a standard pivot duality
\begin{equation}
\label{eq:pivot}
V \hookrightarrow H \hookrightarrow V^\ast,
\end{equation}
where $V^\ast$ stands for the topological antidual of $V$ endowed with the dual norm, and $H$ is identified with its antidual by means of Riesz' theorem. Recall also that on smooth compact manifolds the Rellich--Kondrachov theorem provides compact Sobolev embeddings \cite[Chapter 3, Proposition 4.4]{Tay11book}, and thus the embeddings in \cref{eq:pivot} are in fact {compact}. Next, we  define a duality map $A \in \L(V, V^\ast)$ by
\begin{equation}
\label{eq:A-dual}
\langle Au, v \rangle_{V^\ast, V}  \triangleq \int_\Omega \nabla u \cdot \overline{\nabla v} \,   \d x + \int_{\Gamma_0} \nabla_\Gamma u \cdot \overline{\nabla_\Gamma v} \, \d \sigma, \quad u, v \in V,
\end{equation}
as well as a control operator $B \in \L(L^2(\Gamma_0), H)$ by
$
Be = (0, e)
$ for all $e \in L^2(\Gamma_0)$. Then, \cref{eq:IBVP} takes the form of the abstract evolution equation
\begin{equation}
\label{eq:IBVP-abstract}
u'' + BB^\ast u' + Au = 0,
\end{equation}
where the adjoint $B^\ast \in \L(H, L^2(\Gamma_0))$ of $B$ is given by $B^\ast(v,w) = w$ for $(v, w) \in H$; in particular, $B^\ast v = v|_{\Gamma_0}$ whenever $v \in V$.  Now, in view of \cref{eq:pivot} we may see $A$ as an unbounded operator on $H$ by letting
\begin{equation}
\label{eq:H-real-A}
\dom(A) \triangleq \{ u \in V : Au \in H \},
\end{equation}
so that $\langle Au, v\rangle_{V^\ast,V} = \langle Au, v\rangle_H$ for all $u \in \dom(A)$ and $v \in V$. This $H$-realization of $A$ is a nonnegative (strictly positive if $\Gamma_1 \not = \emptyset$) self-adjoint operator with compact resolvent, and here we are able to give an explicit characterization of its domain.
\begin{prop}[Characterization of the domain]
\label{prop:carac-dom-A}
We have
\begin{subequations}
\begin{align}
\label{eq:carac-dom-A}
&\dom(A) = \mleft \{ u \in H^2(\Omega) : u|_{\Gamma_0} \in H^2(\Gamma_0),~ u|_{\Gamma_1} = 0 \mright \}, \\
\label{eq:A-explicit}
& Au =  (-\Delta u, -\Delta_\Gamma u + \partial_\nu u), \quad u \in \dom(A),
\end{align}
\end{subequations}
and the norm $\|\cdot\|_{H^2(\Omega)} + \|\cdot\|_{H^2(\Gamma_0)}$ is equivalent to the graph norm of $A$. 
% Also, $Au = (-\Delta u, -\Delta_\Gamma u + \partial_\nu u)$ for all $u \in \dom(A)$.
\end{prop}
\begin{proof}
Let $u \in V$. Then $u \in \dom(A)$ if and only if there exists $(f, e) \in H$ such that
\begin{equation}
\label{eq:weak-A-dom}
\int_\Omega \nabla u \cdot \overline{\nabla v} \,   \d x + \int_{\Gamma_0} \nabla_\Gamma u \cdot \overline{\nabla_\Gamma v} \, \d \sigma = \int_\Omega f \overline{v} \, \d x + \int_{\Gamma_0} e \overline{v} \, \d \sigma, \quad v \in V.
\end{equation}
If $u$ belongs to the set at the right-hand side of \cref{prop:carac-dom-A} then we see, using Green's formula, that \cref{eq:weak-A-dom} holds with $(f,e) = (-\Delta u, -\Delta_\Gamma u + \partial_\nu u) \in H$. Conversely, if $u \in \dom(A)$ then \cref{eq:weak-A-dom} with $v$ taken in $\D(\Omega)$ shows that $-\Delta u = f \in L^2(\Omega)$, where $-\Delta u$ was defined \emph{a priori} in $\D(\Omega)$. Since $u \in H^1(\Omega)$ and $\Delta u \in L^2(\Omega)$, recall that $\partial_\nu u$ is (continuously) well-defined in $H^{-1/2}(\Gamma)$ by
\begin{equation}
\label{eq:weak-norm}
\langle \partial_\nu u, \psi\rangle = \int_\Omega \nabla u \cdot \nabla \varphi + \Delta u \varphi \, \d x = \int_\Omega \nabla u \cdot \nabla \varphi - f  \varphi \, \d x , \quad \psi \in H^{1/2}(\Gamma), 
\end{equation}
where $\varphi$ is any element of $H^1(\Omega)$ such that $\varphi|_{\Gamma} = \psi$ (such a function always exists), and the brackets denote the bilinear (as opposed to sesquilinear in, e.g., \cref{eq:A-dual}) duality pairing between $H^{-1/2}(\Gamma)$ and $H^{1/2}(\Gamma)$. Plugging \cref{eq:weak-norm} into \cref{eq:weak-A-dom} with $\varphi = \overline{v}$, we then see that
\begin{equation}
\label{eq:weak-bpde}
\int_{\Gamma_0} \nabla_\Gamma u \cdot \overline{\nabla_\Gamma v} \, \d \sigma = - 
\langle \partial_\nu u, \overline{v}\rangle + \int_{\Gamma_0}e \overline{v} \, \d \sigma, \quad v \in V.
\end{equation}
In the case $\Gamma_1 = \emptyset$ we  infer from \cref{eq:weak-bpde} that $-\Delta_\Gamma u = -\partial_\nu u + e$ in the sense of distributions on $\Gamma = \Gamma_0$; indeed, $V$ contains (say) $H^2(\Omega)$, so that the image of $V$ under the trace map contains the entirety of $H^{3/2}(\Gamma)$, including $\D(\Gamma)$. If $\Gamma_1 \not = \emptyset$, since anyway, $\overline{\Gamma_0} \cap \overline{\Gamma_1} = \emptyset$ we can find a function $\rho \in \C^\infty(\overline{\Omega})$ such that $\rho = 0$ near $\Gamma_1$ and $\rho = 1$ near $\Gamma_0$. Then, $\{ \rho v, v \in H^2(\Omega)\} \subset V$ and similarly we arrive at the conclusion that, in the sense of distributions,  $-\Delta_\Gamma u = - \partial_\nu u +e$ on $\Gamma_0$. In any case,  elliptic regularity for the Laplace--Beltrami operator on the boundaryless manifold $\Gamma_0$ leads to $u|_{\Gamma_0} \in H^{3/2}(\Gamma_0)$; use for instance \cite[Chapter 5, Theorem 1.3]{Tay11book} and interpolation. Because $\overline{\Gamma_0} \cap \overline{\Gamma_1} = \emptyset$ and $u = 0$ on $\Gamma_1$, it  follows in fact that $u|_{\Gamma} \in H^{3/2}(\Gamma)$. Now, recalling that $\Delta u \in L^2(\Omega)$, by elliptic regularity for the Laplacian with non-homogeneous boundary data we deduce that $u \in H^2(\Omega)$; see, e.g., \cite[Preface or Chapter 2, Theorem 5.2]{LioMag68book}. Then trace theory yields $\partial_\nu u \in H^{1/2}(\Gamma) \subset L^2(\Gamma)$ and using again elliptic regularity for $\Delta_\Gamma$ leads to $u|_{\Gamma_0} \in H^2(\Gamma_0)$, which completes the proof of the inclusion from left to right in \cref{eq:carac-dom-A}. \Cref{eq:A-explicit} follows from an application of Green's formula in \cref{eq:weak-A-dom}. We omit the detailed proof of the last statement of \cref{prop:carac-dom-A}: simply keep track of the norms in the above arguments.
% recall from \cite{Kom94book}
\end{proof}

% \begin{coro}[Compactness]
% The operator $A : \dom(A) \to H$ has compact resolvent.
% \end{coro}
% \begin{proof}
% dz
% \end{proof}
Note that the special second-order structure \cref{eq:IBVP-abstract} is now standard in the literature and in particular corresponds to the setting of \cite[Section 3]{AnaLea14}; see also, e.g., \cite{AmmTuc01,AmmNic14book,ChiPau23}. As we know, the natural phase space for \cref{eq:IBVP} is
\begin{equation}
\H \triangleq V \times H
\end{equation}
% Note that the second-order structure \cref{eq:IBVP-abstract} is well-understood in the literature
equipped with its product Hilbertian structure. We define an unbounded operator $\A$ on $\H$ by
\begin{equation}
\A \triangleq \begin{pmatrix} 0 & - 1 \\ A & BB^\ast \end{pmatrix}, \quad \dom(\A) \triangleq \dom(A) \times V.
% \begin{aligned} \dom(\A) &\triangleq \dom(A) \times V \\ & = \left \{ (u_0, u_1) \in H^2(\Omega) \times H^1(\Omega) :
% (u_0, u_1)|_{\Gamma_0} \in H^2(\Gamma_0) \times H^2(\Gamma_0)\right \},
% \end{aligned}
\end{equation}
With \cref{prop:carac-dom-A} we immediately see that
% where we used \cref{prop:carac-dom-A} to give the explicit form of the domain.
\begin{equation}
\dom(\A) = \left \{ (u_0, u_1) \in H^2(\Omega) \times H^1(\Omega) :
 (u_0, u_1)|_{\Gamma_0} \in H^2(\Gamma_0) \times H^1(\Gamma_0),~ (u_0, u_1)|_{\Gamma_1} = 0\right \}
\end{equation}
% $\dom(\A)$ is made of all pairs $(u_0, u_1) \in H^1(\Omega) \times H^1(\Omega)$ satisfying
Well-posedness of \cref{eq:IBVP}, through its abstract formulation \cref{eq:IBVP-abstract}, is guaranteed by the following proposition.
\begin{prop}[Semigroup generation]
\label{prop:wp}
The operator $-\A$ is the infinitesimal generator of a strongly continuous semigroup $\{e^{-t\A} \}_{t \geq 0}$ of linear bounded operators on $\H$. If $\Gamma_1 \not = \emptyset$ then $\{e^{-t\A} \}_{t \geq 0}$ is in fact a contraction semigroup.
\end{prop}
\begin{proof}
Up to a bounded perturbation if $\Gamma_1 = \emptyset$, this is a standard application of the Lumer--Phillips theorem on contraction semigroups: use the special structure of \cref{eq:IBVP-abstract} along with, for instance, \cite[Chapter II, Corollary 3.20 and Proposition 3.23]{EngNag00book} and \cite[Chapter III, Theorem 1.3]{EngNag00book}.
\end{proof}
\begin{rem}
In fact, $\{e^{-t\A}\}_{t\geq0}$ is a strongly continuous \emph{group}.
\end{rem}
Therefore, the unique \emph{weak} solution $u$ (in the sense of \cref{eq:weak-form,eq:IBVP-abstract})  to \cref{eq:IBVP} with initial data $(u_0, u_1) \in \H$ is given by $(u(\cdot, t), \partial_tu(\cdot, t)) = e^{-t\A}(u_0, u_1)$ for all $t\geq 0$, so that \begin{equation}
u \in \C([0, +\infty), V) \cap \C^1([0, +\infty), H).
\end{equation}
In addition, when $(u_0, u_1) \in \dom(\A) = \dom(A) \times V$, then $u \in \C([0, +\infty), \dom(A)) \cap \C^1([0, +\infty), V)$ is a \emph{classical} solution in the semigroup sense and satisfies \cref{eq:IBVP} in a pointwise a.e.\ fashion.
% \begin{rem}
% If $(u_0, u_1)$ is merely in $\H$ then the trace (should it make sense)... \edn{remarque sur la "compatibilité"...}
% \end{rem}
We conclude this section with an easy but important fact concerning the operator $\A$.

\begin{lemma}
\label{lem:A-eigen}
The operator $\A$ has compact resolvent and in particular its spectrum $\sigma(\A)$ is made entirely of isolated eigenvalues with finite multiplicity. Also, $\ker(\A) = \ker(A) \times \{0\}$.
\end{lemma}
\begin{proof}
The first statement readily follows from compactness of the resolvent of $A$ and the Riesz--Schauder theory for compact operators; see, e.g., \cite[Section X.5, Theorems 1 and 2]{Yos65book} or \cite[Chapter IV, Corollary 1.19]{EngNag00book}. The second statement is clear.
\end{proof}

\subsubsection{Energy and stability}
\label{sec:ener-stab}

Clearly, the energy $E$ defined in \cref{eq:energy} as well as the energy identity \cref{eq:energy-id} make sense for all weak solutions to \cref{eq:IBVP}. In the case $\Gamma_1 \not = \emptyset$, recalling our normalization of $V$, for all solutions $u$ to \cref{eq:IBVP} with initial data $(u_0, u_1) \in \H$ we have
\begin{equation}
\label{eq:energy-sg}
E(u, t) = \frac{1}{2}\|(u(\cdot, t), \partial_t u(\cdot, t))\|^2_\H = \frac{1}{2}\|e^{-t\A}(u_0, u_1)\|^2_\H, \quad t \geq 0.
\end{equation}
On the other hand, if $\Gamma_1 = \emptyset$ then the energy only defines a seminorm on $\H$. Following \cite[Section 4]{AnaLea14} we  consider a Riesz projector $\Pi \in \L(\H)$ defined by
\begin{equation}
\Pi = \frac{1}{2\i \pi} \int_\gamma (\A + z)^{-1} \, \d z,
\end{equation}
where $\gamma$ is a positively oriented circle in $\Co$ centered on $0$ and of radius small enough so that by \cref{lem:A-eigen} the intersection of its interior with the spectrum of $\A$ is either reduced to $\{0\}$ if $\Gamma_1$ is empty, or else  empty, in which case $\Pi = 0$. Then, $\Pi \H = \ker(\A)$, $\dot{\H} \triangleq (1 - \Pi) \H$ equipped with the norm
\begin{equation}
\label{eq:norm-energy}
\|(u_0, u_1)\|^2_{\dot{\H}} \triangleq \langle Au_0, u_0\rangle_{V^\ast,V} + \|u_1\|^2_H, \quad (u_0, u_1) \in \dot{\H},
\end{equation}
is a Hilbert space and, as a closed subspace of $\H$, is in topological direct sum with $\ker(\A)$. Furthermore, the operator $\dot{\A} \triangleq \A|_{\dot{\H}}$ with domain $\dom(\dot{\A}) \triangleq \dom(\A) \cap \dot{\H}$ generates a contraction semigroup \smash{$\{e^{-t\dot{\A}}\}_{t\geq 0}$} on $\dot{\H}$ and
\begin{equation}
\label{eq:dec-sg}
e^{-t\A}(u_0, u_1) =  e^{-t\dot{\A}}(1 - \Pi)(u_0, u_1) + \Pi(u_0, u_1), \quad t \geq 0, \quad (u_0, u_1) \in \H;
\end{equation}
see \cite[Lemma 4.3]{AnaLea14}. Note that $\sigma(\dot{\A}) = \sigma(\A) \setminus \{0\}$. As a consequence of \cref{eq:norm-energy,eq:dec-sg},  for all solutions $u$ to \cref{eq:IBVP} with initial data $(u_0, u_1) \in \H$,
\begin{equation}
\label{eq:energy-quotient-sg}
E(u, t) = \frac{1}{2} \|e^{-t\dot{\A}}(1 - \Pi)(u_0, u_1)\|^2_{\dot{\H}} = \frac{1}{2}\|(1 - \Pi)e^{-t\A}(u_0, u_1)\|^2_{\dot{\H}}, \quad t \geq 0.
\end{equation}
(In the case $\Gamma_1 \not = \emptyset$, \cref{eq:energy-sg,eq:energy-quotient-sg} are the same.) It is also convenient to define for smoother solutions a higher-order energy $E_1$ as follows: for (classical) solutions $u$ to \cref{eq:IBVP} with initial data $(u_0, u_1) \in \dom(A) \times V$, and for $t \geq 0$,
\begin{equation}
\label{eq:high-order-e}
E_1(u, t) \triangleq \frac{1}{2} \int_\Omega |\nabla u|^2 + |\Delta u|^2 + |\partial_t u|^2 + |\nabla \partial_t u|^2 \, \d x  +  \frac{1}{2} \int_{\Gamma_0} |\Delta_\Gamma u|^2 + |\partial_t u|^2 + |\nabla_\Gamma \partial_t u|^2 \, \d \sigma.
\end{equation}
Note that \cref{eq:high-order-e} makes sense for such solutions due to \cref{prop:carac-dom-A}; in fact, $E_1(u, \cdot)$ is a continuous and bounded function of the time variable. The following lemmas relate $E_1$ to the graph norm of $\dot{\A}$

\begin{lemma}\label{lem:poinc} There exists $K > 0$ such that
\begin{equation}
\label{eq:poinc}
\frac{1}{2K} \|Au\|^2_H \leq  \mleft( \int_\Omega |\Delta u|^2  + |\nabla u|^2 \, \d x +  \int_{\Gamma_0} |\Delta_\Gamma u|^2  \, \d \sigma \mright) \leq \frac{K}{2}\|Au\|^2_H, \quad u \in \dom(A).
\end{equation}
\end{lemma}
\begin{proof}
Given $u \in \dom(A)$, define $\tilde{u}$ a.e.\ by
\begin{equation}
\tilde{u} \triangleq \left \{
\begin{aligned}
&u - \frac{1}{|\Omega|} \int_{\Omega}  u \, \d x &&\mbox{if}~\Gamma_1 = \emptyset, \\
&u &&\mbox{if}~\Gamma_1 \not = \emptyset,
\end{aligned}
\right.
\end{equation}
where $|\Omega|$ denotes the Lebesgue measure of $\Omega$; then, $\tilde{u} \in \dom(A)$.
% \edn{and that \cref{eq:poinc} holds for $u$ if and only if it holds for $\tilde{u}$.} 
In what follows, $K, K'\dots$ denote positive constants that may change from line to line but  are independent of the choice of $u \in \dom(A)$.
First, by \cref{prop:carac-dom-A},
\begin{equation}
\|Au\|_H = \|A \tilde{u}\|_H\leq K \left(\|\tilde{u}\|_{H^2(\Omega)} +  \|\tilde{u}\|_{H^2(\Gamma_0)}\right) \leq K' \|Au\|_H, \quad u \in \dom(A).
\end{equation}
This already gives the right inequality in \cref{eq:poinc}.
Furthermore,
 elliptic theory \cite{LioMag68book} yields 
\begin{equation}
\|\tilde{u}\|_{H^2(\Omega)} \leq K \mleft ( \|\Delta \tilde{u}\|_{L^2(\Omega)}  + \|\tilde{u}\|_{L^2(\Omega)}+ \|\tilde{u}\|_{H^{3/2}(\Gamma_0)} \mright), \quad u \in \dom(A).
\end{equation}
By the Poincar\'e--Wirtinger  inequality (if $\Gamma_1 = \emptyset$) or the Poincar\'e-type inequality for functions vanishing on $\Gamma_1$ (if $\Gamma_1 \not = \emptyset$),
\begin{equation}
\|\tilde{u}\|_{L^2(\Omega)} \leq K  \|\nabla \tilde{u}\|_{L^2(\Omega)^{d}} = K \|\nabla u\|_{L^2(\Omega)^{d}}, \quad u \in \dom(A).
\end{equation}
On the other hand, with elliptic theory on $\Gamma_0$  \cite{Tay11book} and trace continuity we obtain
\begin{equation}
\|\tilde{u}\|_{H^{2}(\Gamma_0)} \leq K \left ( \|\Delta_\Gamma \tilde{u}\|_{L^2(\Gamma_0)}  + \|\tilde{u}\|_{L^2(\Gamma_0)} \right) \leq K \|\Delta_\Gamma u \|_{L^2(\Gamma_0)} + K' \|\nabla u\|_{L^2(\Omega)^d}, \quad u \in \dom(A).
\end{equation}
The left part of \cref{eq:poinc} follows at once.
\end{proof}

\begin{lemma}
\label{lem:sandwich-hot}
There exists $K > 0$ such that, for all (classical) solutions $u$ to \cref{eq:IBVP} with initial data $(u_0, u_1) \in \dom(A) \times V$,
\begin{equation}
\label{eq:sandwich-hot}
\frac{1}{2K} \| \dot{\A}(1 - \Pi)e^{-t\A}(u_0, u_1)\|^2_{\dot{\H}} \leq E_1(u, t) \leq \frac{K}{2} \| \dot{\A}(1 - \Pi)e^{-t\A}(u_0, u_1)\|^2_{\dot{\H}}, \quad t \geq 0.
\end{equation}
\end{lemma}

\begin{proof}
Recall that solutions to \cref{eq:IBVP} with data in $\dom(A)\times V$ remain in $\dom(A) \times V$; as a result it suffices to prove \cref{eq:sandwich-hot} for $t = 0$. For all $(u_0, u_1) \in \dom(A) \times V$,
% Again we denote by $K$ positive constants that do not depend on the data.
 we have
\begin{equation}
\label{eq:A-pi-A}
\begin{aligned}
\frac{1}{2} \|\dot{\A}(1 - \Pi)(u_0, u_1)\|^2_{\dot{\H}} &=  \frac{1}{2} \|(1 - \Pi) \A (u_0, u_1)\|^2_{\dot{\H}} \\ 
&= \frac{1}{2} \langle Au_1, u_1 \rangle_{V^\ast,V} + \frac{1}{2} \|Au_0 + BB^\ast u_1\|^2_H.
% \\
% & = \frac{1}{2} \int_\Omega 
\end{aligned}
\end{equation}
Since $BB^\ast \in \L(H)$, we readily deduce from \cref{eq:A-pi-A} that
there exists a constant $K > 0$ such that
\begin{equation}
\frac{1}{K} \|\dot{\A}(1 - \Pi)(u_0, u_1)\|^2_{\dot{\H}} \leq \langle Au_1, u_1 \rangle_{V^\ast,V} + \|u_1\|^2_H + \|Au_0\|^2_H \leq K \|\dot{\A}(1 - \Pi)(u_0, u_1)\|^2_{\dot{\H}}
\end{equation}
for all $(u_0, u_1) \in \dom(A) \times V$.
We complete the proof by applying \cref{lem:poinc} to $u_0$.
% % for all $(u_0, u_1) \in \dom(\A)$.
% % On the other hand, by elliptic regularity
% Since $\partial_\nu u_0 = \nabla u_0 \cdot \nu$, with trace continuity $H^1(\Omega) \to L^2(\Gamma_0)$ we obtain some constant $K > 0$ 
% such that
% \begin{equation}
% \end{equation}
% $K \int$
% \begin{equation}
% \|A u_0\|^2_H \leq \int_\Omega |\Delta u_0|^2 \, \d x + 2 \int_{\Gamma_0} |\Delta_\Gamma u_0|^2 + |\partial_\nu u_0|^2 \, \d \sigma \leq 
% \end{equation}
\end{proof}

Finally, we introduce a quadratic operator pencil $P$ defined by
\begin{equation}
\label{eq:pencil}
P(z) \triangleq A + zBB^\ast + z^2, \quad z \in \Co,
\end{equation}
which is related to the Laplace transform of the abstract evolution equation \cref{eq:IBVP-abstract}. For $z \in \Co$, $P(z)$ is well-defined as a map in $\L(V, V^\ast)$ but also as an unbounded operator on $H$ with $\dom(P(z)) = \dom(A)$. 
We are now ready to state and prove a first stability result.

\begin{prop}[Semi-uniform stability]
\label{prop:s-u-stab}
We have $\sigma(\dot{\A}) \cap \i \R = \emptyset$ and
\begin{equation}
\label{eq:s-u-stab}
\|e^{-t\dot{\A}}\dot{\A}^{-1}\|_{\L(\dot{\H})} \to 0, \quad t \to + \infty.
\end{equation}
Thus, there exists a rate $r : [0, +\infty) \to [0, +\infty)$ with $r(t) \to 0$ as $t \to + \infty$ such that, for (classical) solutions $u$ to \cref{eq:IBVP} with initial data $(u_0, u_1) \in \dom(A) \times V$,
\begin{equation}
% E(u, t) \leq r(t) \left ( \int_\Omega |\nabla u_0|^2 + |\Delta u_0|^2 + |u_1|^2 + |\nabla u_1|^2 \, \d x +  \int_{\Gamma_0} |\nabla_\Gamma u_0|^2 + |\Delta_\Gamma u_0|^2 + |u_1|^2 + |\nabla_\Gamma u_1|^2 \, \d \sigma  \right) 
\label{eq:s-u-energy}
E(u, t) \leq r(t) E_1(u, 0), \quad t \geq 0.
\end{equation}
Furthermore, for (weak) solutions $u$ to \cref{eq:IBVP} with initial data $(u_0, u_1) \in \H$,
\begin{equation}
\label{eq:as-stab-energy}
E(u, t) \to 0, \quad t \to + \infty.
\end{equation}
\end{prop}

\begin{proof}
Recall that the semigroup $\{e^{-t \dot{\A}}\}_{t \geq 0}$ is contractive and in particular uniformly bounded. By virtue of \cite[Theorem 1]{BatDuy08}, the property that $\sigma(\dot{\A}) \cap \i \R = \emptyset$ and \cref{eq:s-u-stab} are actually equivalent, so it suffices to prove the latter. Keeping in mind that the spectrum of $\dot{\A}$ is made of eigenvalues only, it is easy to see that, given $\lambda \in \R \setminus \{0\}$, $\i \lambda \in \sigma(\dot{\A})$ if and only if there exists $u \in \dom(A)$, $u \not = 0$, such that $P(\i \lambda)u = 0$; see for instance the proof of \cite[Lemma 4.2]{AnaLea14}.
For such $u$,
\begin{equation}
0 = \langle P(\i \lambda)u, u \rangle_{H} = \langle Au, u\rangle_H + \i \lambda \|u\|_{L^2(\Gamma_0)}^2 - \lambda^2 \|u\|^2_H,
\end{equation}
which leads to $u = 0$  on $\Gamma_0$ after taking the imaginary part ($A$ is self-adjoint). 
On the other hand, $u$ solves the stationary boundary value problem
\begin{subequations}
    \label{eq:im-mode-sys}
    \begin{align}
    \label{eq:im-mode}
    &(-\Delta - \lambda^2)u = 0 &&\mbox{in}~\Omega, \\ \label{eq:im-mode-bc}
    &(-\Delta_\Gamma + \i \lambda - \lambda^2)u  = - \partial_\nu u&&\mbox{on}~\Gamma_0,
    \end{align}
\end{subequations}
and \cref{eq:im-mode-bc} reduces to $\partial_\nu u = 0$ on $\Gamma_0$. Therefore, $u$ solves \cref{eq:im-mode}, an elliptic equation with analytic coefficients, and has vanishing Cauchy data on $\Gamma_0$; an application of the John--Holmgren theorem on unique continuation across non-characteristic hypersurfaces (see, e.g.,\cite[Uniqueness Theorem, Section 5]{Joh82book}) finally yields $u = 0$, which 
% is a contradiction and
proves the desired spectral condition.
Now we must turn the operator convergence property \cref{eq:s-u-stab} into energy decay \cref{eq:s-u-energy}. For all $(u_0, u_1) \in \dom(A) \times V = \dom(\A)$, $(1 - \Pi)(u_0, u_1) \in \dom(\dot{\A})$, i.e., $\dot{\A}(1 - \Pi)(u_0, u_1) \in \dot{\H}$: denoting by $u$ the solution to \cref{eq:IBVP} originating from $(u_0, u_1)$, with \cref{eq:energy-quotient-sg} we have
% combining \cref{eq:energy-quotient-sg,eq:s-u-stab} leads to
\begin{equation}
E(u, t) = \frac{1}{2} \|e^{-t\dot{\A}} \dot{\A}^{-1} \dot{\A}(1 - \Pi)(u_0, u_1)\|^2_{\dot{\H}} \leq \frac{1}{2} \|e^{-t \dot{\A}} \dot{\A}^{-1}\|_{\L(\dot{\H})}^2 \|\dot{\A}(1 - \Pi)(u_0, u_1)\|^2_{\dot{\H}}, \quad t \geq 0.
\end{equation}
By \cref{lem:sandwich-hot},
\begin{equation}
E(u, t) \leq K \|e^{-t\dot{\A}}\dot{\A}^{-1}\|^2_{\L(\dot{\H})} E_1(u, 0), \quad t \geq 0,
\end{equation}
so that letting
\begin{equation}
r(t) \triangleq K \|e^{-t\dot{\A}}\dot{\A}^{-1}\|^2_{\L(\dot{\H})} \to 0, \quad t \to + \infty,
\end{equation}
gives \cref{eq:s-u-energy}. In particular, $E(u, t) \to 0$ for a set of initial data that is dense in $\H$, and
% . Since the semigroup \smash{$\{e^{-t \dot{\A}}\}_{t \geq 0}$} is contractive and therefore uniformly bounded, 
thus \cref{eq:as-stab-energy} follows from a standard density argument along with \cref{eq:energy-quotient-sg}.
% Recall from, e.g., \cite[Lemma 4.2]{AnaLea14} that, given $\lambda \in \R \setminus \{0\}$, $\i \lambda \not \in \sigma(\dot{\A})$ if and only if $P(\i \lambda)$ is an isomorphism from $\dom(A)$ onto $H$.
\end{proof}

\begin{rem}[\emph{A priori} logarithmic decay rate]
% Although the setting is slightly different,
As mentioned in the introduction, the results, and analysis of \cite{Buffe17} suggest that, regardless of geometry and dynamical hypotheses, the rate $r$ in \cref{prop:s-u-stab} can be taken as
\begin{equation}
\label{eq:buffe-log}
r(t) = \frac{K}{\log(2 +t)^2}, \quad t \geq 0,
\end{equation}
for some positive constant $K$.  In fact, if $\Gamma_1 = \emptyset$ then \cref{eq:buffe-log} is precisely given by \cite[Theorem 1.2]{Buffe17}.
\end{rem}

\begin{rem}
From \cref{prop:s-u-stab} and its proof we see that $P(\i \lambda)^{-1}$ is well-defined in $\L(H)$ for all $\lambda \in \R \setminus \{0\}$.
\end{rem}

\section{Energy decay rate and analysis of quasimodes}

\subsection{Statement of the main result}

We now state the full version of \cref{th:intro}.

\begin{theo}[Polynomial stability]\label{th:poly-stab} Suppose that \cref{as:dyn} holds. Then,
\begin{equation}
\label{eq:poly-stab}
\|e^{-t\dot{\A}} \dot{\A}^{-1}\|_{\L(\dot{\H})} = O(t^{-1/2}), \quad t \to + \infty.
\end{equation}
Thus, there exists $K > 0$ such that, for all (classical) solutions $u$ to \cref{eq:IBVP} with initial data $(u_0, u_1) \in \dom(A) \times V$,
\begin{equation}
\label{eq:poly-energy}
E(u, t) \leq K (1 + t)^{-1} E_1(u, 0), \quad t \geq 0.
\end{equation}
Furthermore, for \emph{fixed} data $(u_0, u_1) \in \dom(A) \times V$,
\begin{equation}
\label{eq:o-energy}
E(u, t) = o(t^{-1}), \quad t \to +\infty.
\end{equation}
\end{theo}

As it is now standard, the asymptotic theory of strongly continuous semigroups on Hilbert spaces \cite{BorTom10} and the special structure \cref{eq:IBVP-abstract} of \cref{eq:IBVP} \cite{AnaLea14} will allow us to deduce \cref{th:poly-stab} from the following resolvent estimate.
% high-frequency asymptotics for $P(\i \lambda)^{-1}$.
\begin{prop}[Second-order resolvent estimate]\label{prop:res} Under \cref{as:dyn},
\begin{equation}
\label{eq:second-order-res-estimate}
\|P(\i \lambda)^{-1}\|_{\L(H)} = O(|\lambda|), \quad \lambda \to \pm \infty.
\end{equation}
\end{prop}

\begin{proof}[Proof of \cref{th:poly-stab} assuming \cref{prop:res}] To deduce \cref{eq:poly-stab} from \cref{eq:second-order-res-estimate}, use, e.g., \cite[Proposition 2.4]{AnaLea14}. Passing from \cref{eq:poly-stab} to \cref{eq:poly-energy} is done exactly as in the proof of \cref{prop:s-u-stab}. Finally, the $o$-improvement \cref{eq:o-energy}  for individual orbits is contained in \cite[Theorem 2.4]{BorTom10}, which we apply to the contraction semigroup generated by $- \dot{\A}$.
\end{proof}

\Cref{prop:res} is in fact the main achievement of this paper, and is proved in \cref{sec:quasimodes} below.

% \edn{Blablaa...}

\subsection{Analysis of quasimodes}

\label{sec:quasimodes}

After \emph{reductio ad absurdum}, establishing energy decay rates for \cref{eq:IBVP} through high-frequency
asymptotics for $P(\i \lambda)^{-1}$ amounts to analyzing the limit behavior of so-called \emph{quasimodes} at certain semiclassical scales.
Here we give a definition of quasimodes for \cref{eq:IBVP} inspired by that of \cite{AnaLea14}.
\begin{defi}[Quasimodes]\label{def:qm} Let $\delta \geq 0$.
A sequence $\{(h_n, u_n)\}_{n \in \N}$ of positive reals $h_n$ with $h_n \to 0$, $n \to + \infty$, and elements $u_n \in \dom(A)$ such that
\begin{subequations}
\begin{align}
&\|u_n\|_H = 1, &&n \in \N, \\
&\|(h_n^2 A + \i h_n BB^\ast -1)u_n\|_H = o(h_n^{1 + \delta}), &&n \to + \infty,
\end{align}
\end{subequations}
is called a sequence of (normalized) $o(h_n^{1+\delta})$-\emph{quasimodes}.
\end{defi}

% % We shall use a semiclassical formulation of quasimodes. 
%  In this section we consider elements $u_n \in \dom(A)$ such that

In the sequel, subscripts indicating the dependence on $n$ and $h_n$ are omitted: with a slight abuse in notation we will speak of $o(h^{1+\delta})$-quasimodes $u$ and consider the limit $h \to 0$. Furthermore, $\delta$ will always denote a nonnegative parameter and, when considering quasimodes $u$, we will write
\begin{equation}
(f, e) \triangleq (h^2 A + \i h BB^\ast - 1)u \in H.
\end{equation}
With that in mind, we see that
% Following the terminology of \cite{AnaLea14}, we will speak of  sequences of (normalized) \emph{$o(h^3)$-quasimodes}. Having let $(f, e) \triangleq (h^2 A + \i h BB^\ast + 1)u \in H$, we see that 
$o(h^{1+\delta})$-quasimodes $u$ solve the stationary boundary value problem
\begin{subequations}
\label{eq:quasim-PDE}
\begin{align}
&(-h^2\Delta - 1)u = f&&\mbox{in}~\Omega, \\
\label{eq:quasim-gamma}
&(-h^2 \Delta_\Gamma + \i h - 1)u = - h^2 \partial_\nu u + e &&\mbox{on}~\Gamma_0, \\
&u = 0&&\mbox{on}~\Gamma_1,
\end{align}
\end{subequations}
with
\begin{equation}
\label{eq:base-qm}
\|u\|^2_{L^2(\Omega)} + \|u\|_{L^2(\Gamma_0)}^2 = 1, \quad 
\|f\|_{L^2(\Omega)} = o(h^{1+\delta}), \quad \|e\|_{L^2(\Gamma_0)} = o(h^{1+\delta}), \quad h \to 0.
\end{equation}
% \edn{maybe a definition would be better...}
(Compare \cref{eq:quasim-PDE} with \cref{eq:im-mode-sys}, which we studied to rule out purely imaginary modes for \cref{eq:IBVP}.)
\begin{rem}
Since we aim at establishing a non-uniform energy decay rate $o(t^{-1})$, it suffices to study $o(h^3)$-quasimodes, which ultimately we prove \emph{do not} exist under \cref{as:dyn}.
% Nevertheless, since our proof is anyway broken down into several lemmas and auxiliary results,
When possible, we nevertheless consider more general $o(h^{1+\delta})$-quasimodes in an attempt to better highlight the obstruction to uniform stabilization  and the role of the geometric control condition. In this regard, we point out that uniform (exponential) energy decay would correspond to the case $\delta = 0$.
\end{rem}

% \begin{rem}

% Here, the semiclassical (small) parameter $h$ corresponds to the inverse of a (large) frequency $\lambda$: the behavior of quasimodes in the limit $h \to 0$ is related to high-frequency asymptotics for the quadratic operator pencil $P$. In our analysis, quasimodes are a proof device; they need not exist, and ultimately the purpose of this section is to prove that {there are no} $o(h^3)$-quasimodes if $\Gamma_0$ satisfies the geometric control condition.
% % $o(h^3)$-quasimodes.
% \end{rem}

% \begin{rem}
% With a slight but convenient abuse in notation, \edn{notation $h$}...
% \end{rem}

\subsubsection{Preliminary estimates}

Our first estimates are  immediate consequences of the abstract dissipative second-order structure \cref{eq:IBVP-abstract}.
\begin{lemma}
\label{lem:V-dissip-h}
For $o(h^{1+\delta})$-quasimodes $u$,
\begin{equation}
\label{eq:V-dissip-h}
 h \|u\|_V = O(1), \quad  \|u\|_{L^2(\Gamma_0)} = o(h^{\delta/2}), \quad h \to 0.
\end{equation}
\end{lemma}
\begin{proof}
Since $(h^2 A + \i h BB^\ast - 1)u = (f, e)$, taking the scalar product with $u$ in $H$ yields
\begin{equation}
\label{eq:pairing-id}
h^2 \langle Au, u\rangle_{V^\ast,V} + \i h \|u\|^2_{L^2(\Gamma_0)} - \|u\|^2_H = \langle (f,e), u\rangle_H.
\end{equation}
Using the real part of \cref{eq:pairing-id} and the Cauchy--Schwarz inequality leads to
\begin{equation}
\label{eq:pairing-id-real}
h^2\langle Au, u\rangle_{V^\ast,V} = 1 + \re \langle (f,e), u\rangle_H \leq 1 + \|(f, e)\|_H = O(1), \quad h \to 0.
\end{equation}
Since $\langle A\cdot, \cdot\rangle_{V^\ast, V} + \|\cdot\|_H$ is equivalent to the squared $V$-norm, \cref{eq:pairing-id-real} implies the first part of \cref{eq:V-dissip-h}. On the other hand, by taking the imaginary part of \cref{eq:pairing-id} we see that
\begin{equation}
h \|u\|^2_{L^2(\Gamma_0)} = \im \langle (f,e), u\rangle_H \leq \|(f, e)\|_H = o(h^{1+\delta}), \quad h \to 0,
\end{equation}
which completes the proof of \cref{eq:V-dissip-h}.
\end{proof}

Next, we give a differential multiplier identity that will be used at different points of our analysis; see, e.g., \cite[Lemma 3.7]{Spe14} for a proof.
\begin{lemma}[Rellich identity]
\label{lem:rellich}
For all $h > 0$, for all $u \in H^1(\Omega)$ with $\Delta u \in L^2(\Omega)$ $\partial_\nu u \in L^2(\Gamma)$ and $u \in H^1(\Gamma)$, and for all real-valued vector fields $q = (q_1, \dots, q_d) \in \C^1(\overline{\Omega})^d$,
\begin{multline}
\label{eq:rellich}
\int_\Omega  2 \re \mleft ( \overline{q \cdot \nabla u} (h^2 \Delta u +  u) \mright )  - (\dive q) (h^2|\nabla u|^2 - |u|^2) + 2 h^2 \re \sum_{i=1}^d \sum_{j=1}^d \dl{q_j}{x_i} \dl{u}{x_i} \overline{\dl{u}{x_j}} \, \d x \\ =  \int_{\Gamma} 2 h^2  \re \left( \overline{q \cdot \nabla u} \dl{u}{\nu} \right) + (|u|^2 - h^2 |\nabla u|^2)(q \cdot \nu) \, \d \sigma.
\end{multline}
\end{lemma}
% \begin{rem}
% \Cref{eq:rellich}
% \end{rem}

With \cref{lem:rellich} we can give an \emph{a priori} estimate of the Neumann trace of quasimodes on $\Gamma_0$. 

\begin{lemma}\label{lem:a-priori-neumann} For $o(h^{1+\delta})$-quasimodes $u$,
\begin{equation}
\label{eq:a-priori-neumann}
h \|\partial_\nu u\|_{L^2(\Gamma_0)} = O(1), \quad h \to 0.
\end{equation}
\end{lemma}
\begin{proof}
Since $\Omega$ is smooth and $\overline{\Gamma_0} \cap \overline{\Gamma_1} = \emptyset$, by \cite[Lemma 2.1]{Kom94book} we may find a smooth vector field $q : \overline{\Omega} \to \R^d$ such that $q = \nu$ on $\Gamma_0$ and $q = 0$ on $\Gamma_1$. With this choice of $q$, the Rellich identity from \cref{lem:rellich} applied to $o(h^{1+\delta})$-quasimodes $u$ yields
\begin{multline}
\label{eq:rellich-trace}
 \int_\Omega - 2 \re \left ( \overline{q \cdot \nabla u} f \right )  - (\dive q) (h^2|\nabla u|^2 - |u|^2) + 2 h^2 \re \sum_{i=1}^d \sum_{j=1}^d \dl{q_j}{x_i} \dl{u}{x_i} \overline{\dl{u}{x_j}} \, \d x \\ =  \int_{\Gamma_0}  h^2   \left| \dl{u}{\nu} \right|^2 + (|u|^2 - h^2 |\nabla_\Gamma u|^2) \, \d \sigma,
\end{multline}
where we used that $|\nabla u|^2 = |\partial_\nu u|^2 + |\nabla_\Gamma u|^2$ on $\Gamma_0$. From \cref{eq:rellich-trace} it is easy to deduce that
\begin{equation}
h^2 \|\partial_\nu u\|^2_{L^2(\Gamma_0)} = O\mleft(\|u\|_{H^1(\Omega)}\|f\|_{L^2(\Omega)} + h^2 \|u\|_{H^1(\Omega)}^2 + \|u\|^2_{L^2(\Gamma_0)} + h^2 \|u\|^2_{H^1(\Gamma_0)} \mright), \quad h \to 0,
\end{equation}
which in turn leads to the desired estimate \cref{eq:a-priori-neumann} due to  \cref{eq:base-qm,eq:V-dissip-h}.
\end{proof}
We can now estimate tangential derivatives of quasimodes on $\Gamma_0$.

\begin{lemma}
\label{lem:quasim-H1-gamma} For $o(h^{1+\delta})$-quasimodes $u$,
\begin{equation}
h \|\nabla_\Gamma u\|_{L^2(\Gamma_0)^{d-1}} = o(h^{\delta/2} + h^{(2 + \delta)/4}), \quad h \to 0.
\end{equation}
\end{lemma}
\begin{proof}
Multiply \cref{eq:quasim-gamma} by $\overline{u}$ and integrate by parts over $\Gamma_0$ to see that, after taking  real parts,
\begin{equation}
\label{eq:boundary-pairing}
h^2 \int_{\Gamma_0} |\nabla_\Gamma u|^2 \, \d \sigma = \int_{\Gamma_0} |u|^2 \, \d \sigma - \re \int_{\Gamma_0}  h^2 \partial_\nu u \overline{u} - e\overline{u} \, \d \sigma. 
\end{equation}
Apply \cref{lem:V-dissip-h,lem:a-priori-neumann} to complete the proof.
% \begin{equation}
% h^2 \int_{\Gamma_0} |\nabla_\Gamma u|^2 \, \d \sigma = o(h^\delta) + o(h^{3\delta/2}
% \end{equation}

\end{proof}

\subsubsection{Microlocalized estimates near the boundary}

\label{sec:microlocal}

The purpose of this section is to investigate the behavior of quasimodes $u$ near $\Gamma_0$ and, specifically for $o(h^3)$-quasimodes, improve the normal derivative estimate of \cref{lem:a-priori-neumann}. To do so, we will carry out a more involved analysis in two separate microlocal regimes, namely, at low and high \emph{tangential} frequencies.

% \edn{[ici ou avant...]}Since $\Gamma$ is compact, there exist an open covering of $\Gamma$ made of such charts $C_n = (O_n, \kappa_n)$, $n = 1, \dots, N$. We use a partition of unity: by, e.g., \cite[Lemma 9.3]{Bre11book} there exist smooth functions $\theta_n : \R^{d} \to [0, 1]$, $n = 0, \dots, N$ such that:
% \begin{enumerate}
%     \item 
% \begin{equation}
% \sum_{n=0}^N\theta_n(x) = 1, \quad x \in \R^d;
% \end{equation}
%     \item For $n = 1, \dots, N$, $\theta_n|_{O_n} \in \D(O_n)$;
%     \item The support of $\theta_0$ is contained in $\R^{d} \setminus \Gamma$ and $\theta_0|_{\Omega} \in \D(\Omega)$.
% \end{enumerate}
% Recall that we use the identification $T_x \R^{d} \simeq \R^d$. For $n = 1, \dots, N$ we may define the lifted diffeomorphism $\tilde{\kappa}_n : O_N \times \R^d \to \kappa_n(O_n) \times \R^d$ as follows:
% \begin{equation}
% \tilde{\kappa}_n(x, \eta) \triangleq (\kappa_n(x), (\nabla \kappa_n(x)^\ast)^{-1} \eta), \quad x \in O_k, \quad \eta \in \R^{d}.
% \end{equation}
% Also, for each $n = 0, \dots, N$ we let $\theta'_n$ be a function in $\D(O_n)$ such that $\theta'_n = 1$ near the support of $\theta_n$.
We shall use the local setting of \cref{sec:normal-geodesic}.
First, let us recall and introduce some notation. Let $\R_+^{d} \triangleq \R^{d-1} \times (0, +\infty)$ be the open half-space made of points $y = (y_1,\dots,y_{d-1}, y_d) = (y_\tau, y_d) \in \R^{d}$ with $y_d > 0$. We will refer to $y_\tau$ and $y_d$ as the tangential and normal components of $y \in \R^{d}$, respectively. Similarly, $\xi_\tau = (\xi_1, \dots,\xi_{d-1}) \in \R^{d-1}$ indicates the tangential Fourier variable. We recall that
\begin{equation}
D_{y_i} = - \i \partial_{y_i}, \quad i  =1, \dots, d-1.
\end{equation}
In the subsequent analysis of quasimodes, it will be convenient to use semiclassical norms for $H^1$-spaces. If $X$ denotes $\Omega$, $\R^d_+$ or $\R^{d-1}$, then for $h > 0$ we let $H^1_h(X)$ be the space $H^1(X)$ equipped with the $h$-dependent norm
\begin{equation}
\label{eq:sc-norm}
\|w\|^2_{H^1_h(X)} \triangleq \int_X h^2 |\nabla w|^2 + |w|^2 \, \d x, \quad w \in H^1(X),
\end{equation}
where $\nabla$ and $\d x$ are the ambient gradient and Lebesgue measure. Similarly, if $M$ is $\Gamma$ or $\Gamma_0$, $H^1_h(M)$ is $H^1(M)$ with the norm
\begin{equation}
\|w\|_{H^1_h(M)}^2 \triangleq \int_M h^2 |\nabla_\Gamma w|^2 + |w|^2 \, \d \sigma, \quad w \in H^1(M).
\end{equation}
Now, given a  function $w$ of $\Omega$, we will write
\begin{equation}
w^{C_n} \triangleq (\kappa_n^{-1})^\ast w = w \circ \kappa_n^{-1}, \quad \tilde{w}^{C_n} \triangleq (\kappa_n^{-1})^\ast (\theta_nw), \quad n = 0, \dots, N,
\end{equation}
so that $w^{C_n}$, defined in $\kappa_n(O_n \cap \Omega)$, is the pullback of $w$ by $\kappa_n$, whereas $\tilde{w}^{C_n}$ is the pullback of its localized version $\theta_n w$ and trivially extends to a function of $\R^{d}_+$ if $n = 1, \dots N$ or of $\R^{d}$ if $n = 0$. On the other hand, if $n = 1, \dots, N$ and $w$ is instead a function of $\Gamma_0$, we use the same notation, so that the localized pullback $\tilde{w}^{C_n}$ becomes a function of $\R^{d-1}$.  In what follows we will not write the dependence on $n = 1, \dots, N$, and now $C=(O, \kappa)$ is \emph{any} of the $N$ local charts  defined in \cref{sec:normal-geodesic}.

% It will also be convenient to use semiclassical scaling

% and $\theta, \theta'$ are the corresponding cutoff function.

% \edn{[à retravailler...]} \edn{écarter $n =0$} 
% Consider $o(h^{1+\delta})$-quasimodes $u$. 

% For $n = 0, \dots, N$, we let 
% \begin{equation}
% u^{C_n} \triangleq (\kappa_n^{-1})^\ast u = u \circ \kappa^{-1}, \quad \tilde{u}^{C_n} \triangleq \theta_n u^{C_n}.
% \end{equation}

Consider now $o(h^{1+\delta})$-quasimodes $u$. Then, localizing \cref{eq:quasim-PDE} via multiplication by $\theta$ and  taking pullbacks by $\kappa$, we see that \smash{$\tilde{u}^C$} solves 
% \edp{Ici on a déjà $D_y=-i \partial_{y}$}
% : e.g., $\tilde{u}^C = \tilde{u}^C_n$. Then, $\tilde{u}^C$ solves
\begin{subequations}
\label{eq:utilde}
\begin{align}
&(-h^{2}\Delta^{C}- 1 ) \Tilde{u}^C = \tilde{f}^C-[\theta,h^{2}\Delta]^Cu^{C}&& \mbox{in}~ \kappa(O\cap \Omega),\\
&(-h^{2}\Delta_{\Gamma}^{C}+ \i h - 1 ) \Tilde{u}^C = -\i h^2 \theta^C D_{y_d}u^C + \tilde{e}^C -[\theta,h^{2}\Delta_{\Gamma}]^Cu^{C}&& \mbox{in}~\kappa(O \cap \Gamma_0).
\end{align}
\end{subequations}
Here and in the sequel brackets $[X, Y] = XY -YX$ denote {commutators}. Since $\tilde{u}^C$ and all the terms at the right-hand sides of \cref{eq:utilde} vanish outside of $\kappa(\supp(\theta))$, we may rewrite \cref{eq:utilde} as 
\begin{subequations}
\label{eq:utilde-tilde}
\begin{align}
\label{eq:utilde-domain}
&(-h^{2}\tilde{\Delta}^{C}- 1 ) \Tilde{u}^C = \tilde{f}^C-[\theta^C,h^{2}\tilde{\Delta}^C]u^{C}&& \mbox{in}~ \R^{d}_+,\\
\label{eq:utilde-boundary}
&(-h^{2}\tilde{\Delta}_{\Gamma}^{C}+ \i h - 1 ) \Tilde{u}^C = -\i h^2 \theta^C D_{y_d}u^C + \tilde{e}^C -[\theta^C,h^{2}\tilde{\Delta}_{\Gamma}^C]u^{C}&& \mbox{in}~\R^{d-1},
\end{align}
\end{subequations}
where we recall from \cref{sec:normal-geodesic} that $\tilde{\Delta}^C$ and $\tilde{\Delta}^C_{\Gamma}$ coincide with their ``tilteless'' counterparts near $\kappa(\supp(\theta))$ but have tangential coefficients that vanish away from $\kappa(O)$. By the way, we may pick $\theta'' \in \D(O)$ taking values in $[0, 1]$ such that $\supp(\theta'')$ is contained in an open subset of $ \{ x \in O : \theta'(x) = 1 \}$ and $\theta'' = 1$ near $\supp(\theta)$.\footnote{Here the prime symbols {do not} denote derivatives.} 

Next, choose a smooth and (strictly) increasing cutoff $\rho : [0, +\infty) \to [0, 1]$ such that $\rho(r) = 0$ for $r^2 \leq 4/c$ and $\rho(r) = 1$ for (say) $r^2 \geq 8/c $, where $c > 0$ is the ellipticity constant from \cref{eq:elliptic}.
Define a ($h$-independent) symbol $\chi$ as follows: 
\begin{equation}
\label{eq:chi-rho}
\chi(y, \xi_\tau) \triangleq \theta''(\kappa^{-1}(y)) \rho(|\xi_\tau|), \quad y \in \R^d_+, \quad \xi_\tau \in \R^{d-1}.
\end{equation}
% \begin{equation}
% \rho(r) = 1~\mbox{if}~r^2 \leq \sqrt{\frac{2}{c}} , \quad \rho = 
% % 0~\mbox{near}\left[\sqrt{\frac{2}{c}}, + \infty \right)
% \end{equation}
% that $\rho = 1$ near $0$ and $\rho = 0$ near $$
Then, $\chi$ belongs to the class of  tangential symbols $S^{0}_{\tau,h}$ defined in \cref{sec:tangential}. Note that $\chi$ has compact spatial support. Define a zeroth-order tangential semiclassical pseudodifferential operator $\Chi \in \Psi^0_{\tau,h}$ by
 % We quantify $dz$Let $\Chi \in \edn{\Psi^0_h}$ be 
% the tangential semiclassical pseudodifferential operator 
\begin{equation}
\Chi \triangleq \Opht (\chi),
\end{equation}
where the space $\Psi^0_{\tau,h}$ and
 the quantization procedure $\Opht$ are described in \cref{sec:tangential}, which also contains some additional notation and facts concerning pseudodifferential operators on 
 $\R^{d-1}$. 
 % ; in particular $\Chi$ belongs
% (We will alleviate notation in subsequent proofs.)

% In what follows, it will be convenient to use semiclassical norms for $H^1$-spaces. Given
The first result of this section is an estimate for quasimodes in the regime where tangential derivatives are dominant.

\begin{prop}[High-frequency regime] \label{prop:HF} For $o(h^{1+\delta})$-quasimodes $u$, letting $v \triangleq \Chi\tilde{u}^C$, we have 
\begin{equation}
\|v\|_{H_h^1(\R^d_+)}   + h\|D_{y_d}v|_{y_d = 0}\|_{L^2(\R^{d-1})} = o(1), \quad h \to 0.
\end{equation}
%\edp{\[\|v\|_{H_h^1(\R^d_+)}=O(h^{1/2}), \\  h\|D_{y_d}v|_{y_d = 0}\|_{L^2(\R^{d-1})} = O(h^{1/4}), \quad h \to 0.\]}
\end{prop}
\begin{proof}
% Here we will omit the superscript ``$C$''  indicating pullbacks relative to the local chart $C = (O, \kappa)$ under consideration and will only work in $y$-coordinates.
It follows from \cref{eq:utilde-domain} that $v$ solves
\begin{equation}
\label{eq:v-equation}
(-h^2 \tilde{\Delta}^C -1) v = \Chi \tilde{f}^C - \Chi[\theta^C, h^2 \tilde{\Delta}^C] u^C - [\Chi, h^2 \tilde{\Delta}^C]\tilde{u}^C \quad\mbox{in}~\R^d_+.
\end{equation}
Recall from \cref{sec:normal-geodesic} that $-h^2\tilde{\Delta}^C$ has the form
\begin{equation}
\label{eq:pullback-h-delta}
-h^{2}\tilde{\Delta}^C = h^2 D^2_{y_d} +  h^2 \sum_{i,j = 1}^{d-1}\tilde{a}_{ij}(y) D_{y_i}D_{y_j}.
\end{equation}
% where the modified coefficients 
% % $a_{ij} : \R^{d-1} \times (0, +\infty) \to \Co$ are smooth, \edn{vanish away from $\kappa(O)$}, and
% $\tilde{a}_{ij}$ satisfy the ellipticity condition \cref{eq:elliptic} restricted (in space) to $\supp(\theta'')$ since they coincide with the original coefficients $a_{ij}$ there.
    % \begin{equation}
    % \label{eq:elliptic1}
    % \sum_{i,j = 1}^{d-1} a_{ij}(y) \xi_i \xi_j \geq c |\xi_\tau|^2, \quad \xi_\tau = (\xi_1, \dots, \xi_{d-1}) \in \R^{d-1}, \quad y \in \kappa(O),
    % \end{equation}
    % for some positive constant $c$. In view of \cref{eq:pullback-h-delta,eq:elliptic1}, 
    % Now, \cref{eq:pullback-h-delta} shows that we my decompose  $-h^2\Delta$ into a normal and a tangential part $P_\tau = P_\tau(h)$ by letting
    We denote by $P_\tau$ the second term at the right-hand side of \cref{eq:pullback-h-delta}, that is, $P_\tau = P_\tau(h)$ is the tangential part of $-h^2 \tilde{\Delta}^C$, so that
    % \begin{equation}
    % P_\tau \triangleq h^2 \sum_{i,j = 1}^{d-1}a_{ij}(y) D_{y_i}D_{y_j}, \quad h > 0,
    % \end{equation}
    % so that
    $-h^2\tilde{\Delta}^C = h^2 D_{y_d}^2 + P_\tau$. We also denote by $r$ the right-hand side of \cref{eq:v-equation}.
    % ,  which is made of commutator terms and, as we will see later, is a lower-order rest.
    \Cref{eq:v-equation} in this notation reads as follows:
    \begin{equation}
    \label{eq:v-equation-bis}
    (h^2D^{2}_{y_d} + P_\tau - 1)v = r \quad \mbox{in}~\R^{d}_+.
    \end{equation}
    The proof is divided into several steps. 
    
    \emph{Step 1: elliptic estimates.}
    Multiplying \cref{eq:v-equation-bis} by $\overline{v}$ and integrating over $\R_+^{d}$ lead to
    \begin{equation}
    \label{eq:scalar-v}
    \iint_{\R^d_+} h^2 D^2_{y_d} v \overline{v} \, \d y_d \, \d y_\tau + \iint_{\R_+^d} (P_\tau - 1)v \overline{v} \, \d y_d \, \d y_\tau = \int_{\R^{d-1}} \int_0^{+\infty} r \overline{v} \, \d y_d \, \d y_\tau.
    \end{equation}
    We start with the first term in \cref{eq:scalar-v}: integrating by part we  obtain
    \begin{equation}
        \label{eq:ipp}
        \iint_{\R^d_+} h^2 D^2_{y_d} v \overline{v} \, \d y_d \, \d y_\tau =  h^2 \int_{\R^d_+} |D_{y_d}v|^2 \, \d y_d \, \d y_\tau - h^2 \int_{\R^{d-1}} D_{y_d}v \overline{v} |_{y_d = 0} \, \d y_\tau,
    \end{equation}
    where we used that $v$ vanish for large $y_d$ since $\tilde{u}^C$ is compactly supported and $\Chi$ is nonlocal only in the tangential variable.
    % \edn{(justify that $v$ vanish for large $y_d$, because $\Chi$ only acts in the tangential direction...)} 
    To estimate  the second term in \cref{eq:scalar-v}, 
    % by Fubini's theorem we have
    % \begin{equation}
    % \label{eq:inv-int}
    % \int_{\R^{d-1}} \int_0^{+\infty} (P_\tau - 1)v \overline{v} \, \d y_d \, \d y_\tau = \int_0^{+\infty} \int_{\R^{d-1}} (P_\tau - 1)v \overline{v} \, \d y_\tau \, \d y_d
    % \end{equation}
    we shall now employ an argument based on G\aa rding's inequality.
    % to estimate the inner integral at the right-hand side uniformly with respect to the normal variable $y_d$. 
    The (semiclassical) principal symbol of the tangential differential operator $P_\tau -1$ is given by
    \begin{equation}
    \sigma_h(P_\tau - 1)(y, \xi_\tau) = \sum_{i,j=1}^{d-1}\tilde{a}_{ij}(y) \xi_i \xi_j -1, \quad y \in \R^d_+, \quad \xi_{\tau} = (\xi_1, \dots, \xi_{d-1}) \in \R^{d-1}.
    \end{equation}
    We will now use the properties the position and frequency cutoffs $\theta''$ and $\rho$ involved in the definition \cref{eq:chi-rho} of our microlocalization symbol $\chi$. On the one hand, in some open neighborhood $N \subset \kappa(O) \cap \overline{\R_+^d}$ of $\supp(\theta'' \circ \kappa^{-1})$ the modified coefficients $\tilde{a}_{ij}$ coincide with the original ones $a_{ij}$ and, as such, satisfy the ellipticity condition \cref{eq:elliptic} for all $\xi_\tau \in \R^{d-1}$. On the other hand, $\rho(|\xi_\tau|) > 0$ if and only if
    $|\xi_\tau|^2 > 4/c$. Let $U \triangleq N \times \{ \xi_\tau \in \R^{d-1} : |\xi_\tau|^2 > 2/c \}$. Then $\supp(\chi) \subset U$ and
    \begin{equation}
    \sigma_h(P_\tau - 1)(y, \xi_\tau) =  \sum_{i,j=1}^{d-1}{a}_{ij}(y) \xi_i \xi_j -1 \geq c|\xi_\tau|^2 - 1 \geq \frac{c}{2}|\xi_{\tau}|^2, \quad (y, \xi_\tau) \in U.
    \end{equation}
    As a result, the microlocal tangential G\aa rding inequality from \cref{th:garding}, which is adapted from \cite[Theorem 2.50]{RouLeb22a}, provides positive constants $K$ and $K'$ such that
    \begin{multline}
    \label{eq:appli-garding}
    \iint_{\R^d_+} (P_\tau-1)\Chi \varphi \overline{\Chi \varphi} \, \d y_d \, \d y_\tau \geq K\mleft 
    ( \iint_{\R^d_+} |\Chi \varphi|^2 + h^2  \sum_{i=1}^{d-1} |D_{y_i}\Chi \varphi|^2 \, \d y_d \, \d y_\tau
    % (h^2 \|\Chi \varphi\|^2_{L^2(0, +\infty; H^1(\R^{d-1}))} + \|\Chi \varphi\|^2_{L^2(\R^d_+)}    
    \mright) \\ - {h^2} K' 
    % \|\varphi\|^2_{L^2(\R^d_+)}
    \iint_{\R^d_+} |\varphi|^2 \, \d y_d \, \d y_\tau
    \end{multline}
    for all functions $\varphi$ in, say, $\overline{\S}(\R^d_+)$,\footnote{Following \cite{RouLeb22a}, $\overline{\S}(\R^d_+)$ denotes
    the space of restrictions to $\overline{\R_+^d}$ of Schwartz functions on $\R^d$.}
    and all positive (small enough) $h$. Applying\footnote{That is, after a suitable density argument, which we omit.} \cref{eq:appli-garding} to $o(h^{1+\delta})$-quasimodes $u$ (recall that $v = \Chi \tilde{u}^C$), and plugging the result into \cref{eq:scalar-v,eq:ipp}, we obtain,
    % \footnote{Here we also use that $H^1(\R^+_d) = L^2\hl{H^1(\R^{d-1})} \cap H^1\hl{L^2(\R^{d-1})}$.}
    for some constant $K > 0$ independent of quasimodes,
    \begin{equation}
    \label{eq:bilan-ell}
    \|v\|^2_{H_h^1(\R^d_+)}
    \leq K \mleft (h^2 \|\tilde{u}^C\|^2_{L^2(\R^d_+)} + |\langle r, v \rangle_{L^2(\R^d_+)}|  + h^2|\langle D_{y_d} v , v\rangle_{L^2(\R^{d-1})}| \mright ).
    \end{equation}
    Here and in the sequel, for the ease of notation we might omit explicit indications that boundary traces are taken at $y_d =0$ when writing norms or scalar products on $\R^{d-1}$.
    The remainder of the proof consists in estimating the terms at the right-hand side of \cref{eq:bilan-ell}.

    \emph{Step 2: normal derivative.} Here we carry out a simple differential multiplier argument, not unrelated to that of \cref{lem:a-priori-neumann}. It is therefore convenient to put $P_\tau$ in divergence form and write
    \begin{equation}
    P_\tau = h^2\sum_{i, j = 1}^{d-1} D_{y_i} (\tilde{a}_{ij}(y) D_{y_j}) + R_\tau,
    \end{equation}
    with $R_\tau \in h \diff_{\tau,h}^1$. Pick a smooth cutoff $\psi : [0, +\infty) \to [0, 1]$ such that $\psi = 1$ near $0$ and $\psi = 0$ away from zero. We shall multiply \cref{eq:v-equation-bis} by $\psi(y_d) \overline{D_{y_d} v}$ and integrate over $\R^d_+$.
    First, note that
    \begin{equation}
    \label{eq:ipp-D-nor}
    % \begin{aligned}
    h^2 \im \iint_{\R^d_+} \psi D_{y_d}^2 v \overline{D_{y_d} v} \, \d y_d \, \d y_\tau 
    % & = \frac{1}{2} \im \iint_{\R^d_+} \psi D_{y_d}  |D_{y_d} v|^2 \, \d y_\tau \, \d y_d \\ 
     = \frac{h^2}{2} \mleft. \int_{\R^{d-1}} |D_{y_d} v|^2  \, \d y_\tau \mright|_{y_d = 0} +  \frac{h^2}{2}  \iint_{\R^d_+} \psi' |D_{y_d} v|^2  \, \d y_d \, \d y_{\tau},
    % \end{aligned}
    \end{equation}
    where $\psi' = (\d /\d y_d) \psi$. Also,
    \begin{multline}
    \label{eq:ipp-D-tan}
     h^2\im \iint_{\R^d_+} \sum_{i,j=1}^{d-1} D_{y_i}(\tilde{a}_{ij}(y) D_{y_j}v) \psi(y_d) \overline{D_{y_d}v} \, \d y_d \, \d y_\tau  
     % - h^2 \im  \iint  \sum_{i,j=1}^{d-1} \psi(y_d) \tilde{a}_{ij}(y) D_{y_j} v \overline{D^2_{y_i y_d} v}   \\ =  - \frac{h^2}{2} \iint  \sum_{i,j=1}^{d-1} \psi(y_d)  \tilde{a}_{ij}(y) D_{y_d} (D_{y_j} v \overline{D_{y_i}v})   
     = - \mleft. \frac{h^2}{2} \int_{\R^{d-1}} \sum_{i,j = 1}^{d-1}\tilde{a}_{ij}(y) D_{y_j} v \overline{D_{y_i} v}  \, \d y_\tau \mright|_{y_d = 0} \\ - \frac{h^2}{2} \iint_{\R^{d}_+} \sum_{i,j = 1}^{d-1} D_{y_d}(\psi(y_d) \tilde{a}_{ij}(y)) D_{y_j} v \overline{D_{y_i} v} \, \d y_d \, \d y_\tau.
    \end{multline}
 % Multiply \cref{eq:v-equation-bis} byand take the real part to obtain
 %    % \edn{[Divergence form?]}
 %    \begin{equation}
 %    -\frac{h^2}{2} \iint_{\R^{d}_+} \psi D_{y_d} |D_{y_d}v|^2 \, \d y_d \, \d y_\tau 
 %    \end{equation}
 With \cref{eq:ipp-D-nor,eq:ipp-D-tan} we readily deduce that 
 % there exists $K' > 0$  such that
 \begin{multline}
 \label{eq:normal-r-r}
h^2\|D_{y_d}v\|_{L^2(\R^{d-1})}^2 \\ \leq K' \mleft( \|r\|_{L^2(\R^d_+)} \|D_{y_d} v\|_{L^2(\R^d_+)} + \|(R_\tau -1) v\|_{L^2(\R^d_+)}\|D_{y_d} v\|_{L^2(\R^d_+)} + \|v\| _{H^1_h(\R^d_+)}^2 + \|v\|^2_{H_h^1(\R^{d-1})} \mright).
 \end{multline}
 for some constant positive constant $K'$ (again, independent of quasimodes).
% \edp{Step 3?}

\emph{Step 3: tangential derivatives.} Recall that, for $\varphi \in \overline{\S}(\R^d_+)$, $\Opht(\chi)\varphi|_{y_d = 0} = \Oph(\chi|_{y_d = 0})(\varphi|_{y_d = 0})$, where $\Oph$ indicates  standard semiclassical quantization on $\R^{d-1}$. Note also that since $\chi$ belongs to the class of tangential symbols \smash{$S^0_{\tau,h} = S^0_{\tau, h}(\overline{\R^d_+} \times \R^{d-1})$} then its boundary trace $\chi|_{y_d = 0}$ is in $S^0_h(\R^{d-1} \times \R^{d-1})$; see \cite[Remark 2.40]{RouLeb22a}.
% This justifies using the same symbol $\Chi$ for the operators $\Opht(\chi) \in \Psi_{\tau, h}^0 = \PsiÔ_{\tau,h}(\R^d_+)$ and \smash{$\Oph(\chi|_{y_d = 0}) \in \Psi^0_h(\R^{d-1})$}.
In particular $\Chi$ is bounded uniformly in $h$ on the spaces $L^2(\R^{d-1})$ and $H^1_h(\R^{d-1})$; see, e.g., \cite[Proposition E.19]{DyaZwo19book}.
With this in mind, since $v = \Chi \tilde{u}^C$ we obtain
\begin{equation}
\label{eq:v-H1-tang}
% \begin{aligned}
\|v\|_{H_h^1(\R^{d-1})} 
% = h \|\Chi(\tilde{u}^C|_{y_d = 0})\|_{H^1(\R^{d-1})} 
= O(\|\tilde{u}^C\|_{H_h^1(\R^{d-1})}) 
= O(\|u\|_{H_h^1(\Gamma_0)}) = o(1), \quad h \to 0,
% \end{aligned}
\end{equation}
% by applying \cite[Proposition E.19]{DyaZwo19book} and
where we also used
\cref{lem:V-dissip-h,lem:quasim-H1-gamma}.

    \emph{Step 4: commutators and lower-order terms.} We start by stating further boundedness properties of $\Chi$. Since its symbol $\chi$ is in the tangential class \smash{$S^0_{\tau,h}$, $\Chi$} is bounded uniformly in $h$ on the spaces  $L^2(\R^d_+)$  and also $L^2(0, +\infty; H^1_h(\R^{d-1}))$ by \cite[Theorem 2.47]{RouLeb22a}. Furthermore $[\Chi, hD_{y_d}] = h\Opht(D_{y_d}\chi) \in h\Psi_{\tau,h}^0(\R^d_+)$ (see \cite[Remark 2.46]{RouLeb22a}), and thus $\Chi$ is bounded on $H^1_h(\R^d_+)$ as well. We are now ready to deal with the rest $r$ from \cref{eq:v-equation,eq:v-equation-bis} and also the term involving $R_\tau$ in \cref{eq:normal-r-r}.
    By construction $R_\tau \in h \diff_{\tau,h}^1$ and
\begin{equation}
\label{eq:R-tau}
\|(R_\tau - 1) v\|_{L^2(\R^d_+)} \|D_{y_d}v\|_{L^2(\R^d_+)} = O( \|v\|^2_{H_h^1(\R^d_+)}), \quad h \to 0.
\end{equation}% Now, recall from \cref{eq:v-equation,eq:v-equation-bis} the form of the rest $r$.
Also,
% since \smash{$\chi \in S^0_{\tau,h}$, $\Chi$} is bounded on $L^2(\R^d_+)$ uniformly in $h$ by \cite[Theorem 2.47]{RouLeb22a}, and in particular
\begin{equation}
\|\Chi \tilde{f}^C\|_{L^2(\R^d_+)} = O(\|\tilde{f}^C\|_{L^2(\R^d_+)}) = O(\|f\|_{L^2(\Omega)}) = o(h^{1+\delta}), \quad h \to 0.
\end{equation}
% In fact, $\Chi$ has further boundedness properties: \cite[Theorem 2.47]{RouLeb22a} shows that  there exists $C > 0$ such that
% \begin{equation}
% h \|\Chi \varphi\|_{L^2(0, +\infty; H^1(\R^{d-1})} \leq C\mleft( h \| \varphi\|_{L^2(0, +\infty; H^1(\R^{d-1})} + \|\varphi\|_{L^2(\R^d_+)}\mright), \quad \varphi \in \overline{\S}(\R^d_+), \quad 0 < h \leq 1,
% \end{equation}
% and since $[\Chi, D_{y_d}] = \Opht(D_{y_d}\chi) \in \Psi_{\tau,h}^0(\R^d_+)$ (see \cite[Remark 2.46]{RouLeb22a}) we also have
% \begin{equation}
% h\|D_{y_d} \Chi \varphi\|_{L^2(\R^d_+)} \leq C' h \mleft(\| D_{y_d} \varphi\|_{L^2(\R^d_+)} + \|\varphi\|_{L^2(\R^d_+)}\mright), \quad \varphi \in \overline{\S}(\R^d_+), \quad 0 < h \leq 1,
% \end{equation}
% for some other constant $C' > 0$. With that in hand we can handle the two commutator terms in $r$.
% % $\Chi \tilde{f}^C$ is of suitable size $o(h^{1 + \delta})$ in $L^2$-norm and 
It remains to estimate the two commutator terms.  
First, $[\theta^C, h^2 \tilde{\Delta}^C] \in h \diff^1_{\tau,h}$ with coefficients vanishing of a compact subset  of $\kappa(O)$; thus, using also \cref{lem:V-dissip-h} we see that
\begin{equation}
\|\Chi[\theta, h^2 \tilde{\Delta}^C]u^C\|_{L^2(\R^d_+)} = O(\|[\theta, h^2 \tilde{\Delta}^C]u^C\|_{L^2(\R^d_+)}) = O(h \|u\|_{H^1_h(\Omega)}) = O(h), \quad h \to 0.
\end{equation}
Next we look at the commutator $[\Chi, h^2 \tilde{\Delta}^C]$, which we may write as (minus) the sum of $[\Chi, h^2 D_{y_d}^2]$ and $[\Chi, P_\tau]$. Observe on the one hand that
$[\Chi, h^2D^2_{y_d}] = h^2[\Chi, D_{y_d}]D_{y_d} + h^2 D_{y_d}[\Chi, D_{y_d}]$:
% \begin{equation}
% [\Chi, h^2D^2_{y_d}] = h^2[\Chi, D_{y_d}]D_{y_d} + h^2 D_{y_d}[\Chi, D_{y_d}]
% \end{equation}
% $[\Chi, D_{y_d}] = \Opht(D_{y_d}\chi) \in \Psi_{\tau,h}^0(\R^d_+)$ (see \cite[Remark 2.46]{RouLeb22a}), 
therefore, using the fact that $[\Chi, D_{y_d}] \in \Psi_{\tau,h}^0$ combined with continuity properties of $\Chi$, we obtain
% \begin{equation}
% \end{equation}
% $[\Chi, h^2D_{y_d}] \in h \Psi_{\tau,h}^1$ and 
\begin{equation}
\|[\Chi, h^2 D_{y_d}]\tilde{u}^C\|_{L^2(\R^d_+)} 
= O(h\|\tilde{u}^C\|_{H^1_h(\R^d_+)})
= O(h \|u\|_{H^1_h(\Omega)}) 
% O(h^2 \|u\|_{H^1(\Omega)} + \|u\|_{L^2(\Omega)}) 
= O(h), \quad h \to 0.
\end{equation}
% by \cite[Theorem 2.47]{RouLeb22a} applied at the $L^2$- and $H^1$-levels.
On the other hand, the Poisson bracket (in the tangential variables) between the principal symbols of $\Chi$ and $P_\tau$ is given by
\begin{equation}
\label{eq:poisson-a}
% \mleft \{ \chi(y, \xi_\tau), \sum_{i,j =1}^{d-1} \tilde{a}_{ij}(y) \xi_i \xi_j \mright\} \\ 
\{ \chi, \sigma_h(P_\tau)\}(y, \xi_\tau)
% =   \sum_{k=1}^{d-1}  \dl{\chi}{y_k}(y, \xi_\tau)  \mleft ( 2\tilde{a}_{kk}(y) \xi_k + 2 \sum_{i =1}^{d-1} \tilde{a}_{ik}(y) \mright)   + \dl{\chi}{\xi_k}(y, \xi_\tau) \sum_{i,j =1}^{d-1}\dl{\tilde{a}_{ij}}{y_k}(y) \xi_i \xi_j
% \end{equation}
% \edp{Reverifié\[
=\sum_{k=1}^{d-1}\dl{\chi}{\xi_k}(y, \xi_\tau) \sum_{i,j =1}^{d-1}\dl{\tilde{a}_{ij}}{y_k}(y) \xi_i \xi_j-\dl{\chi}{y_k}(y, \xi_\tau)  \mleft (2 \sum_{i =1}^{d-1} \tilde{a}_{ki}(y)\xi_{i} \mright)
\end{equation}
for all $y \in \R^d_+$ and $\xi_\tau = (\xi_1, \dots, \xi_{d-1}) \in \R^{d-1}$. After noting that $\partial_{\xi_\tau}\chi$ vanishes for large $|\xi_\tau|$ (uniformly in $y$), we infer from \cref{eq:poisson-a} and \cite[Corollary 2.45]{RouLeb22a} that $[\Chi, P_\tau] \in h \Psi_{\tau,h}^1$, and thus
\begin{equation}
% \|[\Chi, h^2\tilde{\Delta}^C]\tilde{u}^C\|_{L^2(\R^d_+)} =
\|[\Chi, P_\tau]\tilde{u}^C\|_{L^2(\R^d_+)} = O(h\|\tilde{u}^C\|_{H^1_h(\R^d_+)}) = O(h), \quad h \to 0.
\end{equation}
Summing up, we have shown  that the term $r$ in \cref{eq:v-equation-bis} decays like $O(h)$ in $L^2$-norm as $h \to 0$.
% \begin{equation}
% dz
% \end{equation}

\emph{Step 5: conclusion.} Since $\Chi$ is bounded on $H^1_h(\R^d_+)$ uniformly in $h$, it is clear that 
% $v$
% Using again that $\chi \in S_{\tau,h}^0$, we have
\begin{equation}
\|v\|_{H_h^1(\R^d_+)} = O(1), \quad h \to 0.
\end{equation}
In light of the above developments, we may now deduce from \cref{eq:normal-r-r} that
\begin{equation}
h \|D_{y_d}v\|_{L^2(\R^{d-1})} = O(1), \quad h \to 0,
\end{equation}
and at this point, with the Cauchy--Schwarz inequality and the various previous estimates,  it readily follows from \cref{eq:bilan-ell} that, in fact,
\begin{equation}
\|v\|_{H^1_h(\R^d_+)} = O(h^{1/2}),  \quad h \to 0.
\end{equation}
With this improved estimate in hand we may return to \cref{eq:normal-r-r} and finally deduce that
\begin{equation}
h \|D_{y_d}v\|_{L^2(\R^{d-1})} = O(h^{1/4}), \quad h \to 0.
\end{equation}
The proof is now complete.
% Changing the constant $K' > 0$ and taking $h$ small enough if needed, in view of the above developments we may rewrite \cref{eq:normal-r-r} as
% \begin{equation}
% h^2 \|D_{y_d}v\|^2_{L^2(\R^{d-1})} \leq K'
% \end{equation}
%     % \edn{à finir}
\end{proof}

For the estimate at low tangential frequencies, we consider $o(h^3)$-quasimodes only.

\begin{prop}[Low-frequency regime]\label{prop:LF} For $o(h^3)$-quasimodes $u$, letting $w \triangleq (1 - \Chi) \tilde{u}^C$,  we have
\begin{equation}
\label{eq:normal-w}
 h\|D_{y_d}w|_{y_d = 0}\|_{L^2(\R^{d-1})} = o(1), \quad h \to 0.
\end{equation}
\end{prop}

\begin{proof}
We shall keep in mind the remarks,  properties and notation convention from the proof of \cref{prop:HF}.
As a result of \cref{eq:utilde-boundary} we see that $w$ solves
\begin{equation}
\label{eq:w-boundary}
(-h^{2}\tilde{\Delta}_{\Gamma}^{C}+ \i h - 1 ) w = -\i h^2 (1 - \Chi)\theta^C D_{y_d}u^C + (1 - \Chi)\tilde{e}^C -[\theta,h^{2}\Delta_{\Gamma}]^Cu^{C} + [\Chi, h^2 \tilde{\Delta}^C]\tilde{u}^C \quad \mbox{in}~\R^{d-1}.
\end{equation}
% \edp{question!}
% To obtain \cref{eq:w-boundary} we  used that, for $\varphi \in \overline{\S}(\R^d_+)$, $\Opht(\chi)\varphi|_{y_d = 0} = \Oph(\chi|_{y_d = 0})(\varphi|_{y_d = 0})$, where $\Oph$ indicates  standard semiclassical quantization on $\R^{d-1}$. Note also that since $\chi$ belongs to the class of tangential symbols $S^0_{\tau,h} = S^0_{\tau, h}(\R^d_+ \times \R^{d-1})$ then its boundary trace $\chi|_{y_d = 0}$ is in $S^0_h(\R^{d-1} \times \R^{d-1})$; see \cite[Remark 2.40]{RouLeb22a}. This justifies
% using the same notation $\Chi$ in \cref{eq:w-boundary}. 
The proof is divided into two steps.

\emph{Step 1: normal derivative.} We start by noting that
\begin{equation}
\label{eq:comm-Dy}
\begin{aligned}
(1 - \Chi) \theta^C D_{y_d} u^C & = (1 - \Chi) D_{y_d} \tilde{u}^C - (1 - \Chi)(D_{y_d} \theta^C)u^C \\
&= D_{y_d}w + [\Chi, D_{y_d}]\tilde{u}^C - (1 - \Chi)(D_{y_d} \theta^C)u^C,
\end{aligned}
\end{equation}
and thus, in view of the desired estimate \cref{eq:normal-w}, we may rewrite \cref{eq:w-boundary} as
\begin{equation}
\label{eq:w-boundary-bis}
\i h^2 D_{y_d}w = -(-h^{2}\tilde{\Delta}_{\Gamma}^{C}+ \i h - 1 )(1 - \Chi)\tilde{u}^C   + r \quad \mbox{in}~\R^{d-1},
\end{equation}
where the rest $r$ is made of terms left over from \cref{eq:w-boundary,eq:comm-Dy}. On the boundary $\R^{d-1}$, $1 - \Chi = \Oph((1 - \chi)|_{y_d = 0})$ and the (smooth, bounded) symbol $(1 - \chi)|_{y_d = 0}$ vanishes for large $|\xi_\tau|$ (uniformly in $y_\tau$). In particular, it belongs to the class $S^{-2}_h(\R^{d-1} \times \R^{d-1})$. Therefore, by \cite[Proposition E.19]{DyaZwo19book}, $1-\Chi$ continuously (and uniformly in $h$) maps $L^2(\R^{d-1})$ into $H^2_h(\R^{d-1})$, which is simply $H^2(\R^{d-1})$ with the semiclassical norm
\begin{equation}
\|\varphi\|^2_{H^2_h(\R^{d-1})} \triangleq \int_{\R^{d-1}}  \sum_{
\substack{
\alpha \in \N^{d-1} \\
|\alpha| \leq 2}} h^{2|\alpha|
}|D^\alpha \varphi|^2 \, \d y_\tau, \quad \varphi \in H^2(\R^{d-1}).
\end{equation}
% (Here $\alpha$ denotes multiindices of size $d-1$.) 
As an immediate consequence,
% there exists $K > 0$ (independent of $u$ and small $h$) such that
\begin{equation}
\|(-h^{2}\tilde{\Delta}_{\Gamma}^{C}+ \i h - 1 )(1 - \Chi)\tilde{u}^C\|_{L^2(\R^{d-1})} = O(\|(1-\Chi)\tilde{u}^C\|_{H^2_h(\R^{d-1})}) = O(\|\tilde{u}^C\|_{L^2(\R^{d-1})}), \quad h \to 0.
\end{equation}
% \begin{equation}
% h^2 \|(1 - \Chi)\tilde{u}^C\|_{H^2(\R^{d-1})} + \|(1 - \Chi)\tilde{u}^C\|_{L^2(\R^{d-1})} = O(\|\tilde{u}^C\|_{L^2(\R^{d-1})}), \quad h \to 0.
% % \leq \|\tilde{u}^C\|_{L^2(\R^{d-1})} \|\tilde{u}^C\|_{L^2(\R^{d-1})} = O(\|u\|_{L^2(\Gamma_0)}) = o(h), \quad h \to 0.
% \end{equation}
% \edn{[Normes semi-classiques pour raccourcir]?}
Here we specifically consider $o(h^3)$-quasimodes: recall from \cref{lem:V-dissip-h} that
\begin{equation}
\label{eq:H2-u^C}
\|\tilde{u}^C\|_{L^2(\R^{d-1})} = O(\|u\|_{L^2(\Gamma_0)}) = o(h), \quad h \to 0.
\end{equation}
% Since $h^2\tilde{\Delta}_\Gamma^C \in \diff_h^2(\R^{d-1})$, it follows 
Then,  it follows from \cref{eq:w-boundary-bis} and the above observations that
\begin{equation}
h \|D_{y_d} w\|_{L^2(\R^{d-1})} = h^{-1}\|r\|_{L^2(\R^{d-1})} + o(1), \quad h \to 0.
\end{equation}

\emph{Step 2: left-over terms.} Now it remains to handle the term $r$, which we must prove is of size $o(h)$ as $h \to 0$.
% First, noting that $(1 - \Chi)$ also acts as a bounded (uniformly in $h$) operator 
The operator $1 - \Chi$ is bounded on $L^2(\R^{d-1})$ uniformly in $h$ and thus
\begin{equation}
\|(1 - \Chi)\tilde{e}^C\|_{L^2(\R^{d-1})} = O(\|e\|_{L^2(\Gamma_0)}) = o(h^3), \quad h \to 0.
\end{equation}
Similarly,
\begin{equation}
\|(1 - \Chi)(D_{y_d} \theta^C)u^C\|_{L^2(\R^{d-1})} = O(\|u\|_{L^2(\Gamma_0)}) = o(h), \quad h \to 0.
\end{equation}
Next we deal with the commutators. Noting that $[\theta, h^2 \Delta_\Gamma]^C \in h\diff_h^1(\R^{d-1})$ we get
\begin{equation}
\|[\theta,h^{2}\Delta_{\Gamma}]^Cu^{C}\|_{L^2(\R^{d-1})} 
% = O(h^2 \|u^C\|_{H^1(\kappa(O))})
= O(h\|u\|_{H^1_h(\Gamma_0)}) = o(h^2), \quad h \to 0,
\end{equation}
where we also used \cref{lem:quasim-H1-gamma}. Also, just as in the proof of \cref{prop:HF},
% We also notice that $[\Chi, D_{y_d}] = \Opht(D_{y_d}\chi) \in \Psi_{\tau,h}^0(\R^d_+)$ (see \cite[Remark 2.46]{RouLeb22a}), so that
% t $[\Chi, D_{y_d}]$ operates on $L^2(\R^{d-1})$ in a uniformly bounded way
% and
% , again by \cite[Proposition E.19]{DyaZwo19book},
\begin{equation}
\|[\Chi, D_{y_d}]\tilde{u}^C\|_{L^2(\R^{d-1})} = O(\|\tilde{u}^C\|_{L^2(\R^{d-1})}) = o(h), \quad h \to 0.
\end{equation}
The last term to estimate is given by $[\Chi, h^2\tilde{\Delta}^C]\tilde{u}^C$: by a Poisson bracket argument almost identitical to that of the previous proof, together with, e.g., \cite[Proposition E.8]{DyaZwo19book}, we have
% we compute the Poisson bracket \edn{[à vérifier/mieux rédiger]}
% \begin{multline}
% \label{eq:poisson-bracket}
% \mleft \{ \chi(y_\tau, 0, \xi_\tau), \sum_{i,j =1}^{d-1} \tilde{b}_{ij}(y_\tau)\mright\} \\ = \sum_{k=1}^{d-1}  \dl{\chi}{y_k}(y_\tau, 0, \xi_\tau)  \mleft ( 2\tilde{b}_{kk}(y_\tau) \xi_k + 2 \sum_{i =1}^{d-1} \tilde{b}_{ik}(y_\tau) \mright)   + \dl{\chi}{\xi_k}(y_\tau, 0, \xi_\tau) \sum_{i,j =1}^{d-1}\dl{\tilde{b}_{ij}}{y_k}(y_\tau) \xi_i \xi_j,
% % \quad y_\tau = (y_1, \dots, y_{d-1}) \in \R^{d-1}, \quad \xi_\tau = (\xi_1, \dots, \xi_{d-1}) \in \R^{d-1},
% \end{multline}
% and after noting that $\partial_{\xi_\tau} \chi$ vanishes uniformly for large $|\xi_\tau|$ we deduce from \cref{eq:poisson-bracket} together with \cite[Proposition E.8]{DyaZwo19book} that 
$[\Chi, h^2 \tilde{\Delta}^C] \in h \Psi_{h}^1(\R^{d-1})$. Then, finally,
\begin{equation}
\|[\Chi, h^2\tilde{\Delta}^C]\tilde{u}^C\|_{L^2(\R^{d-1})} = O(h^2 \|{u}\|_{H^1(\Gamma_0)} + h \|u\|_{L^2(\Gamma_0)}) = o(h^2), \quad h \to 0,
\end{equation}
and the proof is complete.
\end{proof}

% \begin{rem}[Heuristics for the semiclassical scale] In light of the proof of \cref{prop:LF} we may now explain how to ``guess'' the $o(h^3)$-scale and thus the $o(t^{-1})$-decay of the energy of classical solutions to \cref{eq:IBVP}.

% This argument will be made precise in
% \end{rem}

We conclude this section by gluing together the estimates of \cref{prop:HF,prop:LF}. 
\begin{prop}[Improved normal derivative estimate]
\label{prop:improved-neumann} For $o(h^3)$-quasimodes $u$,
\begin{equation}
h \|\partial_\nu u\|_{L^2(\Gamma_0)} = o(1), \quad h \to 0.
\end{equation}
\end{prop}
% (Compare with \cref{lem:a-priori-neumann}.)
\begin{proof}
We indicate again dependence on the particular local chart $C_n = (O_n, \kappa_n)$ from \cref{sec:normal-geodesic} with subscripts $n$.
With the partition of unity that we defined previously, we have $\partial_\nu u = \sum_{n = 1}^N \partial_\nu (\theta_n u)$ on $\Gamma_0$, and thus
\begin{equation}
% \partial_\nu u = \sum_{n = 0}^N \partial_\nu (\theta_n u) = \sum_{n = 0}^N \partial_\nu \tilde{u}_n
\|\partial_\nu u\|_{L^2(\Gamma_0)} \leq \sum_{n = 1}^{N} \|\partial_\nu (\theta_n u)\|_{L^2(\Gamma_0)} \leq K \sum_{n = 1}^N \|D_{y_d}\tilde{u}^C_n\|_{L^2(\R^{d-1})},
\end{equation}
where the positive constant $K$ comes from changes of variables by means of the diffeomorphisms $\kappa_n$ and is independent of quasimodes. For each $n = 1, \dots, N$,
\begin{equation}
\|D_{y_d}\tilde{u}^C_n\|_{L^2(\R^{d-1})} \leq  \|D_{y_d} \Chi_n\tilde{u}^C_n\|_{L^2(\R^{d-1})} + \|D_{y_d}(1 - \Chi_n)\tilde{u}^C_n\|_{L^2(\R^{d-1})},
\end{equation}
and the result immediately follows from \cref{prop:HF,prop:LF}.
\end{proof}

\subsubsection{Decoupling arguments and conclusion}
\label{sec:decoupling}

At this point, we know from \cref{lem:quasim-H1-gamma,prop:improved-neumann} that $o(h^3)$-quasimodes $u$ satisfy
\begin{equation}
\label{eq:summary-gamma}
\|u\|_{H^1_h(\Gamma_0)} + h \|\partial_\nu u\|_{L^2(\Gamma_0)} = o(1), \quad h \to 0.
\end{equation}
Thus, we have suitable control over boundary traces on $\Gamma_0$ in both tangential and normal directions. Our next goal is to take advantage of the geometric control condition  to propagate our boundary estimates \cref{eq:summary-gamma} over the whole domain $\Omega$.
To that end, we shall use a splitting $u = u_1 + u_2$ of quasimodes in the following manner.
\begin{enumerate}
    \item  Define $u_1$ as a suitable lifting of the boundary data $u|_{\Gamma_0}$. Here we exploit \emph{well-posedness}
    of the Dirichlet problem for the wave equation to derive interior and boundary estimates of $u_1$.
    \item Recover $u_2$ in terms of its Neumann trace
    $\partial_\nu u_2 = \partial_\nu (u - u_1)$
    on $\Gamma_0$.
    This step relies on \emph{exact observability} of the wave equation with homogeneous Dirichlet boundary condition.
\end{enumerate}
% We now introduce some further notation. 
% Finally, $\mathds{1}_{\Gamma_0}$ is the indicator function of $\Gamma_0$ as a subset of $\Gamma$: since $\overline{\Gamma_0} \cap \overline{\Gamma_1} = \emptyset$ by assumption, it is in fact smooth.
To implement the first step of our strategy, we need the following lemma.

\begin{lemma}
\label{lem:dirichlet-pb} 
Let $h > 0$. For every $g \in H^{1/2}(\Gamma)$ there exists a unique $w \in H^1(\Omega)$ such that
\begin{subequations}
\label{eq:dirichlet-pb}
\begin{align}
\label{eq:dirichlet-pde}
&(-h^2 \Delta + 2 \i h - 1)w = 0 &&\mbox{in}~\Omega, \\
&w = g &&\mbox{on}~\Gamma.
\end{align}
\end{subequations}
% Furthermore, if $g \in H^1(\Gamma)$ then $\partial_\nu w \in L^2(\Gamma)$. 
% with continuous dependence from $H^{1/2}(\Gamma)$ into $H^1(\Omega)$.
% dz
% The solution $w$ is given by
% % \begin{equation}
% % w = (-h^2 \Delta_D + 2\i h - 1)^{-1} \Delta_D Dg
% % \end{equation}
\end{lemma}
\begin{proof}
Let $D \in \L(H^{1/2}(\Gamma), H^1(\Omega))$ denote the Dirichlet \emph{harmonic extension} map, which is defined as follows: $w = Dg$ if and only if $\Delta w = 0$ in $\Omega$ and $w = g$ on $\Gamma$; see, e.g., \cite[Chapter 2]{LioMag68book} for
the existence and continuity of such a map. The Dirichlet Laplacian $\Delta_D$ is defined as an unbounded (negative self-adjoint) operator on $L^2(\Omega)$ by $\dom(\Delta_D) \triangleq \{ w_0 \in H^1_0(\Omega) : \Delta w \in L^2(\Omega)\} = H^2(\Omega) \cap H^1_0(\Omega)$ and $\Delta_D w = \Delta w$ for $w \in \dom(\Delta_D)$; it is negative self-adjoint and also naturally extends to isomorphisms from $H^1_0(\Omega) = \dom((-\Delta_D)^{1/2})$ onto $H^{-1}(\Omega)$\footnote{Strictly speaking, here $H^{-1}(\Omega)$ is seen as a space of \emph{antilinear} forms.} and from $L^2(\Omega)$ onto $\dom(\Delta_D)^\ast$.
% (seen as a space of antilinear forms)
% Recall 
% % $\dom((-\Delta_D)^{1/2}) = H^1_0(\Omega)$
% that $\Delta_D$ maps $H^1_0(\Omega)$ into $H^{-1}(\Omega)$. 
With this in mind, we see that,
given $g \in H^{1/2}(\Gamma)$, $w \in H^1(\Omega)$ satisfies \cref{eq:dirichlet-pb} if and only if $w - Dg \in H^1_0(\Omega)$  and
% rewriting \cref{eq:dirichlet-pde} after recalling that $\Delta_D$ maps $H^1_0(\Omega)$ into $H^{-1}(\Omega)$
% $ -h^2\Delta( w -h^{-2}Dg) +2 \i h w - w = 0$ in $\D'(\Omega)$ 
$-h^2 \Delta_D(w - h^{-2}Dg) + 2\i h w - w = 0$ in $H^{-1}(\Omega)$. On the other hand, since the operator $-h^2 \Delta_D$ is positive, $1 - 2 \i h \not \in \R$ lies in its resolvent set, and we may therefore let $w \triangleq -(-h^2 \Delta_D + 2\i h - 1)^{-1}\Delta_DDg$. The candidate $w$ is \emph{a priori} defined in $L^2(\Omega)$; following the above observations, it is immediate to check that $w$ belongs to $ H^1(\Omega)$ and is indeed the unique solution to \cref{eq:dirichlet-pb}.
% The additional boundary regularity $\partial_\nu w \in L^2(\Omega)$ if $g \in H^1(\Gamma)$ follows from
% \edp{Je suis un peu perdu, $\Delta_DDg=0$?} \edn{ici la nuance est que $\Delta Dg = 0$ (laplacien au sens des distributions) mais $\Delta_D Dg \not = 0$ (operateur "Dirichlet laplacien")}
\end{proof}
Next, we give estimates for solutions $w$ in \cref{lem:dirichlet-pb} that are uniform in the small parameter $h$.
% The following proposition
\begin{prop}[Dirichlet-to-Neumann bounds]
\label{prop:diri-neu}
If $g \in H^1(\Gamma)$ then the corresponding solution $w$ to \cref{eq:dirichlet-pb}, as given by \cref{lem:dirichlet-pb}, satisfies $\partial_\nu w \in L^2(\Gamma)$. In fact, there exists $K > 0$ such that, 
% for all $h > 0$, for all $g \in H^1(\Gamma)$ and all corresponding solutions $w$ to \cref{eq:dirichlet-pb} as given by \cref{lem:dirichlet-pb},
\begin{equation}
\label{eq:diri-to-neu}
\|w\|_{H^1_h(\Omega)} + h \|\partial_\nu w\|_{L^2(\Gamma)} 
  \leq K \|g\|_{H^1_h(\Gamma)}, \quad g \in H^1(\Gamma), \quad 0 < h \leq 1.
\end{equation}
\end{prop}
\begin{proof}
The property that $\partial_\nu w$ (\emph{a priori} defined in $H^{-1/2}(\Gamma)$ as in \cref{eq:weak-norm}) belongs to $L^2(\Gamma)$ when $g \in H^1(\Gamma)$ is already contained in \cite[Theorem 1.1, Chapter 5]{Nec12book};  alternatively we may establish \cref{eq:diri-to-neu} for sufficiently smooth functions first and recover boundary regularity via density arguments.
% In what follows we may assume that $w$ is smooth: indeed, \cite[Appendix A]{MoiSpe14survey} shows that $\C^\infty(\overline{\Omega})$ is dense in the space $\{ w \in H^1(\Omega) : \Delta w \in L^2(\Omega),\ w \in H^1(\Gamma),\ \partial_\nu w \in L^2(\Gamma) \}$ (equipped with its natural norm), and establishing
Now, it follows from \cref{eq:dirichlet-pde} that
\begin{equation}
\label{eq:pairing-w}
\int_\Omega h^2  |\nabla w|^2 \, \d x +(2 \i h -1)  |w|^2 \, \d x = h^2 \int_\Gamma \partial_\nu w \overline{g} \, \d \sigma, \quad g \in H^1(\Gamma), \quad h > 0.
\end{equation}
% \begin{equation}
% h^2 \int_\Omega |\nabla u|^2 \, \d x = \frac{1}{2}h  \im \int_\Gamma \partial_\nu w \overline{g} \, \d \sigma + \re \int_\Gamma \partial_\nu w \overline{g} \, \d \sigma 
% \end{equation}
By successively taking the imaginary and real parts of \cref{eq:pairing-w} and using Young's inequality, we deduce that, for some constant $K > 0$,
\begin{equation}
\label{eq:pairing-young}
\int_\Omega h^2 |\nabla w|^2 + |w|^2 \, \d x \leq \frac{\eps h^2K}{2} \int_{\Gamma} |\partial_\nu w|^2 \, \d \sigma + \frac{K}{2\eps} \int_\Gamma |g|^2 \, \d \sigma, \quad \eps > 0, \quad g \in H^1(\Gamma), \quad 0 < h \leq 1.
\end{equation}
Similarly as in the proof of \cref{lem:a-priori-neumann}, we may use the multiplier identity \cref{eq:rellich} from \cref{lem:rellich} with a vector field $q$ that satisfies $q \cdot \nu = 1$ on $\Gamma$. After a couple of Cauchy--Schwarz and Young inequalities we then see that there exists a constant $K' > 0$ such that
\begin{equation}
\label{eq:diff-diri-to-neu}
% h^2 \|\partial_\nu w\|^2_{L^2(\Gamma)} 
h^2 \int_\Gamma |\partial_\nu w|^2 \, \d \sigma
\leq K \mleft( \int_\Omega |w|^2 + h^2 |\nabla w|^2 \, \d x + \int_{\Gamma} |g|^2 + h^2|\nabla_\Gamma g|^2 \, \d \sigma  \mright), \quad g \in H^1(\Gamma), \quad h > 0.
\end{equation}
(Here $w$ satisfies \cref{eq:rellich-trace} with $f$ replaced by $-2\i h w$ and $\Gamma_0$ replaced by $\Gamma$.) To complete the proof, simply make the sum of \cref{eq:pairing-young,eq:diff-diri-to-neu} with $\eps$ chosen sufficiently small.
\end{proof}

\begin{rem}
In \cref{eq:dirichlet-pb} the additional ``damping'' term $2 \i h w$ allows us to stay away from the spectrum of the Dirichlet Laplacian, so that $w$ is uniquely determined by its boundary data $g$. For the pure Helmholtz equation $(-h^2 \Delta -1)w = 0$, Dirichlet-to-Neumann bounds such as \cref{eq:diri-to-neu} must take the form of \emph{a priori estimates}; see for instance \cite{Spe14,BasSpe15}. 
\end{rem}

% \begin{rem}
% On a similar note, the asymptotics of \cref{prop:diri-neu} are very much related to time-domain regularity results for the wave equation with Dirichlet control; see, e.g., \cite[Theorem 2.1]{LasLio86}, which  shows that $H^1(\Gamma \times (0, T))$-Dirichlet data with suitable compatibility conditions produce finite-energy solutions.
% \end{rem}

% \begin{rem}
% dz
% \end{rem}

We are ready to introduce our splitting of quasimodes $u$. Let $u_1$ be the unique solution to \cref{eq:dirichlet-pb} with Dirichlet data $g  = u|_{\Gamma}$, and let $u_2 \triangleq u - u_1$.
% \Cref{lem:dirichlet-pb,prop:diri-neu} 
\begin{coro}
\label{coro:neu-dir-qm}
For $o(h^{{1 + \delta}})$-quasimodes $u$, 
% letting $u_1$ be the unique solution to \cref{eq:dirichlet-pb} with data , we have
\begin{equation}
\label{eq:o-u1}
\|u_1\|_{H^1_h(\Omega)}+ h \|\partial_\nu u_1\|_{L^2(\Gamma)}  = o(1), \quad h \to 0.
\end{equation}
% For general $o(h^{1+\delta})$-quasimodes, \cref{eq:o-u1} holds with $O(1)$ instead of $o(1)$.
% we in fact have
% \begin{equation}
% h \|\partial_n u_1\|_{L^2(\Gamma)} +
% h \|u_1\|_{H^1(\Omega)} + \|u_1\|_{L^2(\Omega)} = o(1), \quad h \to 0.
% \end{equation}
\end{coro}
\begin{proof}
This is an immediate consequence of \cref{lem:quasim-H1-gamma,prop:diri-neu}. Here we also use that $\|u\|_{H^1_h(\Gamma)} = \smash\|u\|_{H^1_h(\Gamma_0)}$ \smash{since $\overline{\Gamma_0} \cap \overline{\Gamma_1} = \emptyset$ and $u|_{\Gamma_1} = 0$.}
% the indicator function $\mathds{1}_{\Gamma_0}$ is smooth on $\Gamma$ due to the assumption that $\overline{\Gamma_0} \cap \overline{\Gamma_1} = \emptyset$.
\end{proof}
Now we deal with the remaining part $u_2$. This is the final step towards the conclusion that there are actually no $o(h^3)$-quasimodes under the geometric control condition.
\begin{prop}[Observation from the boundary]
\label{prop:boundary-qm}
Suppose that \cref{as:dyn} holds.
For $o(h^{3})$-quasimodes $u$, 
\begin{equation}
\|u_2\|_{H^1_h(\Omega)}  = o(1), \quad h \to 0.
\end{equation}
\end{prop}
\begin{proof}
Under the geometric control condition and by virtue of \cite[Theorem 3.8]{BarLeb92}, the wave equation with homogeneous Dirichlet boundary condition is exactly observable through the Neumann trace on $\Gamma_0$.
\Cref{lem:hautus} expresses this fact
% the following observability inequality holds, here 
in resolvent form as a \emph{Hautus test} and yields $M,m >0$ such that
\begin{equation}
\label{eq:hautus-corp}
\|w\|_{H^1_h(\Omega)} \leq h^{-1} M \|(-h^2 \Delta - 1)w\|_{L^2(\Omega)} +  mh  \|\partial_\nu w\|_{L^2(\Gamma_0)}, \quad w \in \dom(\Delta_D), \quad h > 0.
\end{equation}
(Here $\Delta_D$ denotes the Dirichlet Laplacian, introduced in the proof of \cref{lem:dirichlet-pb}.) By construction $u_2 = u - u_1$ solves
\begin{subequations}
\begin{align}
&(-h^2 \Delta - 1)u_2 = f + 2 \i h u_1&&\mbox{in}~\Omega, \\
&u_2 = 0&&\mbox{on}~\Gamma.
\end{align}
\end{subequations}
In particular $u_2 \in \dom(\Delta_D)$ and \cref{eq:hautus-corp} gives, with the triangle inequality,
\begin{equation}
 \|u_2\|_{H^1_h(\Omega)} \leq h^{-1}M  \mleft( \|f\|_{L^2(\Omega)} + 2 h \|u_1\|_{L^2(\Omega)} \mright) + mh \mleft(\|\partial_\nu u\|_{L^2(\Gamma_0)} + \|\partial_\nu u_1\|_{L^2(\Gamma_0)}\mright)
\end{equation}
Thus the result follows from
\cref{prop:improved-neumann,coro:neu-dir-qm}.
\end{proof}
\begin{coro}
\label{coro:no-qm}
Under \cref{as:dyn} there are in fact no $o(h^3)$-quasimodes.
\end{coro}
\begin{proof}
By \cref{coro:neu-dir-qm,prop:boundary-qm}, $o(h^3)$-quasimodes $u$ must satisfy $\|u\|_{L^2(\Omega)} = o(1)$ as $h \to 0$. Since $\|u\|_{L^2(\Gamma_0)}
 \to 0$ as well, this  contradicts $\|u\|_H = 1$.
\end{proof}
% \begin{rem}
% \edn{Bla bla}
% \end{rem}

Our study of quasimodes is now complete, and at last we can write down the proof of \cref{prop:res}.
\begin{proof}[Proof of \cref{prop:res}]
Observe first that 
\begin{equation}
\|P^{-1}(\i \lambda)\|_{\L(H)} = \|P^{-1}(- \i \lambda)^\ast\|_{\L(H)} = \|P^{-1}(- \i \lambda)\|_{\L(H)}, \quad \lambda \in \R, \quad \lambda \not = 0.
\end{equation}
Thus it suffices to prove that \cref{eq:second-order-res-estimate} holds  for $\lambda \to + \infty$. Arguing by contradiction, suppose that for all integers $n \geq 1$ there exists a real $\lambda_n > n$ and a vector $(f_n, e_n) \in H$ such that
\begin{equation}
\lambda_n^{-1} \|P(\i \lambda_n)^{-1}(f_n, e_n)\|_H > n \|(f_n, e_n)\|_{H}.
\end{equation}
Letting $v_n \triangleq P(\i \lambda_n)^{-1}(f_n, e_n) \in \dom(A) \setminus \{0\}$ and $u_n \triangleq v_n / \|v_n\|_H$, we have $\|u_n\|_H = 1$ and
\begin{equation}
\label{eq:qm-lambda}
\|(-A + \i \lambda_n BB^\ast - \lambda_n^2)u_n\|_H < \frac{1}{n \lambda_n}, \quad n \geq 1.
\end{equation}
Writing $h_n \triangleq \lambda_n^{-1}$ and multiplying \cref{eq:qm-lambda} by $h_n^2$ we immediately deduce that $\{(h_n, u_n)\}_{n \geq 1}$ is a sequence of normalized $o(h_n^3)$-quasimodes in the sense of \cref{def:qm}, which is impossible by \cref{coro:no-qm}.
\end{proof}

\subsection{Focusing modes in the disk and optimality of the decay rate}

\label{sec:focusing}

% Our goal in this section is to prove that the decay rate of \cref{th:intro} is optimal.
Now we are in a position to prove \cref{th:intro-sharp}, which guarantees that the energy decay rate of \cref{th:intro} is sharp.
To do so we investigate the \emph{spectrum} $\sigma(P)$ of the quadratic operator pencil $P$ in the case where 
$\Omega$ is the unit disk $ \mathbb{D} \triangleq \{x \in \R^2 : |x| < 1\}$, with $\Gamma = \Gamma_0 = \partial \mathbb{D}$. By spectrum we mean the set of complex numbers $z$ such that $P(z)$ does not have an inverse in $\L(H)$; as we have seen in \cref{sec:sg-generation}, it is made of \emph{eigenvalues} $z$, for which the equation $P(z)u = 0$ has nontrivial solutions $u$ which we refer to as \emph{eigenvectors}. In \cref{prop:roots-bessel} below, we exhibit a sequence of eigenvalues $\{z_k\}_{k\gg1}$ such that
\begin{equation}
\im z_k \to + \infty, \quad \re z_k < 0, \quad \re z_k = O((\im z_k)^{-2}), \quad k \to + \infty.
\end{equation}
This will show that the resolvent estimate of \cref{prop:res} cannot be improved (see \cref{coro:qm-bessel}) and yield \cref{th:intro-sharp} as a consequence of semigroup theory.
% and yields \cref{th:intro-sharp} as a consequence of  semigroup theory.
% , that, for all $\eps > 0$, there are classical solutions to \cref{eq:IBVP} with energy decay rate slower than $O(t^{\eps-1})$ for all $\eps > 0$.

Analyzing the spectrum of $P$ amounts to investigating a quadratic eigenvalue problem for the Laplacian with a special, $z$-dependent boundary condition. In the case of the disk, we shall therefore make heavy use of \emph{Bessel functions} and their properties; see the monograph \cite{Wat95book} for extensive details.
% The analysis of Laplacian eigenvalues on the disk heavily relies on the properties of \emph{Bessel functions}.
Recall that the {Bessel function} $J_n$ of the first kind and of order $n \in \N$ is given by
\begin{equation}
J_n(z) = \sum_{m = 0}^{+\infty} \frac{(-1)^m}{m! (m + k)!} \mleft( \frac{z}{2} \mright)^{2m + n}, \quad z \in \Co.\end{equation}
Each $J_n$ is an entire function
% , that is, a holomorphic function of the whole complex plane
that solves the Bessel differential equation
\begin{equation}
\label{eq:bessel-ode}
z^2 \frac{\d^2 J_n}{\d z^2}  + z \frac{\d J_n}{\d z} + (z^2 - n^2) J_n = 0.
\end{equation}
As we will see in a moment, \cref{eq:bessel-ode} typically arises when expressing the Laplacian in \emph{polar} coordinates $(r, \theta)$:
\begin{equation}
\label{eq:laplace-polar}
\Delta = \ddl{}{r} + \frac{1}{r} \dl{}{r} + \frac{1}{r^2} \ddl{}{\theta}.
\end{equation}
Note that the Laplace--Beltrami operator  $\Delta_\Gamma$ and the normal derivative $\partial_\nu$ on $\Gamma = \partial \mathbb{D}$  simply read as 
\begin{equation}
\Delta_\Gamma = \ddl{}{\theta}, \quad \dl{}{\nu} = \dl{}{r}.
\end{equation}
% In what follows, with a slight abuse in notation we will use polar coordinates freely and directly write, for instance, $u(x_1, x_2) = u(r, \theta)$, $|x|= r$
Here and in the sequel, with a slight abuse in notation, we write polar coordinates freely: for instance, for a function $u$ of $\overline{\mathbb{D}}$, $u(x_1, x_2) = u(r, \theta)$, $r = (x_1^2 + x_2^2)^{1/2}$, $0 \leq r \leq 1$, $x_1 = r \cos \theta$, $x_2 = r \sin \theta$, $\theta \in \R/2\pi$. That is, we omit change of variable maps.
\begin{lemma}
\label{lem:bessel-eg}
Let $n \in \N$ and $\lambda \in \Co$. If $\lambda$ is nonzero and satisfies
\begin{equation}
\label{eq:roots-bessel}
(n^2 +\i \lambda - \lambda^2) J_n(\lambda)  + \lambda J'_n(\lambda) = 0,
\end{equation}
then $\i \lambda \in \sigma(P)$\footnote{Here we let $\lambda$ be complex, in contrast with the previous sections.} and a corresponding eigenvalue $u \in \dom(A)$ is given by
\begin{equation}
\label{eq:bessel-eigenvalue}
u(r, \theta) = J_n(\lambda r) e^{\i n \theta}, \quad 0 \leq r \leq 1, \quad \theta \in \R \setminus 2\pi.
\end{equation}
\end{lemma}
\begin{proof}
Since $J_n$ is entire, $J_n(\lambda r)$ cannot vanish for all $0 \leq r \leq 1$ unless $\lambda = 0$; as a result the function $u$ given by \cref{eq:bessel-eigenvalue} is nontrivial. Now, using \cref{eq:laplace-polar} and then \cref{eq:bessel-ode} we see that
\begin{equation}
\label{eq:pol-int}
- \Delta u(r, \theta) = -\frac{1}{r^2}\mleft(\lambda^2 r^2J''_n(\lambda r) + \lambda r J'_n(\lambda r) - n^2 J_n(\lambda r)\mright) e^{\i n \theta} = \lambda^2 J_n(\lambda r)e^{\ n\theta} = \lambda^2 u(r, \theta)
\end{equation}
for all $0 < r < 1$ and $\theta \in \R / 2\pi$. On the other hand,
\begin{equation}
\label{eq:pol-bound}
(-\Delta_\Gamma + \i \lambda + \partial_\nu) u(\theta) = (n^2 + \i \lambda)J_n(\lambda)e^{\i n \theta} + \lambda J'_n(\lambda) e^{\i n\theta} = \lambda^2 J_n(\lambda)e^{\i n \pi} = \lambda^2 u(1, \theta)
\end{equation}
for all $\theta \in \R / 2\pi$ due to \cref{eq:roots-bessel}. Recalling the explicit form of $P$, we immediately deduce from \cref{eq:pol-int,eq:pol-bound} that $P(\i \lambda)u = 0$, as required.
\end{proof}
\begin{rem}
We are not concerned with the multiplicity of eigenvalues or the  description of all eigenvectors.
\end{rem}
Since we look for eigenvalues $z$ of $P$ in the high-frequency regime, we shall investigate complex roots $\lambda$ of \cref{eq:roots-bessel} with large real part and small imaginary part. For fixed $n \in \N$, after dividing \cref{eq:roots-bessel} by $n^2 + \i \lambda - \lambda^2$, we expect that roots satisfy
\begin{equation}
J_n(\lambda) \simeq 0, \quad \re \lambda \gg 1.
\end{equation}
This suggests seeking our $\lambda$ as  perturbations of \emph{real} large roots $\alpha$ of
\begin{equation}
\label{eq:bessel-hom}
J_n(\alpha) = 0,
\end{equation}
which, in terms of Laplacian eigenvalues, corresponds to the well-known case of homogeneous Dirichlet boundary condition. Positive roots of \cref{eq:bessel-hom} are discrete and denoted by $\alpha_{nk}$, $n,k \in \N$, where the sequence is increasing in $k$ for fixed $n$, and corresponding eigenmodes with zero Dirichlet data are spanned by
\begin{equation}
u^{\pm}_{nk}(r, \theta) = J_n(\alpha_{nk}r)e^{\pm \i n \theta}, \quad 0 \leq r \leq 1, \quad  \theta \in \R \setminus 2\pi.
\end{equation}
In particular, for fixed $n$,  the limit $k \to + \infty$ yields \emph{focusing modes} that are highly oscillating in the radial direction and concentrate around the origin: see \cite{NguGre13,GreNgu13review} for an in-depth investigation of their localization properties. Our subsequent analysis will shed some light on how the associated eigenvalues behave upon
replacing the homogeneous Dirichlet boundary condition with our dissipative dynamic boundary condition \cref{eq:HBC}. We consider the case $n = 0$ only.

\begin{prop}[Root asymptotics]
\label{prop:roots-bessel}
There exists a sequence $\{\lambda_k\}_{k\gg 1}$ of complex roots of
\begin{equation}
\label{eq:bessel-roots-0}
(\i \lambda - \lambda^2)J_0(\lambda) + \lambda J'_0(\lambda) = 0
\end{equation}
that satisfies
\begin{equation}
\lambda_k = k\pi - \frac{\pi}{4} + \frac{1}{8\mleft( k\pi - \frac{\pi}{4}\mright)} + \frac{k\pi - \frac{\pi}{4}}{1 + \mleft( k \pi - \frac{\pi}{4}\mright)^2} + \frac{\i}{1 + \mleft( k \pi - \frac{\pi}{4}\mright)^2} + O\mleft( \frac{1}{k^3}\mright), \quad k \to + \infty.
\end{equation}

\end{prop}

\begin{proof}
Let us first recall from \cite{Wat95book} some further facts regarding Bessel functions. For all $n \in \N$ and $z \in \Co$, we have the bound $|J_n(z)| \leq e^{\im z}$ and the recurrence relation
\begin{equation}
\frac{\d }{\d z} (z^n J_n(z)) = z^n J_{n-1}.
\end{equation}
In particular,  $J'_0 = - J_1$, and both $J_0$ and $J'_0$ remain bounded in any horizontal strip of finite width. Furthermore, $J''_0(z) = - J_1'(z) = - (J_0(z) + J_1(z)/z)$, and thus $J'_0$ is uniformly Lipschitz on, say the region $\{z \in \Co : \re z \geq 1,\ - 1 \leq \im z \leq 1\}$. 
We shall use the Hankel expansion \cite[Section 7.21]{Wat95book}, which reads as
\begin{equation}
\label{eq:hankel}
J_0(z) = \sqrt{\frac{2}{\pi z}}  \mleft ( p(z) \cos \mleft( z - \frac{\pi}{4} \mright) - q(z) \sin \mleft( z - \frac{\pi}{4} \mright)    \mright), \quad - \pi < \arg z < \pi, \quad |z| \to + \infty,
\end{equation}
where we take the principal branch of the square root, and
\begin{equation}
\label{eq:prop-p-q}
p(z) = 1 
% - \frac{9}{2(8z)^2} 
+ O\mleft(\frac{1}{|z|^2}\mright), \quad q(z) = - \frac{1}{8z} + O\mleft(\frac{1}{|z|^3}\mright), \quad |z| \to + \infty.
\end{equation}
Similarly,
\begin{equation}
\label{eq:hankel-prime}
J'_0(z) =  - \sqrt{\frac{2}{\pi z}} \mleft( f(z) \sin \mleft ( z - \frac{\pi}{4}\mright)  + g(z) \cos \mleft ( z - \frac{\pi}{4}\mright) \mright), \quad - \pi < \arg z < \pi, \quad |z| \to + \infty,
\end{equation}
where
\begin{equation}
\label{eq:prop-f-g}
f(z) = 1 + O\mleft( \frac{1}{|z|^2} \mright), \quad 
g(z) = \frac{3}{8z} + O\mleft(\frac{1}{|z|^3}\mright), \quad |z| \to + \infty.
\end{equation}
Furthermore, zeroes $\alpha_{0k}$ of $J_0$ satisfy the McMahon expansion \cite[Section 15.52]{Wat95book}
\begin{equation}
\label{eq:mcmahon}
\alpha_{0k} = k \pi - \frac{\pi}{4} + \frac{1}{8\mleft ( k \pi - \frac{\pi}{4} \mright)} + O\mleft( \frac{1}{k^3} \mright), \quad k \to + \infty.
\end{equation}
% This stems from the Hankel expansion
% \begin{equation}
% J_0(z) = \sqrt{\frac{2}{\pi z}} 
% \end{equation}
With that in hand, we may proceed to the proof, which is divided into three steps. In what follows, since we consider the Bessel function $J_0$ only, we drop the subscript indicating the order $n = 0$.

\emph{Step 1:  lower bound away from zeroes.} Fix $\eps \in \Co$ with $|\eps|  \leq 1$.
Then, according to \cref{eq:hankel},
\begin{equation}
\label{eq:lower-bound-start}
H_k(\eps) \triangleq 
\sqrt{\frac{\pi (\alpha_k + \eps)}{2}} J(\alpha_k + \eps) = \cos\mleft ( \alpha_k + \eps - \frac{\pi}{4} \mright)
+ O\mleft ( \frac{1}{k} \mright), \quad k \to + \infty
\end{equation} 
Using \cref{eq:mcmahon} at leading order, we see that
\begin{equation}
\alpha_k + \eps - \frac{\pi}{4} = 
k\pi - \frac{\pi}{2} + \eps + O\mleft(\frac{1}{k}\mright), 
% \mleft (k \pi - \frac{\pi}{2}\mright) \mleft (1 + O\mleft( \frac{1}{k} \mright) \mright),
% \theta_k \mleft( 1 + \frac{\eps}{\theta_k} \mright), \quad \frac{\eps}{\theta_k} = O\mleft ( \frac{1}{k} \mright), 
\quad k \to + \infty,
\end{equation}
and thus, with a trigonometric formula,
\begin{equation}
\label{eq:cos-sin-eps}
 \cos\mleft ( \alpha_k + \eps - \frac{\pi}{4} \mright) = (-1)^k \sin \mleft (\eps + O\mleft(\frac{1}k\mright)\mright), \quad k \to + \infty.
\end{equation}
Now recall that the only complex root of $\sin z = 0$ in the open unit disk is $0$. As a result, there surely exists $c > 0$ such that $|\sin z| \geq c$ if $1/2 \leq |z| \leq 1$. In particular, if $|\eps| = 3/4$ (say), then it follows from \cref{eq:cos-sin-eps,eq:lower-bound-start} that, whenever $k$ is sufficiently large,
\begin{equation}
\label{eq:lower-eps}
|H_k(\eps)| = 
\mleft |\sqrt{\frac{\pi (\alpha_k + \eps)}{2}}J(\alpha_k + \eps) \mright | 
\geq \frac{c}{2}.
\end{equation}
\emph{A priori} \cref{eq:lower-eps} holds for integers $k$ greater than some quantity that depends on $\eps$. Let us check that, actually, a uniform lower bound holds.
First,
\begin{equation}
\label{eq:H-k-p}
H'_k(\eps) = 
\frac{\d}{\d \eps} \sqrt{\frac{\pi (\alpha_k + \eps)}{2}} J(\alpha_k + \eps) =  \sqrt{\frac{\pi}{2}} \mleft ( - \frac{J(\alpha_k + \eps)}{2\sqrt{\alpha_k + \eps}} +  \sqrt{\alpha_k + \eps} J'(\alpha_k + \eps) \mright).
\end{equation}
\Cref{eq:hankel-prime,eq:mcmahon} show that $\sqrt{\alpha_k} J'(\alpha_k) = O(1)$ as $k \to \infty$, and since, for large $k$, the $\alpha_k + \eps$ live in a region where $J'$ is uniformly Lipschitz continuous, we see that $\sqrt{\alpha_k + \eps} J'(\alpha_k + \eps)$, and thus $H'_k(\eps)$ by \cref{eq:H-k-p}, remains bounded uniformly in $|\eps| \leq 1$ and $k$. Thus, when $k$ is sufficiently large, $H_k$
% and this expression is bounded (in modulus) uniformly in $|\eps| \leq 1$ and $k \gg 1$, so that the function of $\eps$ given by the left-hand side of \cref{eq:lower-bound-start}
is Lipschitz continuous on the unit disk with a constant that we may choose independently of $k$.
Therefore, since the circle of radius $3/4$ is compact, with a simple covering argument we may deduce from \cref{eq:lower-eps} that there exists $K \in \N$ such that $|H_k(\eps)| \geq c/4$ for all $\eps \in \Co$ with $|\eps| = 3/4$ and $k \geq K$. In particular, using \cref{eq:mcmahon} we deduce that there exists $c' > 0$ such that, taking $K$ larger if needed,
\begin{equation}
\label{eq:lower-bound-Jk}
\inf_{\eps \in \Co,\ |\eps| = \frac{3}{4}} {|J(\alpha_k + \eps)|} \geq \frac{c'}{\sqrt{k}}, \quad k \geq K.
\end{equation}

\emph{Step 2: existence of complex roots.} Our argument here is based on Rouch\'e's theorem.
We consider roots $\lambda$ with large real part and small imaginary part, and thus both $\lambda$ and $\i \lambda - \lambda^2$ may be assumed nonzero. Therefore, \cref{eq:bessel-roots-0} is equivalent to
\begin{equation}
\label{eq:equi-J}
J(\lambda) + \frac{1}{\i - \lambda}J'(\lambda) = 0.
\end{equation}
As mentioned above, we are looking for roots of the form $\lambda_k = \alpha_{k} + \eps_k$, where $\eps_k$ is a complex number of small modulus. Recall that $\alpha_{k} \to + \infty$ as $k \to + \infty$, and thus, by taking $k$ sufficiently large, we may assume that $|\i - \alpha_{k} - \eps| \geq 1$ for all complex numbers $\eps$ with $|\eps| \leq 1$.
Then, we define functions $F_k$ and $G_k$ of the  variable $\eps$ by
\begin{equation}
F_k(\eps) \triangleq J(\alpha_k + \eps), \quad 
G_k(\eps) \triangleq J(\alpha_{k} + \eps) + \frac{1}{\i - \alpha_{k} - \eps} J'(\alpha_{k} + \eps).
\end{equation}
These are holomorphic on the open unit disk, where each $F_k$ has a single simple zero, $0$. Let
\begin{equation}
a_k \triangleq \sup_{\eps \in \Co,\ |\eps| = \frac{3}{4}} \mleft | \frac{1}{\i - \alpha_{k} - \eps} J'(\alpha_{k} + \eps)  \mright | =  \sup_{\eps \in \Co,\ |\eps| = \frac{3}{4}} |G_k(\eps) - F_k(\eps)|,  \quad b_k \triangleq \inf_{\eps \in \Co,\ |\eps| = \frac{3}{4}} {|F_k(\eps)|}.
\end{equation}
Then, using \cref{eq:mcmahon}, boundedness of $J'$ near the real axis and the lower bound \cref{eq:lower-bound-Jk} established at the previous step, we see that $a_k/b_k = O(1/\sqrt{k}) \to 0 $ as $k \to + \infty$. As a result, there exists $K \in \N$ such that
\begin{equation}
\label{eq:sin}
|F_k(\eps) - G_k(\eps)| < |G_k(\eps)|, \quad  |\eps| = \frac{3}{4}, \quad k \geq K.
\end{equation}
% \begin{equation}
% \frac{a_k}{b_k} = O\mleft ( \frac{1}{\sqrt{k}} \mright), \
% \end{equation}
Then, Rouch\'e's theorem guarantees that, for all $k \geq K$, $F_k$ possesses a unique zero $\eps_k$, which is simple, in the open disk of radius $3/4$; by \cref{eq:equi-J}, $\lambda_k \triangleq \alpha_k + \eps_k$ is indeed a complex root of \cref{eq:bessel-roots-0}.
% Each $F_k$ is meromorphic, and holomorphic if restricted to a sufficiently small open neighborhood of the origin. 
% Recall that, for all $n \in \N$ and $z \in \Co$,  $|J_n(z)| \leq e^{\im z}$, and since $J'_0(z) = - J_1(z)$, there exists $M > 0$ such that $|J'_0(\alpha_{0k} + \eps)| \leq M$ for all $k \gg 1$ and $|\eps| < 1$. As a result,
% \begin{equation}
% \sup_{\eps \in \Co, |\eps| < 1} \mleft | \frac{1}{\i - \alpha_{0k} - \eps} J'_0(\alpha_{0k} + \eps)  \mright | = O\mleft ( \frac{1}{k} \mright), \quad k \to + \infty.
% \end{equation}
% Then, Rouch\'e's theorem

\emph{Step 3: asymptotic behavior of roots.}  
Let us now derive finer asymptotics for the sequence $\{\lambda_k\}_{k \gg 1}$. By \cref{eq:hankel,eq:hankel-prime}, since $\lambda_k = \alpha_k + \eps_k$ are roots of \cref{eq:bessel-roots-0} they must satisfy
\begin{multline}
\label{eq:big-expansion}
p(\alpha_k + \eps_k) \cos\mleft( \alpha_k + \eps_k - \frac{\pi}{4} \mright) - q(\alpha_k + \eps_k) \sin \mleft ( \alpha_k + \eps_k - \frac{\pi}{4} \mright) \\ =  \frac{1}{\i - \alpha_k - \eps_k} \mleft ( f(\alpha_k + \eps_k) \sin \mleft( \alpha_k + \eps_k - \frac{\pi}{4} \mright) + g(\alpha_k + \eps_k) \cos \mleft( \alpha_k + \eps_k - \frac{\pi}{4}\mright)\mright).
\end{multline}
To ease the notation, let $\varphi_k \triangleq k\pi - \pi/4$. In what follows we shall use McMahon's expansion \cref{eq:mcmahon} at the order $O(1/k^3)$. First, similarly as in \cref{eq:cos-sin-eps}, when $k \to + \infty$,
\begin{equation}
\label{eq:cos}
\cos\mleft(\alpha_k + \eps_k - \frac{\pi}{4}\mright) = \cos \mleft ( k \pi - \frac{\pi}{2} + \frac{1}{8\varphi_k} + \eps_k + O\mleft(\frac{1}{k^3}\mright) \mright) = (-1)^k \sin \mleft ( \eps_k + \frac{1}{8\varphi_k} + O\mleft(\frac{1}{k^3}\mright )  \mright),
\end{equation}
We claim that $\eps_k \to 0$ as $k \to + \infty$. Indeed, by \cref{eq:prop-p-q},
\begin{equation}
\mbox{Left-hand side of \cref{eq:big-expansion}}~= \cos\mleft(\alpha_k + \eps_k - \frac{\pi}{4}\mright) + O\mleft(\frac{1}{k} \mright), \quad k \to + \infty.
\end{equation}
But we also see that the right-hand side of \cref{eq:big-expansion} converges to $0$. Since for large $k$, the argument of the sinus in \cref{eq:cos} lies within the unit disk, that convergence is possible only if $\eps_k \to 0$ as $k +\infty$, which proves our claim. In fact, the right-hand side of \cref{eq:big-expansion} is of size $O(1/k)$; and since Taylor expansions in $\eps_k$ are now possible, \cref{eq:cos} reveals in turn that $\eps_k = O(1/k)$ as well.
Then,
% We may now use Taylor expansion in $\eps_k$ and obtain
\begin{equation}
\sin \mleft ( \eps_k + \frac{1}{8\varphi_k} + O\mleft(\frac{1}{k^3}\mright )  \mright) = \eps_k + \frac{1}{8\varphi_k} + O\mleft(\frac{1}{k^3}\mright), \quad k \to + \infty.
\end{equation}
On the other hand
\begin{equation}
% \begin{aligned}
\sin\mleft ( \alpha_k + \eps_k - \frac{\pi}{4} \mright) =  (-1)^{k+1} \cos \mleft ( \eps_k + \frac{1}{8\varphi_k} + O\mleft(\frac{1}{k^3}\mright) \mright) = (-1)^{k+1} + O\mleft (\frac{1}{k^2}\mright), \quad k \to + \infty,
% \\
% & = (-1)^{k+1} \mleft ( 1 - \frac{\eps^2_k}{2} - \frac{1}{2(8\varphi_k)^2} - \frac{\eps_k}{8\varphi_k} + O\mleft(\frac{1}{k^3} \mright) \mright), \quad k \to + \infty,
% \end{aligned}
\end{equation}
and
\begin{equation}
% \begin{aligned}
\frac{1}{\i - \alpha_k - \eps_k} 
% = \frac{1}{\i - \varphi_k} \frac{1}{1 - \frac{\eps_k}{\i - \varphi_k}- \frac{1}{8\varphi_k(\i - \varphi_k)}+ O\mleft(\frac{1}{k^4}\mright)} \\
% &=\frac{1}{\i - \varphi_k} \mleft (1 + \frac{\eps_k}{\i - \varphi_k} + \frac{1}{8\varphi_k (\i - \varphi_k)} + O\mleft(\frac{1}{k^2}\mright) \mright) \\
= \frac{1}{\i - \varphi_k} + O\mleft(\frac{1}{k^3}\mright) 
,\quad k \to + \infty.
% \end{aligned}
\end{equation}
Our goal now is to expand \cref{eq:big-expansion} at the order $O(1/k^3)$. Keeping in mind the order of each term in our previous developments, we observe that the following asymptotic estimates, which result from \cref{eq:prop-p-q,eq:prop-f-g}, are good enough: as $k \to + \infty$,
\begin{subequations}
\begin{align}
&p(\alpha_k + \eps_k) = 1 + O\mleft(\frac{1}{k^2}\mright), && q(\alpha_k + \eps_k) 
% = - \frac{1}{8(\alpha_k + \eps_k)} + O\mleft(\frac{1}{k^3}\mright)
= - \frac{1}{8\varphi_k} + O\mleft (\frac{1}{k^3} \mright), \\
% & q(\alpha_k + \eps_k) = - \frac{1}{8(\alpha_k + \eps_k)} + O\mleft(\frac{1}{k^3}\mright) = - \frac{1}{8\varphi_k} + O\mleft (\frac{1}{k^3} \mright), \\
& f(\alpha_k + \eps_k) = 1 + O\mleft(\frac{1}{k^2}\mright), && g(\alpha_k + \eps_k) = O\mleft( \frac{1}{k}\mright).
% \\
% &g(\alpha_k + \eps_k) = \frac{3}{8\varphi_k} + O\mleft(\frac{1}{k^3}\mright).
\end{align}
\end{subequations}
By carefully plugging the above computations into \cref{eq:big-expansion} we finally get
\begin{equation}
(-1)^k \eps_k 
=
% \end{equation}
% along with
% \begin{equation}
% \mbox{Right-hand side of \cref{eq:big-expansion}}~=
\frac{(-1)^{k+1}}{\i - \varphi_k}  + O\mleft (\frac{1}{k^3} \mright), \quad k \to + \infty,
\end{equation}
which shows that
\begin{equation}
\eps_k = - \frac{1}{\i - \varphi_k} + O\mleft (\frac{1}{k^3} \mright) = \frac{\varphi_k}{1 + \varphi_k^2} + \frac{\i}{1 + \varphi_k^2} +   O\mleft (\frac{1}{k^3} \mright), \quad k \to + \infty,
\end{equation}
as required.
\end{proof}
% \Cref{prop:roots-bessel} leads to a lower bound

\begin{coro}
\label{coro:qm-bessel}
 For all sufficiently large integers $k$, let $\lambda_k$ be the complex root of \cref{eq:bessel-roots-0} given by \cref{prop:roots-bessel}. Then $z_k \triangleq \i \lambda_k$ is an eigenvalue of $P$. Furthermore, writing  $\omega_k \triangleq \im z_k$, and taking any sequence of normalized eigenvectors $u_k$ associated with $z_k$, we have
 \begin{equation}
\|P(\i \omega_k)u_k\|_{\L(H)} = O(\omega_k^{-1}), \quad k \to + \infty.
 \end{equation}
 % Then, if $u_n$ is an eigenvector of $P_\uD$ associated with $z_n$, normalized so that  $\|u_n\|_{H_\uD} = 1$, we have
\end{coro}
\begin{proof}
The first statement immediately follows from \cref{lem:bessel-eg}. We have $\omega_k = \re \lambda_k > 0$ and $\sigma_k \triangleq \re z_k = - \im \lambda_k < 0$. Then, taking normalized eigenvectors $u_k$,
\begin{equation}
\label{eq:dev-P}
0 = P(z_k)u_k = (A + z_k BB^\ast + z_k^2)u_k = (A + \i \omega_k BB^\ast - \omega_k^2)u_k  + ( \sigma_k BB^\ast + \sigma_k^2 + 2 \i \sigma_k \omega_k)u_k,
\end{equation}
and recognizing here the expression of $P(\i \omega_k)$ we deduce that
% since $BB^\ast \in \L(H)$,
\begin{equation}
\|P(\i \omega_k)u_k\|_H = O(|\sigma_k|) + O(\sigma_k^2) + O (|\sigma_k|\omega_k), \quad k \to + \infty.
\end{equation}
Owing to the asymptotics of \cref{prop:roots-bessel}, $\sigma_k = O(1/k^2)$ and $\omega_k \sim k \pi$, and the result follows.
\end{proof}
\begin{proof}[Proof of \cref{th:intro-sharp}]
 If for some $\eps > 0$, the supremum in \cref{eq:op-sup} were finite for all data $(u_0, u_1) \in \dom(A) \times V$ then \cite[Proposition 2.4]{AnaLea14}, based on \cite[Theorem 2.4]{BorTom10}, would give 
\begin{equation}
\|P(\i \omega)^{-1}\|_{\L(H)} = O(|\omega|^{1-\eps}), \quad \omega \to \pm \infty,
\end{equation}
which contradicts \cref{coro:qm-bessel}.
\end{proof}

\section{Concluding remarks}
\label{sec:num}

We conclude our article by giving some comments and perspectives.
\begin{itemize}
    \item In our work the whole dynamic boundary is damped. As a possible direction for future work one could investigate \emph{localized} damping, of the form $b(x) \partial_t u$ in \cref{eq:IBVP}, similarly as in \cite{Buffe17}. With a (nonnegative, sufficiently smooth) control function $b$ supported on some subset of $\Gamma_0$ that geometrically controls $\Omega$, it is reasonable to expect that the energy decay rate $1/t$ of \cref{th:poly-stab} remains valid.

    \item We have shown that the geometric control condition is essentially a sufficient condition for the stabilization rate $1/t$. With our approach we should also be able to establish slower rates when the geometric control conditions fails to hold but \emph{non-uniform} Hautus tests are available nonetheless, for instance when the Schr\"odinger equation in $\Omega$ with homogeneous Dirichlet data is exactly observable from its Neumann trace on $\Gamma_0$. In this more general setting however, there are reasons to believe that our black box observability argument would lead to non-optimal results; see the brief discussion regarding the result of \cite{NicLao10} in our introduction. Whether or not the geometric control condition is necessary for achieving the energy decay rate $1/t$ remains an open question.

    \item An alternative approach to establishing \cref{th:poly-stab} would be to prove that the Schr\"odinger equation with first-order (in time) Wentzell boundary condition
    \begin{subequations}
    \begin{align}
    &(\i \partial_t - \Delta) \Psi = 0 &&\mbox{in}~\Omega \times (0, T), \\
    &(\i \partial_t - \Delta_\Gamma) \Psi = - \partial_\nu \Psi &&\mbox{on}~\Gamma_0 \times (0, T), \\
    &\Psi = 0&&\mbox{on}~\Gamma_1 \times (0, T),
    \end{align}
    \end{subequations}
    is exactly observable for some $T > 0$ (with respect to the $L^2(\Omega) \times L^2(\Gamma_0)$-energy) from its Dirichlet trace on $\Gamma_0$. We do not know if this property is equivalent to the stabilization rate $1/t$ for \cref{eq:IBVP}.

    \item \Cref{th:intro-sharp} also provides a non-controllability result in the disk case $\Omega = \mathbb{D}$. 
    % non-controllability counterpart.
    Let us consider the ``open-loop'' system associated with \cref{eq:IBVP}:
    % system arising from \cref{eq:wave-equation}--\cref{eq:HBC} by removing the dissipative term $\partial_t$ on the boundary:
\begin{subequations} 
\label{eq:IBVP cons}
\begin{align}
&(\partial_t^2 - \Delta)u = 0&&\mbox{in}~\Omega \times (0,T), \\
&(\partial_t^2 - \Delta_\Gamma)u = - \partial_\nu u + f&&\mbox{on}~\Gamma \times (0,T),
% &u = 0 &&\mbox{on}~\Gamma_1 \times (0,T),
\end{align}
\end{subequations}
where the control $f$ is taken in $L^2(\Gamma \times (0, T))$ and the state space is still $\H$. Without control ($f = 0$), \cref{eq:IBVP cons} is conservative: the energy $E(u, t)$ of each solution remains constant. Even when $T$ is taken large,
% and $\Gamma_0$ satisfies the geometric control condition, 
this system fails to be \emph{exactly} controllable; otherwise, following Lions' Hilbert uniqueness method, the observability inequality
% Despite GCC being satisfied, system~\eqref{eq:IBVP cons} fails to be exactly controllable from~$\Gamma_0$. Indeed, if exact controllability were true, the corresponding observability inequality
\begin{equation}\label{eq:obs-gamma_0}
E(u,0) \leq C \int_0^T \!\! \int_{\Gamma} |\partial_t u|^2\, \d\sigma\, \d t,
% \qquad \text{ for some } T>0,
\end{equation}
would hold for some $C>0$ and all initial data. In the context of second-order systems of the form \cref{eq:IBVP-abstract}, it is known that such a uniform observability estimate for the uncontrolled problem implies exponential stability of the feedback system \cref{eq:IBVP},\footnote{
This is sometimes referred to as \emph{Russell's principle}; see also \cite{Har89,AmmTuc01,ChiPau23}.}
which is indeed not possible by \cref{th:intro-sharp}. 

\item On the other hand, our analysis does not reveal any \emph{a priori} obstruction to exact controllability or uniform stabilization by means of control or damping on the non-dynamic boundary $\Gamma_1$.
% But, there is nothing that prevents exact controllability (uniform stabilization) to be true by acting on $\Gamma_1$. 
For instance, consider the particular case
% in the particular case
of a two-dimensional annular domain $\Omega = \{ x \in \R^2 : R_1 < |x | < R_2\}$, $0 < R_1 < R_2$,
% $\Omega = A(R_0,R_1) := B(R_1)\setminus \overline{B(R_0)} \subset \mathbb{R}^2$, $0 < R_0 < R_1$,
with inner and outer boundaries  $\Gamma_0 = \{ |x| = R_1\}$ and $\Gamma_1 = \{ |x| = R_2 \}$. In this particular geometry, the authors of \cite{baudouin2022unified} investigated the following system:
% the observability from $\Gamma_1$ was established for the following system:
\begin{subequations} 
\label{eq:BDEM}
\begin{align}
&(\partial_t^2 - \Delta)u = 0&&\mbox{in}~\Omega \times (0,T), \\
&(\alpha\partial_t^2 - \beta\Delta_\Gamma)u = - \partial_\nu u&&\mbox{on}~\Gamma_0 \times (0,T), \\
&u = 0 &&\mbox{on}~\Gamma_1 \times (0,T),
\end{align}
\end{subequations}
where $\alpha$ and $\beta$ are positive parameters. Here the outer boundary $\Gamma_1$ carries a homogeneous Dirichlet boundary condition and satisfies the geometric control condition.
% They proved that solutions  to \cref{eq:BDEM} are exactly observable from their velocity traces on $\Gamma_1$
% , that is
% from the outer boundary~$\Gamma_1$, that is,
% \begin{equation}\label{eq:obs-gamma_1}E(u,0) \le C \int_0^T\!\!\int_{\Gamma_1} |\partial_\nu u|^2\,d\sigma\,dt,
%  \qquad \text{ for some } T>0.\end{equation}
They proved \cite[Theorems 2.2 and 2.4]{baudouin2022unified} that exact observability of solutions to \cref{eq:BDEM} via their velocity traces on the \emph{static} boundary $\Gamma_1$ holds when
when $0<\alpha<\beta$ but \emph{not} when $0<\beta<\alpha$. The case 
$\alpha=\beta$ is open.
% dz
% Observe that in this setting, the exterior boundary of the annulus clearly satisfies GCC. This situation 
These findings
constitute further evidence that, in models such as \cref{eq:IBVP,eq:BDEM}, the dynamic nature of the boundary coupling is the main limiting factor for stabilization and observation. 

% They furthermore suggest that the presence of different propagation speeds in the interior and on the boundary might play 

% They furthermore suggest

% whether uniform stability is possible
% is
% absence the non-uniform nature stability discussed above is not related to a geometric defect but rather to the dynamic nature of the boundary coupling.

\end{itemize}

\appendix

\section{Tangential semiclassical pseudodifferential operators}

\label{sec:tangential}

In this section we define a class of \emph{tangential} pseudodifferential operators on $\R^d_+ \triangleq \R^{d-1} \times (0, +\infty)$ with a small parameter $h > 0$. We will mostly present or adapt results from \cite{Zwo12book,DyaZwo19book,RouLeb22a}. We also wish to point out that supplementary material from \cite{Buffe17,Fil24} has been helpful to the authors.

Here $y = (y_1, \dots, y_{d-1}, y_d) = (y_\tau, y_d) \in \R^d_+$ is the space variable, $y_\tau$ and $y_d$ being the tangential and normal components, respectively, and $\xi_\tau = (\xi_1, \dots, \xi_{d-1}) \in \R^{d-1}$ is the tangential dual variable. 
We fix a positive constant $h_0$ and consider a small parameter $0 <h < h_0$.

First we introduce some semiclassical symbols and pseudodifferential operators on $\R^{d-1}$. For $m \in \Z$, the symbol class $S_h^m(\R^{d-1} \times \R^{d-1})$ is the set of smooth functions $b : \R^{d-1} \times \R^{d-1} \times (0, h_0) \to \Co$ satisfying, for all multiindices $(\alpha, \beta) \in \N^{d-1} \times \N^{d-1}$,
\begin{equation}
\sup_{0 < h < h_0} \sup_{
% \substack{y_\tau \in \R^{d-1} \\ \xi_\tau \in \R^{d-1}}
(y_\tau, \xi_\tau) \in \R^{d-1} \times \R^{d-1}
} 
(1 + |\xi_\tau|^2)^{(|\beta| - m)/2} | \partial_{y_\tau}^\alpha \partial_{\xi_\tau}^\beta b(y_\tau, \xi_\tau; h)| < + \infty.
\end{equation}
We may associate to each symbol $b \in S_h(\R^{d-1}\times\R^{d-1})$ a semiclassical pseudodifferential operator $\Oph(b)$ via the standard quantization procedure:
\begin{equation}
\label{eq:oph}
(\Oph(b) \varphi)(y_\tau) \triangleq  \frac{1}{(2\pi h)^{d-1}}\iint_{\R^{d-1} \times \R^{d-1}} e^{\i h^{-1}\langle y_\tau - y'_\tau, \xi_\tau \rangle} b(y_\tau, \xi_\tau ; h) \varphi(y'_\tau) \, \d y'_\tau \, \d \xi_\tau, \quad y_\tau \in \R^{d-1},
\end{equation}
where $\varphi$ is taken in the Schwartz class $\S(\R^{d-1})$. Such operators are linear and continuous on $\S(\R^{d-1})$ and also $\S'(\R^{d-1})$ (tempered distributions); see \cite[Theorem 4.16]{Zwo12book}.
% They also enjoy various continuity properties on (semiclassical) Sobolev spaces; those are stated when needed throughout our analysis.
In \cref{eq:oph} the integral may be viewed as an \emph{oscillatory integral}, or rewritten as
% \begin{equation}
$
(\Oph(b) \varphi)(y_\tau) = \F_h^{-1}(b(y_\tau, \cdot; h) \F_h \varphi)(y_\tau)
$,
where $\F_h$ denotes the semiclassical Fourier transform;
% % \end{equation
for details regarding convergence, the reader is referred to \cite[Chapters 3 and 4]{Zwo12book} or \cite[Section 2.4]{RouLeb22a}. The resulting space of semiclassical pseudodifferential operators is denoted by $\Psi_h^m(\R^{d-1})$.
% \begin{defi}[Tangential symbols] 
% Let us  start with tangential symbols.

Now,
the class of tangential symbols $S_{\tau,h}^m(\overline{\R^d_+} \times \R^{d-1})$, $m  \in \Z$, is the set of smooth functions \smash{$a : \overline{\R^d_+} \times \R^{d-1} \times (0, h_0) \to \Co$} satisfying, for all multiindices $(\alpha, \beta) \in \N^{d} \times \N^{d-1}$, 
\begin{equation}
\sup_{0 < h < h_0} \sup_{
(y, \xi_\tau) \in \overline{\R^d_+} \times \R^{d-1}
% \substack{y \in \R^{d-1} \times (0, +\infty) \\ \xi_\tau \in \R^{d-1}}
} 
(1 + |\xi_\tau|^2)^{(|\beta| - m)/2} | \partial_y^\alpha \partial_{\xi_\tau}^\beta a(y, \xi_\tau; h)| < + \infty.
\end{equation}
% \end{defi}
% \begin{rem}
% As pointed out in \cite{RouLeb22a}, if a symbol $a$ defined on the larger domain $\R^{d} \times \R^{d-1}$
% \end{rem}
We then introduce a quantization procedure $\Opht$ for symbols in $S^m_{\tau,h}(\overline{\R^d_+} \times \R^{d-1})$. Let $\overline{\S}(\R^d_+)$ be the space of restrictions to \smash{$\overline{\R^d_+}$} of functions in $\S(\R^d)$. Given \smash{$a \in S^m_{\tau, h}(\overline{\R^d_+} \times \R^{d-1})$} and $\varphi \in \overline{\S}(\R^d_+)$, we let
\begin{equation}
\label{eq:opth}
(\Opht(a)\varphi)(y) \triangleq  \frac{1}{(2\pi h)^{d-1}}\iint_{\R^{d-1} \times \R^{d-1}} e^{\i h^{-1}\langle y_\tau - y'_\tau, \xi_\tau \rangle} a(y, \xi_\tau ; h) \varphi(y'_\tau, y_{d})  \, \d y'_\tau \, \d \xi_\tau
\end{equation}
for all $y = (y_\tau, y_d) \in \overline{\R^d_+}$.
% Similarly as before, \cref{eq:opth} can be understood as an oscillatory integral or, more simply for convergence, an expression involving the Fourier transform of $\varphi(\cdot, y_d)$. 
Since $\varphi$ is taken in  $\overline{\S}(\R^d_+)$, the integrand in \cref{eq:opth}  as well as  all its derivatives with respect to $y$ are absolutely integrable, and thus $\Opht(a)\varphi$ is well-defined in \smash{$\C^\infty(\overline{\R^d_+})$} by standard Lebesgue integration theory.
These  tangential  operators constitute a space which we denote by \smash{$\Psi_{\tau,h}^m(\overline{\R^d_+})$}. Additional results can be found in \cite[Section 2.10]{RouLeb22a}.

The subspace of $\Psi_{\tau,h}^m(\overline{\R^d_+})$ (resp. $\Psi_h^m (\R^{d-1}))$ for which $a$ (resp. $b$) is a polynomial of order at most $m \geq 0$ in $\xi_\tau$ with coefficients varying in $y$ (resp. $y_\tau$) is denoted by \smash{$\diff_{\tau,h}^m(\overline{\R^{d}_+})$} (resp. \smash{$\diff^m_h(\R^{d-1})$}; those are spaces of (semiclassical) \emph{differential} operators.

All these operators, defined here for smooth functions,
possess  continuous extension on some natural Sobolev spaces. For the reader's convenience, such properties are given precise statements  when needed throughout the manuscript; again, see also \cite{Zwo12book,DyaZwo19book,RouLeb22a}.
When there is no risk of confusion, we shall drop $\R^{d}_+$ and $\R^{d-1}$ from our notation and simply write $S^m_{\tau,h}$, $\Psi_{\tau, h}$, and so on.

\begin{rem}
\label{rem:small-h}
We use a small parameter $h > 0$ as in \cite{Zwo12book,DyaZwo19book} whereas \cite{RouLeb22a} uses a large parameter $\lambda = 1/h$. The operator classes under consideration are essentially the same but the scaling and notation are different. For instance, if $\Delta_\tau$ is the Laplacian on $\R^{d-1}$ then, as a differential operator with large parameter in the terminology of \cite{RouLeb22a}, $- \Delta_\tau - \lambda^2$ has principal symbol $|\xi_\tau|^2 - \lambda^2$; it is related to the semiclassical operator $- h^2 \Delta_\tau - 1$ by
\begin{equation}
- \Delta_\tau - \lambda^2 = \operatorname{Op}(|\xi_\tau|^2 - \lambda^2) = \lambda^2\Oph(|\xi_\tau|^2 - 1) = \lambda^2 (- h^2 \Delta_\tau - 1),
\end{equation}
where $\operatorname{Op}$ denotes the non-semiclassical standard quantization on $\R^{d-1}$, that is, \cref{eq:oph} with $h = 1$. The reader is referred to \cite[Section 2.11]{RouLeb22a} for further details regarding the correspondence.
\end{rem}

We conclude this section with a positivity inequality adapted  from \cite[Theorem 2.50]{RouLeb22a}.
% \begin{rem}
% The space $\Psi_{\tau,h}^m(\R^d_+)$ is not a subset of $\Psi_h^m(\R^d_+)$, the space
% \end{rem}
% \edn{[En cours de construction]}
\begin{theo}[Microlocal tangential G\aa rding inequality]
\label{th:garding}
% Let $m \in \N$ and $P = \Oph(p) \in \operatorname{Diff}_{\tau,h}^m$. 
Consider
\begin{itemize}
    \item A tangential semiclassical differential operator  $P = \Oph(p) \in \diff^ m_{\tau,h}$, where $m \geq 0$ is even;
    \item A compact subset $C$ of \smash{$\overline{\R_+^d}$} and  an open subset $U$ of $C \times \R^{d-1}$;
    % such that $(y, r \xi_\tau) \in U$ for all $(y, \xi_\tau) \in U$ and $r > 0$ (i.e., $U$ is \emph{conic} in the $\xi_\tau$-variable);
    \item A tangential symbol $\chi \in S^0_{\tau, h}$ that is $h$-independent and satisfies $\supp(\chi) \subset U$.
    % and such that, for some $\xi_0 > 0$, $\xi(x; r \xi_\tau)$
\end{itemize}
Suppose that there exists $c> 0$ such that
\begin{equation}
\re p(y, \xi_\tau) \geq c |\xi_\tau|^m, \quad (y, \xi_\tau) \in U.
% , \quad |\xi_\tau| \geq R
\end{equation}
Then, there exist $K, K' > 0$ and $0 < h_1 < h_0$ such that, writing $\Chi \triangleq \Opht(\chi)$,
    \begin{equation}
    \label{eq:ann-garding}
    \iint_{\R^d_+} P\Chi \varphi \overline{\Chi \varphi} \, \d y_d \, \d y_\tau \geq K
    % \mleft 
    % (
    \iint_{\R^d_+} 
    % |\Chi \varphi|^2 +
    \sum_{
    \substack{
    \alpha \in \N^{d - 1} \\ |\alpha|\leq m/2 
    }
    } h^{2|\alpha|}|D^{(\alpha,0)}\Chi \varphi|^2 \, \d y_d \, \d y_\tau
    % (h^2 \|\Chi \varphi\|^2_{L^2(0, +\infty; H^1(\R^{d-1}))} + \|\Chi \varphi\|^2_{L^2(\R^d_+)}    
    % \mright) 
    - K'{{h^m}} 
    % \|\varphi\|^2_{L^2(\R^d_+)}
    \iint_{\R^d_+} |\varphi|^2 \, \d y_d \, \d y_\tau
    \end{equation}
% $P \in \Psi^m_{\tau,h}$
for all $\varphi \in \overline{\S}(\R^{d}_+)$ and all $0 < h < h_1$.
\end{theo}
\begin{proof}
This is a particular case of \cite[Theorem 2.50]{RouLeb22a} reformulated in the setting of a small parameter $h$ instead of a large parameter $\lambda$; see \cref{rem:small-h}. Note that the homogeneity hypotheses of that theorem are automatically satisfied due to $P$ being a differential operator and $\chi$ being $h$-independent.
\end{proof}

\section{The geometric control condition}
\label{sec:gcc}

In this section we will give a precise statement of the \emph{geometric control condition} (``GCC'') and briefly discuss its
% As a preliminary tool for the study of energy decay in wave systems, we recall the geometric control condition (GCC) and 
central role in observation and control of wave equations. The exposition here is inspired by \cite{Melrose78,BarLeb92,le2017geometric,kleinhenz2023sharp}.

Wave equations possess solutions that are localized near some particular curves
% $(t,x(t))
% $
in space-time called \emph{rays}. 
Rays follow the laws of geometric optics and in the case of a one-dimensional spatial domain  they coincide with characteristic curves.
% ; for example, in one dimension they coincide with the characteristic curves.
To effectively observe or control such solutions, it is necessary that every ray intersects the observation or damping region. The use of rays for the stability analysis of damped hyperbolic systems
% to describe observability phenomena originates in the work 
originates in the work
of Rauch and Taylor \cite{Rauch1974}, which considers boundaryless manifolds and builds upon
% for domains without boundary,
% starting from 
H\"ormander's theorem on the propagation of singularities. In the case of  boundary observation and damping, the theory was subsequently developed by Bardos, Lebeau, and Rauch \cite{BarLeb92}.

Recall that $\Omega$ is a smooth bounded domain of $\R^d$ and let ${M} \triangleq \R \times \overline{\Omega}$ be the associated space-time manifold. 
% \edn{[here we might want to take $M = \R \times \overline{\Omega}$]}
% , with boundary $\partial\mathcal{M}=\partial\Omega\times\R_{t}$,  where $\Omega\subset\R^{d}$ is a smooth bounded domain.
For simplicity of  exposition and since we are interested in the pure wave equation \cref{eq:wave-equation}, we will only consider the wave operator $P=\partial_{t}^{2}-\Delta$, whose principal symbol is $p= - \omega^{2}+|\xi|^{2}$; here, $p = p(t, x, \omega, \xi)$ is a smooth real function of the
% $\omega$ and $\xi$ are the dual time and space variables and live in the 
cotangent bundle $T^\ast M$.
% \edn{[or the whole phase space ?]}. 
The \emph{Hamiltonian vector} field on $T^\ast M$ associated with $p$ is $H_p= - 2\omega \partial_{t}+ 2\xi\cdot\partial_{x}$. A \emph{bicharacteristic} of $P$ is an integral curve of the Hamiltonian vector field $H_p$ in $T^\ast M$ that is contained in the characteristic set 
\begin{equation}
\cha(p) \triangleq \{ (t, x, \omega, \xi) \in T^\ast M : p(t, x, \omega, \xi) = 0, \, (\omega, \xi) \not = 0\}.
\end{equation}
Equivalently, it is a maximal solution 
$s \mapsto (t(s),x(s),\omega(s),\xi(s))$
of the Hamilton–Jacobi system
\begin{equation}
\frac{\d}{\d s} (t, x , \omega, \xi) = (\partial_\omega p, \partial_\xi p, - \partial_t p, - \partial_x p),
\end{equation}
which in our case reduces to
\begin{equation}
\frac{\d}{\d s} (t, x , \omega, \xi) = (- 2\omega ,  2 \xi, 0, 0).
\end{equation}
% \[\dot{t}=\partial_\tau p=2\tau,\qquad 
% \dot{x}=\partial_\xi p=-2\xi,\qquad 
% \dot{\tau}=-\partial_t p=0, \qquad \dot{\xi}=-\,\partial_x p=0.\]
% %\[\begin{cases}\dot{t}(s) = \partial_\tau p = 2\tau,\\[2mm]
% %\dot{x}(s) = \partial_\xi p(t,x,\tau,\xi) = -\,2\xi,\\[2mm]
% %\dot{\tau}(s) = -\,\partial_t p(t,x,\tau,%\xi) = 0,\\[2mm]
% %\dot{\xi}(s) = -\,\partial_x p(t,x,\tau,\xi) = 0,
% %\end{cases}\]
% where the dot denotes differentiation with respect to the Hamiltonian parameter $s$.  
Hence, for the wave operator, $\omega$ and $\xi$ are constant along the flow while $t$ and $x$ evolve linearly. A ray is the projection
of a bicharacteristic onto the $(t,x)$-space.

To define generalized bicharacteristics we  follow
\cite{Melrose78}; 
see also \cite[Chapter~24]{hormander2007III}, \cite[Section 1C1]{le2017geometric} and \cite[Appendix A.1]{kleinhenz2023sharp}.  Near the boundary $\Gamma = \partial\Omega$, using normal geodesic coordinates $(x_\tau,x_d)$, the principal symbol of 
% the Laplacian (in space) 
$-\Delta$ reads as $\xi_d^{2} + \ell(x,\xi_\tau)$, 
% \edn{[symbole de $\Delta$ ou $-\Delta$?]}\edp{$-\Delta$} 
{where $\ell$ is a suitable elliptic tangential symbol}; see  \cref{sec:normal-geodesic}.
% We extend those coordinates to space-time as follows: let $n \triangleq d+1$ and $y  \triangleq (t, x)$, so that $y = (y_1, \dots, y_n) \in \R^n$ with $y_1 = t$ and $y_{i+1} = x_i$ for $i = 1, \dots, d$. We shall write $y = (y_\tau, y_n)$ with tangential and normal components $y_\tau = (t, x_\tau)$ and $y_n = x_d$.
% % denote by $\eta = (\eta_\tau, \eta_n)$ the corresponding dual variable. 
% (Note that our notation here deviates from that of the rest of the manuscript.) In local $y$-coordinates,
% the boundary $\R \times \Gamma$ is given by $\{y_n = 0\}$ and parametrized by $y_\tau$,
% while the open set $\R \times \Omega$ is described by $\{y_n > 0\}$. Let $\eta = (\eta_\tau, \eta_n)$ denote the dual variable associated with $y = (y_\tau, y_n)$.  In local $(y, \eta)$-coordinates, 
% the boundary $\partial T^* M$ of the cotangent bundle $T^\ast M$ is $\{ \rho = (y, \eta) : y_n = 0\}$,
In these local coordinates $(t, x, \omega, \xi) = (t, x_\tau, x_d, \omega, \xi_\tau, \xi_d)$,
the principal symbol $p$ of the wave operator $\partial_t^2 - \Delta$ takes the form
$
p= - \omega^2 + \xi_d^2 + \ell(x, \xi_\tau)
$,
% where $r$ is a suitable elliptic \edn{[plutôt hyperbolique du coup?]} tangential symbol in the space-time tangential variable $\eta_\tau$. 
and the associated Hamiltonian vector field $H_p$ reads as
\begin{equation}
H_p = - 2 \omega \partial_t + 2 \xi_d \partial_{x_d} + (\partial_{\xi_\tau} \ell) \cdot \partial_{x_\tau} - (\partial_{x} \ell) \cdot \partial_\xi.
% (\partial_{\eta_\tau} r)\,\partial_{y_\tau} + 2\eta_n\,\partial_{y_n}
%        - (\partial_{y_\tau} r)\,\partial_{\eta_\tau} - (\partial_{y_n} r)\,\partial_{\eta_n}.
\end{equation}
% In local $y$-coordinates,
% the boundary $\R \times \Gamma$ is defined by $y_n = 0$ and parametrized by $y_\tau$,
% while the open set $\R \times \Omega$ is described by $y_n > 0$.
% %The variables $\eta = (\eta',\eta_n)$ are the cotangent variables associated with $y = (y',y_n)$. 
We will need to define the \emph{order of contact} of points at the boundary. Note that in local coordinates the boundary $\partial T^* M$ of the cotangent bundle $T^\ast M$ is described by  $\{y_n = 0\}$, and let $\Sigma \triangleq \cha(p) \cap \partial T^\ast M$.
% , while the cotangent bundle $T^\ast \partial M$ of the boundary $\partial M$ is given by $\{ y_d = 0,~ \xi_d = 0\}$. \edn{vrai ça ?}
\begin{defi}[Order of contact with the boundary]
\label{def:order-contact}
Let $q = (t, x, \omega, \xi) \in \Sigma$, and let  $f$ be the (smooth, boundary-defining) function defined in local coordinates near $(t, x)$ by $y_d$.
% , that is, 
% $\rho>0$ in $\mathrm{int}(M)$, $\rho=0$ on $\partial M$, and $\nabla\rho\neq 0$ on $\partial M$.
% At a point $(t,x,\omega,\xi)\in \cha(p)\cap T^\ast \partial M$ \edn{$T^\ast \partial M$ ou $\partial T^\ast M$?}, 
The \emph{order of contact} of bicharacteristics  with the boundary $\partial M$ at $q$ is the smallest integer $k\geq1$ such that
% \edn{Here $\cha(p)$ doesn't include boundary points?}
\begin{equation}
(H_p^jf)(q) = 0\quad \mbox{for}~ j=0,\dots,k-1,
\quad \mbox{and}~(H_p^kf)(q)\neq0.
\end{equation}
% (Here the superscript indicates derivatives of the vector field $H_p$.)
If $(H_p^jf)(q)=0$ for all $j$, the contact is said to be of \emph{infinite order}.
\end{defi}

In what follows we make repeated use of the previous local coordinates. First, define a symbol $r$ on $\partial T^\ast M$ by 
\begin{equation}
r(t, x_\tau, \omega, \xi_\tau) = - \omega^2 + \ell(0, x_\tau, \xi_\tau),
\end{equation}
so that $r$ represents the tangential part of $p$ restricted to $\partial T^\ast M$.
% and parametrized by 
% . Now we consider the following symbon on $\partial T^\ast M$:
% is made of all points $\rho$ satisfying $y_n = 0$.
% % precisely $\{ \rho = (y,\eta) \in T^\ast M : y_n = 0 \}$.
% % % \begin{equation}
% % % \partial T^\ast M = \{ \rho = (y,\eta) \in T^\ast M : y_n = 0 \}.
% % % \end{equation}
% % We denote by
%  consider the symbol $r_0 \triangleq r|_{\partial T^\ast M}$, i.e.,
% $r_0(y_\tau, \eta_\tau) = r(y_\tau, 0, \eta_\tau)$ in local coordinates, and 
The associated Hamiltonian vector field $H_r$ on $\partial T^\ast M$ is given  by
\begin{equation}
H_{r} = - 2 \omega \partial_t   + (\partial_{\xi_\tau} \ell) \cdot \partial_{x_\tau} - (\partial_{x_\tau} \ell) \cdot \partial_{\xi_\tau}.
   % = (\partial_{\eta_\tau} r_0)\,\partial_{y_\tau} - (\partial_{y_\tau} r_0)\,\partial_{\eta_\tau}.
\end{equation}
% $r_0$ the restriction of $r$ to $\partial T^{*}\mathcal{M}$, that is, $r_0(y_\tau,\eta_\tau) = r(y_\tau,y_n = 0,\eta_\tau)$. 
% We then introduce the Hamiltonian vector field above the submanifold $\{y_n = 0\}$, $H_{r_0}
%    = (\partial_{\eta_\tau} r_0)\,\partial_{y_\tau} - (\partial_{y_\tau} r_0)\,\partial_{\eta_\tau}$. 
% Having let $\Sigma \triangleq \cha(p) \cap \partial T^\ast M$ and 
We define the
% \Sigma_0 = \cha(p) \cap \partial T^{*}\mathcal{M}$. Using local coordinates,
\emph{glancing set} $\mathcal{G} \subset \Sigma$ by
\begin{equation}
\mathcal{G} = \mleft \{ (t, x_\tau, 0, \omega, \xi_\tau, \xi_d) \in \Sigma : \xi_d = r(t, x_\tau, \omega, \xi_\tau) = 0 \mright \}.
\end{equation}
The \emph{hyperbolic set} $\H \triangleq \Sigma \setminus \mathcal{G}$ is the complement in $\Sigma$ of the glancing set. Given 
$q = (t, x_\tau, 0, \omega, \xi_\tau, \xi_d) \in \Sigma$, observe that if $q \in \H$ then $r(t, x_\tau, \omega, \xi_\tau) = - \xi_d^2$; as a result, it is easy to see that
\begin{equation}
q \in \H ~\mbox{if and only if}~ r(t, x_\tau, \omega, \xi_\tau) < 0, \quad q \in \mathcal{G} ~\mbox{if and only if}~ r(t, x_\tau, \omega, \xi_\tau) = 0.
\end{equation}
% We define the set $\mathcal{H} = \Sigma_0 \setminus \mathcal{G}$ as the \emph{hyperbolic set}. 
% Hence, if $\rho = (y_\tau,y_n = 0,\eta_\tau,\eta_n) \in \Sigma_0$, then
% \[
% \rho \in \mathcal{H} \;\Longleftrightarrow\; r_0(y_\tau,\eta_\tau) < 0,
% \qquad
% \rho \in \mathcal{G} \;\Longleftrightarrow\; r_0(y_\tau,\eta_\tau) = 0.\]
% Now,
% recall \cref{def:order-contact} and \edn{notice that points in the hyperbolic set have order of contact $1$. On the other hand, points in the glancing set have order of contact at least $2$}
We may decompose the glancing set as 
$\mathcal{G} = \mathcal{G}_2 \supset \mathcal{G}_3 \supset \cdots \supset \mathcal{G}^\infty$, where $1 \leq k \leq +\infty$ indicates the order of contact with the boundary. The construction of generalized bicharacteristics requires that the glancing set $\mathcal{G}^2 \setminus \mathcal{G}^3$ of order exactly $2$ be decomposed into the union of a \emph{diffractive set} $\mathcal{G}^2_d$ and a \emph{gliding set} $\mathcal{G}^2_g$:
\begin{equation}
\mathcal{G}^2_d \triangleq \mleft\{ q \in \mathcal{G}^2 \setminus \mathcal{G}^3 : (H^2_p f)(q) > 0\mright\}, \quad \mathcal{G}^2_g \triangleq \mleft\{ q \in \mathcal{G}^2 \setminus \mathcal{G}^3 : (H^2_p f)(q) < 0\mright\},
\end{equation}
where $f$ is as in \cref{def:order-contact}.
% \edn{ici quand on écrit $y_n$ c qu'est-ce qu'on entend précisément?}
More generally, we consider glancing sets $\mathcal{G}^{2l} \setminus \mathcal{G}^{2l + 1}$ of even order $2l$, $l \geq 1$, and we let
\begin{equation}
\mathcal{G}^{2l}_d \triangleq \mleft\{ q \in \mathcal{G}^{2l} \setminus \mathcal{G}^{2l +1} : (H^2_p f)(q) > 0\mright\}, \quad \mathcal{G}^{2l}_g \triangleq \mleft\{ q \in \mathcal{G}^{2l} \setminus \mathcal{G}^{2l+1} : (H^2_p f)(q) < 0\mright\},
\end{equation}
and, just as in \cite{le2017geometric}, we define the set of \emph{diffractive points} as
\begin{equation}
\mathcal{G}_d \triangleq \bigcup_{l \geq 0} \mathcal{G}_d^{2l}.
\end{equation}
We are in a position to present a definition of generalized bicharacteristics. 
% The statement given here is adapted from \cite[Definition 2]{kleinhenz2023sharp}
% we write $\rho = (y_\tau,y_n = 0,\eta_\tau,\eta_n) \in \mathcal{G}_k$ if it has order of contact $k$ with the boundary.
% Finally, consider $\mathcal{G}_2\setminus \mathcal{G}_3$, the glancing set of order precisely two as the union of the diffractive set $\mathcal{G}_d^{2}$ and of the gliding set $\mathcal{G}_g^{2}$.
% \[ \rho\in \mathcal{G}_d^{2} \ \text{(resp.\ } \mathcal{G}_g^{2}) 
% \quad \Longleftrightarrow \quad 
% \rho\in \mathcal{G}_2\setminus \mathcal{G}_3 
% \ \text{and}\ 
% H_p^{2}(y_n)>0\ \text{(resp.\ }<0).
% \]
\begin{defi}[Generalized bicharacteristic]
A \emph{generalized bicharacteristic} of $p$ is a  map 
\begin{equation}
\gamma : \dom(\gamma) \subset \R \to (\cha(p)\setminus \partial T^\ast M)\cup \mathcal{G},
\end{equation}
where $\dom(\gamma)$ is open, that is differentiable and satisfies
% $\dom(\gamma) \subset \R$
% \[\gamma:\R\setminus B \longrightarrow 
%       (\cha(p)\setminus\Sigma_0)\cup \mathcal{G},\]
the following properties: 
\begin{enumerate}
\item For all $s \in \dom(\gamma)$, $\dot{\gamma}(s)=H_p(\gamma(s))$ 
      if $\gamma(s)\in \cha(p) \setminus \partial T^\ast M$
      or $\gamma(s)\in \mathcal{G}_d^{2}$;
\item For all $s \in \dom(\gamma)$, $\dot{\gamma}(s)=H_{r}(\gamma(s))$ if $\gamma(s)\in \mathcal{G}\setminus \mathcal{G}_d^{2}$;
\item For every $s_0\in \R \setminus \dom(\gamma)$, there exists $\delta > 0$ such that $s \in \dom(\gamma)$ if $|s_0 - s| < \delta$ and
\begin{equation}
\gamma(s) \in \cha (p) \setminus \partial T^\ast M, \quad 0 < |s - \delta| < \delta;
\end{equation}
furthermore, $\gamma$ has left and right limits  $q^\pm$ at $s_0$ that satisfy, in local coordinates,
    \begin{equation}
    x_d^-  = x_d^+ = 0, \quad (t^-, x_\tau^-) = (t^+, x_\tau^+), \quad (\omega^-, \xi_\tau^-) = (\omega^+, \xi_\tau^+), \quad \xi_d^- = -  \xi_d^+.
    \end{equation}
      % $\gamma(s)\in \cha(p)\setminus\Sigma_0$
      % for $s\in(s_0-\delta,s_0)\cup(s_0,s_0+\delta)$.
      % Furthermore, the limits 
      % $\lim_{s\to s_0^{\pm}}\gamma(s)=(y^{\pm},\eta^{\pm})$ exist and satisfy $y_n^{-}=y_n^{+}=0$, $y_\tau^{-}=y_\tau^{+}$, $\eta_\tau^{-}=\eta_\tau^{+}$ and $\eta_n^{-}=-\eta_n^{+}$.
\end{enumerate}
\end{defi}
% In Case~(1), the generalized bicharacteristic lies either in the interior or at a
% diffractive point; here it coincides with a segment of a classical bicharacteristic.
% Case~(2) describes how a generalized bicharacteristic enters, leaves, or glides along
% $\partial T^{*}\mathcal{M}$, while Case~(3) corresponds to reflections when the bicharacteristic
% transversally encounters the boundary.
%%%%%
\begin{figure}[ht]
\centering
\includegraphics[scale=0.57]{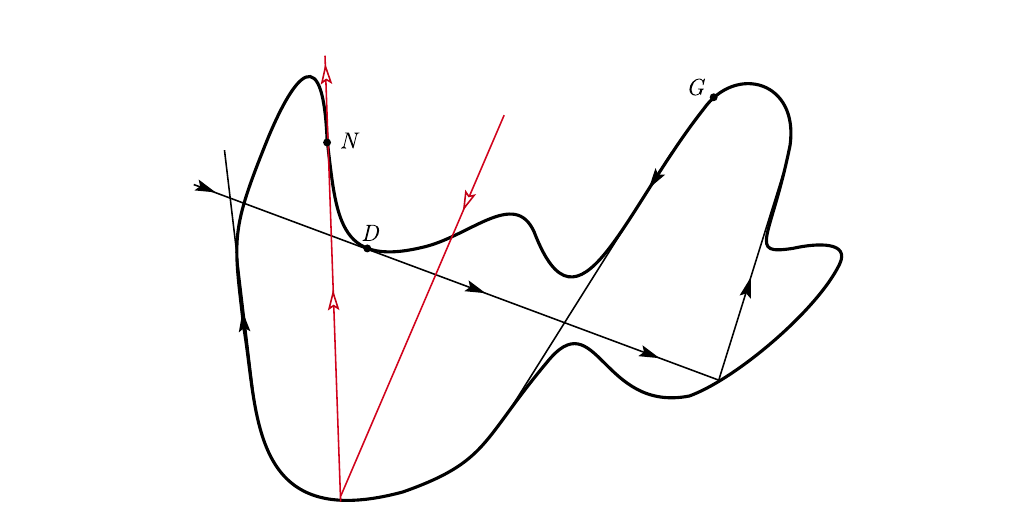}
\caption{Two generalized bicharacteristics (projected onto the $x$-space), with different type of points of contact: nondiffractive ($N$), diffractive ($D$), gliding ($G$). Adapted from \cite[Figure 1.2]{miller2002escape}.}
\label{fig:geo png}
\end{figure}

\Cref{fig:geo png} illustrates projections of generalized bicharacteristics onto the $x$-space and different type of contacts at the boundary.
Next, given a point $(t, x)$ near $\partial M$,
% For $y$ near the boundary of $\mathcal{M}$, define
let
$\prescript{b}{}{T}_{(t, x)}M$  be the tangent space spanned by $\partial_t$, $\partial_{x_1}, \dots, \partial_{x_{d-1}}$, and $x_d\partial_{x_d}$. 
% \edn{If $(t, x) \in M$ is away from the boundary?}
The \emph{compressed cotangent bundle} is 
\begin{equation}
\prescript{b}{}{T}^{*}{M} \triangleq \bigcup_{(t, x)\in{M}}(\prescript{b}{}{T}_{(t, x)}{M})^{*}
\end{equation}
% defined as $\prescript{b}{}{T}^{*}\mathcal{M} = \bigcup_{y\in\mathcal{M}}(\prescript{b}{}{T}_y\mathcal{M})^{*}$ 
and the \emph{compression map} $j:T^{*}{M}\to \prescript{b}{}{T}^{*}{M}$ is defined, using local coordinates, 
by
\begin{equation}
j(t, x_\tau, x_d, \omega, \xi_\tau, \xi_d) = j(t, x_\tau, x_d, \omega, \xi_\tau, x_d \xi_d).
\end{equation}
% $j(y,\eta_\tau,\eta_n) = (y,\eta_\tau,y_n\eta_n)$. Note that
% \begin{itemize}
% \item for $y\in\Omega\times\R$, $\prescript{b}{}{T}^{*}_y\mathcal{M}=j(T^{*}_y\mathcal{M})\simeq T^{*}_y(\Omega\times\R)$;
% \item for $y\in\partial\Omega\times\R$, $\prescript{b}{}{T}^{*}_y\mathcal{M}=j(T^{*}_y\mathcal{M})\simeq T^{*}_y(\partial\Omega\times\R)$.
% \end{itemize}
% The set of points $(y_\tau,y_n=0,\eta_\tau,0)\in \prescript{b}{}{T}^{*}\mathcal{M}|_{y_n=0}$ such that
% $r_0(y_\tau,\eta_\tau)>0$ is called the \emph{elliptic set}~$\E$. 
A \emph{compressed generalized bicharacteristic} is the image under $j$ of a
generalized bicharacteristic of~$p$. We will use the following standard hypothesis on points of contact with the boundary.
\begin{as}[No infinite-order contact]
\label{ass:finite-contact} We have $\mathcal{G}^\infty = \emptyset$.
% We assume that bicharacteristics do not exhibit infinite-order contact with the boundary.
\end{as}
% If $b\gamma=j(\gamma)$, then $b\gamma:\R\to \prescript{b}{}{T}^{*}\mathcal{M}\setminus \E$
% is a continuous curve. Using $t$ as a parameter, its projection onto $\Omega$ yields a unit–speed
% generalized geodesic that remains in~$\Omega$.
% An essential property of compressed generalized bicharacteristics, established in~\cite{Melrose78}, is their uniqueness: 
\begin{prop}[\cite{Melrose78}] Under \cref{ass:finite-contact},
 compressed generalized bicharacteristic 
are uniquely determined by any of their points.
\end{prop}
\begin{rem}
One can show that \cref{ass:finite-contact} is always verified when $\Gamma$ is real-analytic.
\end{rem}
\Cref{ass:finite-contact} guarantees that every bicharacteristic admits a unique continuation 
as a compressed generalized bicharacteristic.  
Now we have all the tools to give a proper definition of the geometric control condition.

\begin{defi}[Geometric control condition]
\label{def:gcc}
Suppose that \Cref{ass:finite-contact} holds. We say that $\Gamma_0$ satisfies the geometric control condition
% relatively open subset $\Gamma_{0}$ of $\partial\Omega$ satisfies the geometric control condition (GCC) 
if there exists $T>0$ such that, for every compressed generalized bicharacteristic $\prescript{b}{}\gamma$ of $p$, there exists $s \in \R$ such that, writing $\prescript{b}{}\gamma(s) = (t(s), x(s), \omega(s), \xi(s)) \in \prescript{b}{}T^\ast M$, we have $0 < t(s) < T$, $x(s) \in \Gamma_0$, and $\prescript{b}{}\gamma(s) \in j(\H \cup (\mathcal{G} \setminus \mathcal{G}_d))$.
% every compressed generalized bicharacteristic of $P$ intersects the set $T^{*}( \Gamma_{0}\times(0,T))$ at a nondiffractive point.
\end{defi}

\begin{rem}
In the absence of infinite-order contact with the boundary (\cref{ass:finite-contact}), the geometric control condition  \emph{characterizes}  exact observability of the wave equation with homogeneous Dirichlet boundary condition via the Neumann trace on $\Gamma_0$; see, in a more general setting, \cite{BarLeb92} for the sufficiency, and \cite{Burq97} for the necessity.
\end{rem}

% is a sufficient condition for the exact observability of the wave equation in $\Omega$
% \end{rem}

% When \Cref{ass:finite-contact} holds, Burq–Gérard  in \cite{Burq97} proved that the GCC is both necessary and sufficient for exact controllability. In particular, if the coefficients of the operator and the boundary are real-analytic, then infinite-order contact cannot occur, the classical GCC of Definition~\ref{def:gcc} provides the sharp microlocal characterization of controllability.

Another important tool used in this work is the so-called \emph{Hautus test}, 
a resolvent-type condition that characterizes exact observability 
% (and hence controllability) 
by means of frequency-domain estimates involving the system's generator;
% of the system's generator
see \cite{BurZwo04,Mil05} and also\cite[Chapter~6]{TucWei09book}.
The next lemma gives the exact observability inequality from  \cite{BarLeb92} in the form of a Hautus test that is most convenient for our analysis.
\begin{lemma}[Hautus test]
% Let $\Omega\subset\R^d$ be a bounded smooth domain, $\Gamma_{0}\subset\partial\Omega$ relatively open. Consider the wave operator $P=\partial_t^2-\Delta$ with Neumann observation on $\Gamma_0$. Assume that $\Gamma_0$ satisfies GCC. 
Suppose that \cref{as:dyn} holds, i.e., \cref{ass:finite-contact} is valid and $\Gamma_0$ satisfies the geometric control condition.
Then there exist constants $M,m>0$ such that
% , independent of $h\in (0,1]$, such that for every $w \in \dom(\Delta_D)$, 
\label{lem:hautus}
\begin{equation}
\label{eq:hautus-sc}
\|w\|_{H^1_h(\Omega)} \leq h^{-1} M \|(-h^2 \Delta - 1)w\|_{L^2(\Omega)} +  mh  \|\partial_\nu w\|_{L^2(\Gamma_0)}, \quad w \in \dom(\Delta_D), \quad 0 < h \leq 1.
\end{equation}
\end{lemma}
\begin{proof}
Let $\H_0 \triangleq H^1_0(\Omega) \times L^2(\Omega) = \dom((-\Delta_D)^{1/2} \times L^2(\Omega)$
% Let $X=H_{0}^{1}(\Omega)\times L^{2}(\Omega)$ 
endowed with its natural energy inner product. The operator $-\A_0$, where
\begin{equation}
\A_0 \triangleq \begin{pmatrix} 0 & - 1 \\ -\Delta_D & 0 \end{pmatrix}, \quad \dom(\A_0) = \dom(\Delta_D) \times H^1_0(\Omega) =  (H^2(\Omega)\cap H_0^1(\Omega))\times H_0^1(\Omega),
\end{equation}
% $A\begin{pmatrix}
% u\\v    
% \end{pmatrix}=\begin{pmatrix}
% v\\ \Delta u    
% \end{pmatrix}$, with $\dom(A)=(H^2(\Omega)\cap H_0^1(\Omega))\times H_0^1(\Omega)$ and $C\begin{pmatrix}
% u\\v    
% \end{pmatrix}=\partial_{\nu}u|_{\Gamma_0}$. Therefore, $A$ is a 
is skew-adjoint 
and generates a strongly continuous group $\{e^{-t\A_0}\}_{t \in \R}$ on $\H_0$ which is unitary and describes all finite-energy solutions to the wave equation with homogeneous Dirichlet conditions. Let $\C : \dom(\A_0) \to L^2(\Gamma_0)$ be the {observation operator} given by
\begin{equation}
\C \triangleq \begin{pmatrix}
\partial_\nu & 0
\end{pmatrix}.
\end{equation}
A standard differential multiplier argument shows that $\partial_\nu u$ 
is well-defined in $L^2(\Gamma \times (0, T))$, $T > 0$, for all finite-energy solutions to the wave equation with zero Dirichlet data; see \cite[Theorem 2.2]{Kom94book} or, alternatively, use \cref{lem:rellich}.
This means that $\C$ is \emph{admissible} for $-\A_0$ in the sense of \cite[Definition 4.3.1]{TucWei09book}. On the other hand, 
\cite[Theorem~3.8]{BarLeb92}, translated into system-theoretic language, asserts that the pair $(- \A_0,\C)$ is exactly observable. Therefore, \cite[Theorem~6.6.1]{TucWei09book} provides constants $M,m>0$ such that
% \begin{equation}
% \|(u_0, u_1)
% \end{equation}
% , this is equivalent to the following Hautus-type resolvent inequality: There exist $M_0,m_0>0$ s
% uch that, for all $\omega \in \R$ and $z=\begin{pmatrix}
% u\\v    
% \end{pmatrix}\in \dom(A)$ 
%\[\|u(0)\|_{H_0^1(\Omega)}^2+\|\partial_t u(0)\|_{L^2(\Omega)}^2 \;\le\; K_T\int_0^T\!\!\int_{\Gamma_0}|\partial_\nu u|^2\,d\sigma\,dt.\]
\begin{equation}\label{eq:Hautus-abstract}
\|z\|^{2}_{\H_0} \leq  M^{2}\,\|(\i\lambda + \A_0)z\|^{2}_{\H_0} + m^{2}\,\|\C z\|_{L^2(\Gamma_0)}, \quad z = (u_0, u_1) \in \dom(\A_0), \quad \lambda \in \R.
\end{equation}
Now, choosing $z$ of the form $z = (w, \i \lambda w)$ in
% For $w \in \dom(\Delta_{D})=H^2(\Omega)\cap H_0^1(\Omega)$, substituting $z=\begin{pmatrix}
% w\\i\omega w    
% \end{pmatrix} \in \dom(A)$ 
in \cref{eq:Hautus-abstract} yields
\begin{equation}
\label{eq:res-omega}
\|\nabla w\|_{L^2(\Omega)^d}^2+\lambda^2\|w\|_{L^2(\Omega)}^2 \leq M^2\|(-\Delta -\lambda^2) w\|_{L^2(\Omega)}^2 + m^2\|\partial_\nu w\|_{L^2(\Gamma_0)}^2, \quad w \in \dom(\Delta_D), \quad \lambda \in \R.
\end{equation}
In the semiclassical setting we take $\lambda=1/h$ with $0 < h \leq 1$, and we rewrite \cref{eq:res-omega} as
% \[-\Delta w-\omega^2 w \;=\; h^{-2}\,(-h^2\Delta-1)w.\]
% Thus
\begin{equation}
\label{eq:res-omega-h}
\|\nabla w\|_{L^2(\Omega)^d}^2+h^{-2}\|w\|_{L^2(\Omega)}^2 \leq M^2 h^{-4} \|(-h^2\Delta-1)w\|_{L^2(\Omega)}^2 +  m^2 \|\partial_\nu w\|_{L^2(\Gamma_0)}^2,
\end{equation}
Finally, we multiply \cref{eq:res-omega-h} by $h^2$, recall from \cref{eq:sc-norm} the definition of the semiclassical norm $\|\cdot\|_{H^{1}_{h}(\Omega)}$ and readily get \cref{eq:hautus-sc} with a slight modification of the constants $m, M > 0$.
% \[\|w\|_{H^{1}_{h}(\Omega)}^2 \;\le\; M_0^2 h^{-2}\,\|(-h^2\Delta-1)w\|_{L^2(\Omega)}^2 +  m_0^2 h^{2}\,\|\partial_\nu w\|_{L^2(\Gamma_0)}^2,\]
% from where we deduce \cref{eq:hautus-sc} immediately.
\end{proof}

\begin{rem}
% The finite--order contact assumption plays a crucial role in \Cref{def:gcc}, 
% as it guarantees that every bicharacteristic admits a unique continuation 
% as a compressed generalized bicharacteristic.  
If one allows infinite-order contact with the boundary, which can occur if $\Gamma$ is merely $\C^\infty$, 
uniqueness of compressed generalized bicharacteristics is lost and the propagation of singularities becomes more delicate. In this case, \cite[Theorem~3.4]{BarLeb92} provides exact observability 
but under stronger geometric conditions, for which, consequently,
our \cref{th:intro} remains valid.
\end{rem}

%%%
\bibliography{HBC.bib}
\bibliographystyle{alpha}

\end{document}